\documentclass[a4paper,10pt,psamsfonts]{amsart} 

\usepackage[utf8]{inputenc} 

\usepackage{amsmath}
\usepackage{amstext}
\usepackage{amsfonts}
\usepackage{amsthm}
\usepackage{amssymb}

\usepackage{dsfont}            
\usepackage{mathrsfs}          
\usepackage{tensor}            

\usepackage{enumerate}

\usepackage{hyperref}
\hypersetup{linktocpage}
\usepackage[toc,page]{appendix}
\usepackage{comment}

\usepackage{color}

\newcommand{\felix}[1]{\marginpar{\footnotesize{{\color{white}{#1}}}}}

\newcommand{\X}{E}

\newcommand{\R}{{\mathds{R}}}
\newcommand{\E}{{\mathds{E}}}
\newcommand{\N}{{\mathds{N}}}

\renewcommand{\P}{{\mathds{P}}} 
\newcommand{\dom}{{\mathrm{dom}}}

\newcommand{\LB}{{\mathcal{L}}}

\newcommand{\diffin}[1]{\mathrm{d}#1}

\newcommand{\id}{\mathrm{id}}

\newcommand{\bD}{{\mathds D}}
\newcommand{\bE}{{\mathds E}}

\newcommand{\bN}{{\mathds N}}

\newcommand{\bP}{{\P}}

\newcommand{\bR}{{\mathds R}}

\newcommand{\cB}{\ensuremath{\mathcal B}}
\newcommand{\cC}{\ensuremath{\mathcal C}}
\newcommand{\cD}{\ensuremath{\mathcal D}}		
\newcommand{\cE}{\ensuremath{\mathcal E}}		
\newcommand{\cF}{\ensuremath{\mathcal F}}

\newcommand{\cL}{\ensuremath{\mathcal L}}
\newcommand{\cM}{\ensuremath{\mathcal M}}
\newcommand{\cN}{\ensuremath{\mathcal N}}
\newcommand{\cO}{\ensuremath{\mathcal O}}
\newcommand{\cP}{\ensuremath{\mathcal P}}

\newcommand{\cS}{\ensuremath{\mathcal S}}

\newcommand{\pred}{\ensuremath{\operatorname{pr}}}
\newcommand{\dl}{\mathrm{d}}
\newcommand{\triple}{{\vert\kern-0.25ex\vert\kern-0.25ex\vert}}
\newcommand{\meas}{\ensuremath{\zeta}}

\newcommand{\NN}{\mathbf N}

\newcommand{\one}{\mathds 1}

\newcommand{\m}{\mu}

\renewcommand{\geq}{\ensuremath{\geqslant}}
\renewcommand{\leq}{\ensuremath{\leqslant}}

\theoremstyle{plain}

\newtheorem{definition}{Definition}[section]
\newtheorem{theorem}[definition]{Theorem}
\newtheorem{lemma}[definition]{Lemma}
\newtheorem{corollary}[definition]{Corollary}
\newtheorem{prop}[definition]{Proposition}
\newtheorem{assumption}[definition]{Assumption}
\newtheorem{setting}[definition]{Setting}
\newtheorem{notation}[definition]{Notation}

\theoremstyle{definition}
\newtheorem{remark}[definition]{Remark}

\newtheorem{example}[definition]{Example}

\begin{document}

\title[Poisson Malliavin calculus]
{Poisson Malliavin calculus in Hilbert space\\ 
with an application to SPDE}

\author[A.~Andersson]{Adam Andersson}
\address{Adam Andersson\\
Syntronic Software Innovations\\
Lindholmspiren 3B\\
SE-417 56 Gothenburg\\
Sweden}
\email{adan@syntronic.com}



\author[F.~Lindner]{Felix Lindner}
\address{Felix Lindner\\
Fachbereich Mathematik\\
Technische Universit\"at Kaiserslautern\\
Postfach 3049, 67653 Kaiserslautern, Germany}
\email{lindner@mathematik.uni-kl.de}

\begin{abstract}
In this paper we introduce a Hilbert space-valued Malliavin calculus for Poisson random measures. 
It is solely based on elementary principles from the theory of point processes and basic moment estimates, and thus allows for a simple treatment of the Malliavin operators.
The main part of the theory is developed for general Poisson random measures, defined on a $\sigma$-finite measure space, with minimal conditions.
The theory is shown to apply to a space-time setting, suitable for studying stochastic partial differential equations. 
As an application, we analyze the weak order of convergence of space-time approximations for a class of linear equations with $\alpha$-stable noise, $\alpha\in(1,2)$.
For a suitable class of test functions, the weak order of convergence is found to be $\alpha$ times the strong order.
\end{abstract}


\maketitle
\tableofcontents

\section{Introduction}

Lévy-driven stochastic partial differential equations (SPDE, for short) have\linebreak drawn much attention in the literature in the last years. 
Such equations are often treated within a Hilbert space framework, see, e.g., the monograph \cite{PesZab2007} and the references therein as well as the more recent articles 
\cite{BessaihHausenblasRazafimandimby2015, BrzGolImkPesPriZab2010, BrzezniakHausenblasZhu2013, Filipovic14, Liu12, MarinelliPrevotRoeckner10, PavlyukevichRiedle15, PesZab13, PriolaZabczyk2011},
to mention but a few.
In this paper we develop a Hilbert space-valued 
Poisson Malliavin calculus which is suitable for treating Hilbert space-valued stochastic evolution equations with 
purely non-Gaussian 
Lévy noise. 
We particularly provide a series of results which can be used to analyze the Malliavin regularity of the solution processes of such equations.
The Malliavin derivative is defined non-locally as a difference operator, similarly to \cite{LastPenrose2011,Picard1996a} in the real-valued case. 
Our main motivation is the analysis of weak errors of numerical approximations of Lévy-driven SPDE. These errors, which are relevant for the analysis of Monte Carlo methods, 
are known to be closely connected to the Malliavin regularity of the solution and its approximations, see, e.g., \cite{AnderssonKovacsLarsson,AnderssonKruseLarsson,KohatsuHiga2,kruse2013}  for corresponding results on SDE and SPDE with Gaussian noise.
While our Poisson Malliavin calculus is general enough to be applicable to a large class of equations with 
additive or multiplicative 
Lévy noise, we intend it to be as simple as possible and therefore avoid technicalities which are not needed 
for our purpose,
such as the use of chaos expansions and associated Fock space structures or the use of closure arguments for the definition of the Malliavin derivative.
Our approach is solely based on elementary principles from the theory of point processes and basic moment estimates, and thus allows for a simple treatment of the Malliavin operators.
For greater clarity, a large part of the theory is developed in a general setting without an underlying space-time structure.
As a first application of our theory, we analyze the weak order of convergence of space-time discretizations for a class of linear SPDE driven by $\alpha$-stable noise, $\alpha\in(1,2)$. 
In the accompanying paper \cite{AnderssonLindner2017b} we also treat semi-linear equations.

There exists an extensive literature and various different approaches to Malliavin calculus for Poisson random measures and jump processes, see, e.g., the monographs \cite{bichteler1987, DiNunno2009, ishikawa2016, privault2009} and the references therein.  
There are roughly two main lines of research: 
In the first line the Malliavin derivative is defined as a local operator acting on the size or the instant of the jumps, cf.~\cite{bally2007, bichteler1987, Bouleau2011, Carlen1990, LeonUtzetVives2014}. 
In the second line it is defined as an annihilation operator based on chaos expansions in terms multiple Poisson stochastic integrals, leading to a non-local difference operator. 
This second approach has originally been developed in
\cite{dermoune1988, Ito1988, Nualart1990, Nualart1995, Picard1996a}
and has later been extended in various directions, see, e.g.,~\cite{Applebaum2009, Ishikawa2006, LastPenrose2011b, LastPenrose2011, Sole2007, Vives2013} and the references therein. 
We follow the second approach in the present article but avoid the use of chaos expansions and Fock space structures.
Note that in the mentioned literature 
concerning the second line of research
only real-valued or finite-dimensional random variables and stochastic processes are considered, and in most cases the Poisson random measure is assumed to be defined on a locally-compact space. 
An exception to the latter restriction are the works by Picard \cite{Picard1996a,Picard1996b}, Last and Penrose \cite{LastPenrose2011b,LastPenrose2011} and Last \cite{Last2014}. These articles are closely related to our work and serve as our main reference, but have a different scope and purpose.
%
To the best of our knowledge, so far the only publications  within the second line of research dealing with a Poisson Malliavin calculus 
for Hilbert or Banach space-valued random variables and processes are
\cite{Dirksen2013}, where a Malliavin framework for Banach space-valued Poisson stochastic integrals is developed, and \cite{Barth2016}, where the framework from \cite{Dirksen2013} is applied to a class of linear SPDE with square-integrable additive Lévy noise in a Hilbert space setting.

The approach and the results in the present article differ considerably from those in \cite{Dirksen2013,Barth2016} in several regards: For instance, in \cite{Dirksen2013} the Malliavin derivative 
is first defined on a core of cylindrical random variables and then extended to larger classes of $L^p$-integrable random variables via a closure argument in a second step.
While this procedure has the advantage of being formally analogous to the construction in the Gaussian case \cite{Maas2010,Nualart2006}, it is not necessary in the Poisson case and comes with the drawback that it hides natural features of the Poisson Malliavin calculus and complicates several proofs. 
In contrast, we introduce the Malliavin derivative from the beginning as a difference operator acting on $L^0$ 
spaces of (equivalence classes of) random variables without prescribed integrability properties, and use Mecke's formula from the theory of point processes \cite{LastPenrose2016,Mecke1967} to ensure that it is well-defined. 
The realizations of the Malliavin derivative in $L^p$ spaces are then merely restrictions of this operator to smaller domains, and the closedness of such a restriction follows from an elementary continuity property of the difference operator in $L^0$ spaces, cf.~Section~\ref{sec:DiffOp}. 
Similarly, we introduce the Kabanov-Skorohod integral as an $L^1$-integral in a pathwise sense, and consider abstract extensions thereof to $L^p$ spaces with $p>1$ only where it is needed, cf.~Section~\ref{subsec:Kabanov-Skorohod}.
These aspects lead to natural simplifications and/or generalizations of several arguments; compare for instance the assertions and proofs of Propositions 5.4, 5.5 in \cite{Dirksen2013} and Proposition 2.6 in \cite{Barth2016} with those of Propositions~\ref{lem:product}, \ref{lem:duality} and Proposition \ref{lem:chain} below.

Apart from that, we derive numerous results which are not included in \cite{Dirksen2013,Barth2016}, but which 
are needed for the analysis of the Malliavin regularity of larger classes of Lévy-driven stochastic evolution equations. 
Among those results are the local duality formulas in Lemma~\ref{lemma:duality2} and Proposition~\ref{prop:duality_space_time_local}, which do not rely on global but only on local integrability properties for the Malliavin derivative and are thus well-suited for handling the typical integrability properties of Lévy processes without finite second moments,
such as the $\alpha$-stable processes considered in Section~\ref{sec:alpha_stable}.
Another example are the general commutation relations between the Malliavin derivative and the Skorohod-Kabanov integral in Propositions~\ref{lem:commutator1}, \ref{prop:commutator_relation_L^2}, \ref{prop:comm_space-time1} and \ref{prop:comm_space-time2}, which are essential for the analysis of the Malliavin regularity of stochastic integral processes. 
These relations also allow for the treatment of equations with multiplicative noise, cf.~the accompanying paper \cite{AnderssonLindner2017b}.
Finally, we note that in \cite{Barth2016} Malliavin calculus methods are used in the spirit of \cite{kruse2013} to derive a weak convergence result for spatial approximations of linear equations with square-integrable Lévy noise, which is very similar to an earlier result from \cite{KovLinSch2015}. In Section~\ref{sec:alpha_stable} of the present article we allow instead for $\alpha$-stable driving processes as an important class of non-square-integrable Lévy processes. Moreover, we consider discretizations in space and time as well as a class of path-dependent test functions. 
Our corresponding result in Theorem~\ref{thm:weak_error_alpha} appears to be the first result in the literature giving an explicit weak convergence rate for approximations of SPDE with non-square-integrable Lévy noise.

To complete the picture, let us also mention that Malliavin calculus methods have been applied to parabolic SPDE with jumps in \cite{Fournier2000} in order to show that the one-dimensional marginal distributions of the solution are absolutely continuous w.r.t.~Lebesgue measure.
Here the Malliavin derivative is defined as a local operator acting on the size of the jumps, and the considered equation is treated in the spirit of Walsh's approach to SPDE \cite{walsh1986}.

\addtocontents{toc}{\protect\setcounter{tocdepth}{1}}
\subsection*{Overview of the article} 
We start by collecting some preliminaries in Section~\ref{sec:Preliminaries}, where we introduce general notation and conventions in Subsection~\ref{sec:Notation_and_conventions} and describe the setting considered throughout the article in Subsection~\ref{sec:PRM_and_Mecke}. It is basically given by a Poisson random measure $N$ on a measurable space $(E,\cE)$ with a $\sigma$-finite intensity measure $\mu$
and a separable real Hilbert space $H$, in which the considered random variables and stochastic processes take their values. 
The $\sigma$-algebra $\cF$ of the underlying probability space $(\Omega,\cF,\bP)$ is assumed to be generated by $N$ (and completed w.r.t.~$\bP$), so that $N$ is the only source of randomness.
We also recall Mecke's formula, which plays a crucial role throughout the article.

In Section~\ref{sec:Malliavin} we develop an $H$-valued Poisson Malliavin calculus for Poisson random measures defined on the general $\sigma$-finite measure space $(E,\cE,\mu)$. In Subsection~\ref{sec:DiffOp} we introduce the Malliavin derivative
$$
D\colon L^0(\Omega;H)\to L^0(\Omega\times E;H)
$$
as a difference operator acting on (equivalence classes of) $H$-valued random variables. 
We derive several properties of the difference operator and introduce first order Sobolev-type spaces $\bD^{1,p}(H)\subset L^p(\Omega;H)$, $p>1$, which take integrability properties into account.
In Subsection~\ref{subsec:Kabanov-Skorohod} we use Mecke's formula to define a pathwise Kabanov-Skorohod integral 
$$
\delta\colon L^1(\Omega\times E;H)\to L^1(\Omega;H)
$$
in such a way that its adjoint coincides with the restricted difference operator $D|_{L^\infty(\Omega;H)}\colon L^\infty(\Omega;H)\to L^\infty(\Omega\times E;H)$. 
The pathwise integral is then extended to $L^p$ spaces, $p>1$,  by introducing abstract versions 
$$
\delta^{(p)}\colon\dom(\delta^{(p)})\subset L^p(\Omega\times E;H)\to L^p(\Omega;H)$$ 
as the adjoint operators of 
the restricted difference operators 
$D|_{\bD^{1,p'}(H)}\colon\bD^{1,p'}(H)\subset L^{p'}(\Omega;H)\to L^{p'}(\Omega\times E;H)$, $p'=\frac{p}{p-1}$. 
The integrals $\delta$, $\delta^{(p)}$ coincide on the intersection of their domains, which allows us to omit the integrability index and to set $\delta(\Phi):=\delta^{(p)}(\Phi)$, $\Phi\in\dom(\delta^{(p)})$.
Again, several properties of the operators are derived. 
Subsection \ref{subsec:higher_order} is concerned with the commutation relation
$$
D_x\delta(\Phi)=\delta(D_x\Phi)+\Phi(x),\quad x\in E,
$$
where $\Phi\colon\Omega\times E\to H$ is a suitable integrand.
We prove $L^1$ and $L^2$ versions of this relation 
as well as an $L^2$ isometry for the Kabanov-Skorohod integral. 
\felix{later $L^p$ versions(?)}
In Subsection~\ref{subsec:higher_order} we extend some of the previous results to higher order difference operators, 
higher order Sobolev spaces 
and multiple Kabanov-Skorohod integrals. 

In Section~\ref{sec:space_time} we assume a space-time structure of the underlying measure space and consider the special case 
$$
(E,\cE,\mu)=\big([0,T]\times U,\,\cB([0,T]\times U),\,\lambda\otimes\nu\big),
$$
where $T\in(0,\infty)$, $U$ is a separable real Hilbert space, $\lambda$ denotes Lebesgue measure, and $\nu$ is a $\sigma$-finite measure on the Borel-$\sigma$-algebra $\cB(U)$. We complement the general theory from Section~\ref{sec:Malliavin} by a series of results which are specifically adapted to this case and particularly relevant for the analysis of Lévy-driven SPDE.
Several auxiliary results for the Malliavin operators $D$ and $\delta$ in the space-time setting are presented in Subsection~\ref{subsec:Malliavin_operators_space-time}. As the Kabanov-Skorohod integral is an extension of an Itô-type integral w.r.t.\ the compensated Poisson random measure $\tilde N$, we can  exploit continuity properties of the latter integral in order to obtain important partial improvements of some of the general results from Section~\ref{sec:Malliavin}. This is done in Subsection~\ref{subsec:duality_comm_space-time}, where we derive an improved local duality formula 
as well as improved $L^p$ versions, $p\in[1,2]$, of the commutation relation between $D$ and $\delta$.
In Subsection ~\ref{subsec:PRMandLP} we show how Hilbert space-valued, purely non-Gaussian Lévy processes can be embedded into our framework.

A first application of our theory is presented in Section~\ref{sec:alpha_stable}, where we analyze the weak order of convergence of finite element discretizations
of linear SPDE of the form $\dl X(t)+AX(t)\,\dl t=\dl L(t)$. 
Here $-A$ is the generator of an analytic semigroup of bounded linear operators on $H$, 
and $L=(L(t))_{t\in[0,T]}$ is an infinite-dimensional Lévy process of $\alpha$-stable type, $\alpha\in (1,2)$.
In Subsection~\ref{sec:setting_alpha_stable} we describe the setting in detail, give concrete examples, and determine the spatio-temporal regularity of the solution process $X=(X(t))_{t\in[0,T]}$. 
For comparison's sake we analyze the strong order of convergence in Subsection~\ref{sec:strong_convergence}.
We obtain the estimate 
\begin{align*}
\sup_{t\in[0,T]}\big\|\tilde X_{h,k}(t)-X(t)\big\|_{L^{\alpha_-}(\Omega;H)}\leq C\big(h^{\beta_-}+k^{\frac{\beta_-}2}\big),
\end{align*}
valid for all $\alpha_-\in[1,\alpha)$ and $\beta_-\in[0,\beta)$, where the parameter $\beta\in(\frac2\alpha-1,\frac2\alpha]$ describes the spatial regularity of the solution $X$, 
$\tilde X_{h,k}=(\tilde X_{h,k}(t))_{t\in[0,T]}$ is the time-interpolated discrete solution, 
$h,k\in (0,1)$ are the space and time discretization parameters, 
and $C\in(0,\infty)$ depends on $\alpha_-$, $\beta_-$, but not on $h,k$. The weak convergence result is shown in Subsection~\ref{sec:weak_convergence}, with the help of results from Sections~\ref{sec:Malliavin} and \ref{sec:space_time}. For a suitable class of real-valued path-dependent test functions $f$ we obtain 
\begin{align*}
\big|\bE\;\!f(\tilde X_{h,k})-\bE\;\!f(X)\big|\leq C\big(h^{\alpha\beta_-}+k^{\frac{\alpha\beta_-}2}\big).
\end{align*}
The weak order of convergence is thus $\alpha$ times the strong order of convergence. This is a natural complement to similar results for equations with Gaussian noise, where the weak order is typically twice the strong order, see, e.g.,~\cite{AnderssonKovacsLarsson,AnderssonKruseLarsson,kruse2013}. It also complements the corresponding results for equations  with square-intergrable Lévy noise from \cite{KovLinSch2015,LindnerSchilling,Barth2016}.

\addtocontents{toc}{\protect\setcounter{tocdepth}{2}}

\section{Preliminaries}
\label{sec:Preliminaries}

\subsection{Notation and conventions}
\label{sec:Notation_and_conventions}
The following notation and conventions are used throughout this article.

If $(S,\cS,m)$ is a $\sigma$-finite measure space and $(X,\|\cdot\|_X)$ a Banach space,
\felix{see notes 14.11.'16} 
we denote by $L^0(S;X):=L^0(S,\cS,m;X)$ the space of (equivalence classes of) strongly $\cS$-measurable functions $f\colon S\to X$. As usual, we identify functions which coincide $m$-almost everywhere. The space $L^0(S;X)$ is endowed with the topology of local convergence in measure; it is metrizable via the metric $d(f,g)=\sum_{k=1}^\infty\frac{2^{-k}}{1+m(A_k)}\int_{A_k}(1\wedge\|f(s)-g(s)\|_X)\,m(\dl s)$, $f,g\in L^0(S;X)$, where $A_1,A_2,\ldots\in\cS$ with $A_k\nearrow S$ and $m(A_k)<\infty$ for all $k\in\N$. 
For $p\in[1,\infty]$, we denote by $L^p(S;X):=L^p(S,\cS,m;X)$ the subspace of $L^0(S;X)$ consisting of all (equivalence classes of) strongly $\cS$-measurable mappings $f\colon S\to X$ such that $\|f\|_{L^p(S;X)}:=\big(\int_S\|f(s)\|_X^p\,m(\dl s)\big)^{1/p}<\infty$ if $p\in[1,\infty)$ and  $\|f\|_{L^\infty(S;X)}:=\operatorname{ess\,sup}_{s\in S}\|f(s)\|_X<\infty$ if $p=\infty$. Note that $L^p(S;X)$ is continuously embedded into $L^0(S;X)$. 
If $(H,\langle\cdot,\cdot\rangle)$ is a Hilbert-space and $f\in L^p(S;H)$, $g\in L^{p'}(S;H)$, where $p\in[1,\infty]$ and $p':=\frac{p}{p-1}\in[1,\infty]$ are conjugate exponents, we also write $\tensor[_{L^p(S;H)}]{\langle f,g\rangle}{_{L^{p'}(S;H)}}$ for $\int_S\langle f(s),g(s)\rangle\, m(\dl s)$.

Given two $\sigma$-finite measure spaces $(S_1,\cS_1,m_1)$, $(S_2,\cS_2,m_2)$ and a Banach space  $(X,\|\cdot\|_X)$, we set $L^p(S_1\times S_2;X):=L^p\big(S_1\times S_2,\cS_1\otimes\cS_2,m_1\otimes m_2;X\big)$, $p\in\{0\}\cup[1,\infty]$.
For $p\in[1,\infty)$ we identify the spaces $L^p(S_1\times S_2;X)$ and $L^p(S_1;L^p(S_2;X))$ via the canonical isometric isomorphism which is determined by the mapping $\one_{A_1\times A_2} \otimes x\mapsto\one_{A_1}\otimes(\one_{A_2}\otimes x)$ and linearity. 
Here, $A_1\in\cS_1$ and $A_2\in\cS_2$ have finite measure, $x\in X$, $\one_{A_1\times A_2} \otimes x\in L^p(S_1\times S_2;X)$ is defined by $(\one_{A_1\times A_1} \otimes x)(s_1,s_2):=\one_{A_1\times A_2}(s_1,s_2)\,x$, and $\one_{A_1}\otimes(\one_{A_2}\otimes x)\in L^p(S_1;L^p(S_2;X))$ is defined by $(\one_{A_1}\otimes(\one_{A_2}\otimes x))(s_1):=\one_{A_1}(s_1)(\one_{A_2}(\cdot)\, x)$, $(s_1,s_2)\in S_1\times S_2$. In an analogous way, the space $L^0(S_1;L^p(S_2;X))$ is continuously embedded into $L^0(S_1\times S_2;X)$, and we usually identify the corresponding elements in both spaces without explicitly indicating it. 
Observe that, if $J$ denotes the embedding
of $L^0(S_1;L^p(S_2;X))$ into $L^0(S_1\times S_2;X)$, 
then the range of $J$ is given by all $f\in L^0(S_1\times S_2;X)$ such that for any representative of $f$, denoted again by $f$, we have $f(s_1,\cdot)\in L^p(S_2;X)$ for $m_1$-almost all $s_1\in S_1$. \felix{see notes 15.1.'17}
Moreover, for any such $f$ and for any representative of $\tilde f:=J^{-1}(f)\in L^0(S_1;L^p(S_2;X))$, denoted again by $\tilde f$, we have $\|f(s_1,\cdot)-\tilde f(s_1)\|_{L^p(S_2;X)}=0$ for $m_1$-almost all $s_1\in S_1$.

Given separable Hilbert spaces $U$ and $H$, we write $\LB(U,H)$ and $\LB_2(U,H)$ for the spaces of linear and bounded operators and Hilbert-Schmidt operators from $U$ to $H$, endowed with the operator norm $\|\cdot\|_{\LB(U,H)}$ and the Hilbert Schmidt norm $\|\cdot\|_{\LB_2(U,H)}$, respectively. $B_U:=\{x\in U:\|x\|_U\leq1\}$ is the closed unit ball in $U$ and $B^c_U:=U\setminus B_U$ its complement. If  $\varphi\colon H\to\bR$ is a Fréchet-differentiable function, its derivative $\varphi'(x)$ at a point $x\in H$ is considered as an element of $H$ via the Riesz isomorphism.
For $\delta\in (0,1)$ we denote by $\cC^{1,\delta}(H,\bR)$ be the space of Fréchet-differentiable functions $\varphi\colon H\to\bR$ such that $\sup_{x,y\in H}\frac{\|\varphi'(x)-\varphi'(y)\|_H}{\|x-y\|^\delta_H}<\infty$.

Given a probability space $(\Omega,\cF,\bP)$, a measure space $(S,\cS,m)$, a Banach space $X$ and a mapping $f\colon\Omega\times S\to X,\;(\omega,s)\mapsto f(\omega,s)$, we often omit the explicit notation of the argument $\omega\in\Omega$ and write $f(s)$ instead of
$f(\omega,s)$ or 
$f(\cdot,s)$, depending on the context. We sometimes also write $(f(s))$ for the mapping $f$ in order to indicate the dependence on the variable $s\in S$.

Finally, $C\in(0,\infty)$ denotes a finite constant which may change its value from line to line. 


\subsection{Poisson random measures and Mecke's formula}
\label{sec:PRM_and_Mecke}

Here we present our general setting and recall Mecke's formula for Poisson random measures. While most of the material in this subsection is fundamental in the theory of point processes and stochastic geometry, it is less standard in the analysis of SPDE. For references we refer to the monographs \cite{DaleyVereJones2008,LastPenrose2016,Schneider2008}.


We begin with formulating our main framework.
\begin{setting}\label{setting}
The following setting is considered throughout this article:
\begin{itemize}
\item $(\Omega,\cF,\P)$ is a complete probablility space. The $\sigma$-algebra $\cF$ coincides with the $\P$-completion of the $\sigma$-algebra $\sigma(N)$ generated by the Poisson random measure $N$ introduced below.
\item $(\X,\cE,\m)$ is a $\sigma$-finite measure space such that 
the diagonal $\{(x,y)\in E\times E:x=y\}$ is contained in the product $\sigma$-algebra $\cE\otimes\cE$.
\felix{nochmal checken, ob ich singletons $\{x\}$ wirklich nur in alter Def. von $\meas\setminus\delta_x$ gebraucht habe!}
We set $\cE_0:=\{B\in\cE:\m(B)<\infty\}$.
\item $\NN=\NN(\X)$ denotes the space of all $\sigma$-finite $\N_0\cup\{+\infty\}$-valued measures on $(E,\cE)$. It is endowed with the $\sigma$-algebra $\mathcal{N}=\mathcal N(\X)$ generated by the mappings $\NN\ni\eta\mapsto\eta(B)\in\N_0\cup\{+\infty\}$, $B\in\cE$.
\item $N\colon\Omega\to\NN$ is a Poisson random measure (Poisson point process) on $\X$ with intensity measure $\m$, allowing for the representation \eqref{eq:representationN} below
(cf.~also Remark~\ref{rem:PRM}).
\item $\tilde N:=N-\m$ is the compensated Poisson random measure associated to $N$, i.e., $\tilde N(B)=N(B)-\m(B)$ for all $B\in\cE_0$.
\item $H$ is a separable real Hilbert space with inner product $\langle\cdot,\cdot\rangle$ and norm $\|\cdot\|$. 
\end{itemize}
\end{setting}

We assume that the Poisson random measure $N$ is constructed in the standard way  as a random sum of Dirac measures, see e.g.~\cite[Theorem~6.4]{PesZab2007}. 
As a consequence, we have the representation
\felix{see notes 12.11.'16}
\begin{align}\label{eq:representationN}
N=\sum_{n=1}^{N(\X)}\delta_{X_n},
\end{align}
where $X_1,X_2,\ldots$ are suitably chosen $\X$-valued random variables and  $\delta_x$ denotes the Dirac measure at $x\in\X$.
\felix{explain notation $N(\omega,B)$, $N(B)$, ...}

\begin{remark}\label{rem:PRM}
Recall that $N$ being a Poisson random measure with reference measure $\m$ means that $N\colon\Omega\to\NN$ is $\cF$-$\cN$-measurable, that $N(B)$ is Poisson distributed with parameter $\m(B)$ if $B\in\cE_0$ and $N(B)=\infty$ $\P$-a.s.\ if $B\in\cE\setminus\cE_0$, and that $N(B_1),\dots,N(B_n)$ are independent for disjoint $B_1,\dots,B_n\in\cE$, $n\in\N$. 
In particular, for $B\in\cE_0$ 
all moments of $N(B)$ are finite and $\E N(B)=\E\big((N(B)-\m(B))^2\big)=\m(B)$.
\end{remark}

\color{black}

For any measurable function $F:\Omega\times\X\to[0,\infty]$, 
\felix{see notes 12.11.'16}
the integral $\int_\X F(\omega,x)N(\omega,\diffin x)$ is well-defined for all $\omega\in\Omega$, possibly infinite, and measurable as a function of $\omega$.
The latter can be seen, e.g., by using the representation \eqref{eq:representationN}. We will frequently work with functions $F$ whose definitions involve the mappings
\begin{align}
\Omega\times\X\ni(\omega,x)&\mapsto N(\omega)+\delta_x\in\NN,\label{eq:N+delta}\\
\Omega\times\X\ni(\omega,x)&\mapsto N(\omega)\setminus\delta_x\in\NN.\label{eq:N-delta}
\end{align}
Here and below we denote for $\eta\in\NN$ and $x\in\X$ by $\eta\setminus\delta_x\in\NN$  the measure defined by 
\begin{align*}
\eta\setminus\delta_x:=
\begin{cases}
\eta-\delta_x,&\text{if }\eta(B)\geq\delta_x(B)\text{ for all }B\in\cE\\
\eta,&\text{else.}
\end{cases}
\end{align*}

\begin{remark}[Measurability]\label{rem:measurability}
\felix{see notes 27.-28.9.'15, 12.11.'16}
The mappings \eqref{eq:N+delta}, \eqref{eq:N-delta} introduced above are $(\mathcal F\otimes\cE)$-$\mathcal N$-measurable. Indeed, considering the mapping \eqref{eq:N+delta}, it is sufficient to check the $(\mathcal N\otimes\cE)$-$\mathcal N$-measurabiliy of $\NN\times\X\ni(\eta,x)\mapsto \eta+\delta_x\in\NN$, and here it is enough to 
note that $\NN\times\X\ni(\eta,x)\mapsto (\eta+\delta_x)(B)\in\N_0\cup\{\infty\}$ is measurable for all $B\in\cE$.
Concerning the mapping \eqref{eq:N-delta}, one may use the representation \eqref{eq:representationN} to observe that
$
N(\omega)\setminus\delta_x=\sum_{n=1}^{N(\omega,\X)}\one_{\{X_n\neq x\}}(\omega)\,\delta_{X_n(\omega)}+\big(\sum_{n=1}^{N(\omega,\X)}\one_{\{X_n=x\}}(\omega)-1\big)^+\delta_x.
$
As the diagonal $\{(x,y)\in E\times E:x=y\}$ is by assumption contained in $\cE\otimes\cE$, it follows that $N(\omega)\setminus\delta_x$ is measurable as an $\NN$-valued function of $(\omega,x)\in\Omega\times\X$. 
\end{remark}

A fundamental result in this context which will be used repeatedly in this article is the following formula by Mecke, see \cite{LastPenrose2016,Mecke1967}.
\begin{prop}[Mecke's formula]
For measurable functions $f\colon\NN\times\X\to[0,\infty]$ we have
\begin{align}\label{eq:Mecke1}
  \E
  \Big[
    \int_\X
      f(N,x)
    \,N(\diffin{x})
  \Big]
  =
  \E
  \Big[
    \int_\X
      f(N+\delta_x,x)
    \,\m(\diffin{x})
  \Big]
\end{align}
and, equivalently, 
\begin{align}\label{eq:Mecke2}
  \E
  \Big[
    \int_\X
      f(N\setminus\delta_x,x)
    \,N(\diffin{x})
  \Big]
  =
  \E
  \Big[
    \int_\X
      f(N,x)
    \,\m(\diffin{x})
  \Big].
\end{align}
Conversely, any point process $N\colon\Omega\to\NN$ satisfying \eqref{eq:Mecke1} or \eqref{eq:Mecke2} for all measurable $f\colon\NN\times\X\to[0,\infty]$ is a Poisson point process with reference measure $\m$.
\end{prop}


Besides others, Mecke's formula ensures that the operators $\varepsilon^+$, $\varepsilon^-$ introduced next are well-defined as mappings acting on equivalence classes of random variables or random fields. 
Given a random variable $F\colon\Omega\to H$, the factorization lemma from measure theory yields the existence of a measurable function $f\colon\NN\to H$ such that $F=f(N)$ $\P$-a.s. Note that this equality does not necessarily hold for all $\omega\in\Omega$, since we have defined the underlying $\sigma$-algebra $\cF$ as the $\P$-completion of $\sigma(N)$ and not as $\sigma(N)$ itself.
We call $f$ a representative of $F$.
In this situation, we define for $x\in\X$,
\begin{align}\label{eq:epsilon+}
\varepsilon^+_x
F:= f(N+\delta_x).
\end{align}
Here and in the sequel, we usually omit the explicit notation of the argument $\omega\in\Omega$ for simplicity.

\begin{lemma}[Well-definedness of $\varepsilon^+$]
\label{lem:epsilon+welldefined}
For $F\in L^0(\Omega;H)$, the definition \eqref{eq:epsilon+} of $\varepsilon^+_xF$ is $\P \otimes\m$-almost everywhere independent of the choice of the representative $f$. In particular, the operator 
\[F\mapsto \big(\varepsilon^+_x F\big)\] 
is well-defined as a mapping from $L^0(\Omega;H)$ to $L^0(\Omega\times \X;H)$.
\end{lemma}

\begin{proof}
The assertion follows from Mecke's formula \eqref{eq:Mecke1}. If $f,\,g\colon\NN\to H$ are measurable functions such that $f(N)=g(N)$ $\P$-almost surely, then
\begin{align*}
0&=\E\int_\X\big\|(f-g)(N)\big\|\,N(\diffin x)=\E\int_\X\big\|(f-g)(N+\delta_x)\big\|\,\m(\diffin x),
\end{align*}
where the integrand of the integral w.r.t.\ $N(\dl x)$ is constant in $x$ and equals zero $\P$-a.s.
We conclude that $f(N+\delta_x)=g(N+\delta_x)$ $\P\otimes\m$-almost everywhere.
\end{proof}

In analogy to \eqref{eq:epsilon+}, for a random variable $F\colon\Omega\to H$ with representative $f\colon\NN\to H$ and $x\in \X$ we may set 
$
\varepsilon^-_xF:=f(N\setminus\delta_x).
$
Mecke's formula \eqref{eq:Mecke2} implies that this definition is $\P\otimes N$-almost everywhere independent of the choice of the representative $f$. 
Here, $\P\otimes N$ denotes the product measure of $\P$ and $N$, the latter being considered as a transition kernel from $(\Omega,\mathcal F)$ to $(\X,\cE)$; it is given by
\felix{see notes 12.11.'16}
\begin{align}\label{eq:PotimesN}
(\P\otimes N)(B):=\int_\Omega\int_\X \one_B(\omega,x)N(\omega,\diffin x)\P(\diffin \omega),\quad B\in\mathcal F\otimes \cE.
\end{align}
In this context we are, however,  mainly interested in the case where $F$ also depends on $x\in\X$.  Given a measurable mapping $F\colon \Omega\times\X\to H$ we set for $x\in\X$
\begin{align}\label{eq:epsilon-}
\varepsilon^-_x F(x):=f(N\setminus\delta_x,x),
\end{align}
where $f\colon\NN\times\X\to H$ is measurable such that  $F(x)=f(N,x)$ $\P\otimes \m$-almost everywhere, called again a representative of $F$.
The proof of the following lemma concerning the well-definedness of $\varepsilon^-$ as an operator acting on equivalence classes of random fields  is similar to that of Lemma~\ref{lem:epsilon+welldefined}, using the identity \eqref{eq:Mecke2} instead of \eqref{eq:Mecke1}, and therefore omitted. 

\begin{lemma}[Well-definedness of $\varepsilon^-$]
\label{lem:epsilon-welldefined}
\felix{proof commented out; see notes 11.11.'16}
For $F\in L^0(\Omega\times\X;H)$, the definition \eqref{eq:epsilon-} of $\varepsilon^-_xF(x)$ is $\P\otimes N$-almost everywhere independent of the choice of the representative $f$. In particular, the operator
\[\big(F(x)\big)\mapsto \big(\varepsilon^-_x F(x)\big)\] 
is well-defined as a mapping from $L^0(\Omega\times\X;H)=L^0\big(\Omega\times\X,\P\otimes\m;H\big)$ to $L^0\big(\Omega\times\X,\P\otimes N;H\big)$.
\end{lemma}

\begin{remark}\label{rem:epsilon_for_random_fields}
We will also work with the following extensions of the above defined operators $\varepsilon^\pm$ to random fields.
\felix{see notes 11.11.'16}
Let $( S,\cS,m)$ be a $\sigma$-finite measure space; typically we have $( S,\cS,m)=(\X^n,\cE^{\otimes n},\m^{\otimes n})$ for some $n\in\N$.
Given a measurable mapping $F\colon \Omega\times S\to H$ and $s\in S$, $x\in\X$ we set
$$
\varepsilon^+_x F(s):=f(N+\delta_x,s),
$$
where $f\colon\NN\times S\to H$ is a measurable function such that  $F(s)=f(N,s)$ $\P\otimes m$-almost everywhere. Lemma~\ref{lem:epsilon+welldefined} implies that this definition of $\varepsilon^+_x F(s)$ is $\P\otimes m\otimes\m$-almost everywhere independent of the choice of the representative $f$  and that the operator 
$\big(F(s)\big)\mapsto\big(\varepsilon^+_xF(s)\big)$ 
is well-defined as a mapping from $L^0(\Omega\times S;H)$ to $L^0(\Omega\times S\times E;H)$.
Similarly, given a measurable mapping $F\colon\Omega\times S\times E\to H$ with representative $f\colon\NN\times S\times E\to H$ and $s\in S$, $x\in E$ we set
$$
\varepsilon^-_xF(s,x):=f(N\setminus\delta_x,s,x).
$$
As a consequence of Lemma~\ref{lem:epsilon-welldefined}, this definition of $\varepsilon^-_xF(s,x)$ is $\P\otimes m\otimes N$-almost everywhere independent of the choice of the representative $f$, and the operator 
$\big(F(s,x)\big)\mapsto\big(\varepsilon^-_xF(s,x)\big)$ 
is well-defined as a mapping from $L^0(\Omega\times S\times E;H)=L^0(\Omega\times S\times E,\P\otimes m\otimes \m;H)$ to $L^0(\Omega\times S\times E,\P\otimes m\otimes N;H)$. Here $\bP\otimes m\otimes N$ is understood as the product measure of the measure $\bP\otimes N$ on $(\Omega\times E,\cF\otimes \cE)$, given by \eqref{eq:PotimesN}, and the measure $m$ on $(S,\cS)$.
\end{remark}

\section{Hilbert space-valued Poisson Malliavin calculus}
\label{sec:Malliavin}

In this section we develop a Hilbert space-valued Malliavin calculus for Poisson random measures defined on a $\sigma$-finite measure space.  
Throughout the section we consider the setting described in Subsection~\ref{sec:PRM_and_Mecke}.

\subsection{Difference operator and first order Sobolev spaces}
\label{sec:DiffOp}

In the Gaussian case the Malliavin derivative is a differential operator. In the Poisson case one possible analogue is a finite difference operator, compare \cite{dermoune1988, Ito1988, Nualart1990, Nualart1995, Picard1996a}. The following definition is meaningful due to Lemma~\ref{lem:epsilon+welldefined}.

\begin{definition}[Difference operator]\label{def:diffOp}
Let the operator
\begin{align*}
  &D\colon L^0(\Omega;H)\rightarrow L^0(\Omega\times\X;H),\;F\mapsto DF=(D_xF)
\end{align*}
be defined by 
\[D_xF:=\varepsilon^+_xF-F,\quad x\in E.\]
That is, we have
$D_x F= f(N+\delta_x)-f(N)$,
where $f\colon\NN\to H$ is a measurable function such that $f(N)$ is a representative of (the equivalence class of random variables) $F\in L^0(\Omega;H)$.
\end{definition}

Immediate algebraic consequences of the definition of $D$ are the following analogues of the chain rule and the product rule. 
For the convenience of the reader we present the proofs. 

\begin{prop}[Chain rule]\label{lem:chain}
Let $F\in L^0(\Omega;H)$ and $h$ be a measurable mapping from $H$ to another (real and separable) Hilbert space $V$. Then,  
\begin{align*}
  D h(F)
& =
  h(F+D F)-h(F).
\end{align*}
\end{prop}

\begin{proof}
The assertion follows directly from the definition of the operators $D$ and $\varepsilon^+$ since, for $x\in\X$,
\begin{align*}
  D_x h(F)
& =
  \varepsilon^+_x h(F)-h(F)
  =
  h(\varepsilon^+_x F)-h(F)\\
& =
  h\big(F+(\varepsilon^+_xF -F)\big)-h(F)
  =
  h(F+D_xF)-h(F).\qedhere
\end{align*}
\end{proof}

\begin{prop}[Product rule]\label{lem:product}
For $F,G\in L^0(\Omega;H)$, 
\begin{align*}
  D\langle F,G\rangle=\langle DF,G\rangle+\langle F,DG\rangle+\langle DF,DG\rangle.
\end{align*}
\end{prop}

\begin{proof}
The definition of the operators $D$ and $\varepsilon^+$ impies that, for $x\in\X$,
\begin{align*}
D_x\langle F,G\rangle
&=\varepsilon^+_x\langle F,G\rangle-\langle F,G\rangle
=\big\langle\varepsilon^+_xF,\varepsilon^+_xG\big\rangle-\langle F,G\rangle\\
&=\big\langle\varepsilon^+_xF-F,\varepsilon^+_xG-G\big\rangle+\big\langle\varepsilon^+_xF-F,G\big\rangle+\big\langle F,\varepsilon^+_xG-G\big\rangle\\
&=\langle D_xF,D_xG\rangle+\langle D_xF,G\rangle+\langle F,D_xG\rangle.
\qedhere
\end{align*}
\end{proof}

%

We next aim at restricting the operator $D$, originally defined on the space of all (equivalence classes of) random variables, to Sobolev-type spaces in which integrability is taken into account. 

\begin{definition}
For $p>1$ we define $\bD^{1,p}(H)$ as the space of all random variables $F\in L^p(\Omega;H)$ such that $D F\in L^p(\Omega\times\X;H)$. It is equipped with the norm 
\begin{align*}
  \|F\|_{\bD^{1,p}(H)}
  :=
  \Big(
    \|F\|_{L^p(\Omega;H)}^p
    +
    \|D F\|_{L^p(\Omega\times\X;H)}^p
  \Big)^\frac1p.
\end{align*}
\end{definition}

The closedness of $D$ as an operator from $L^p(\Omega;H)$ to $L^p(\Omega\times\X;H)$ with domain $\bD^{1,p}(H)$, and thus the Banach space property of $\bD^{1,p}(H)$, can be deduced from the following elementary moment estimate. It enables us to control local $L^q$-norms of $DF$ for $L^p$-integrable random variables $F$ with $1\leq q<p$.

\begin{lemma}[Local $L^q$-estimate]\label{lem:local_estimate_D}
Let $p>1$, $F\in L^p(\Omega;H)$ and $B\in\cE_0$. Then $\one_{\Omega\times B}DF\in L^q(\Omega\times\X;H)$ for all $q\in[1,p)$, and 
\[
\|\one_{\Omega\times B}DF\|_{L^q(\Omega\times\X; H)}
\leq C_B\, \|F\|_{L^p(\Omega;H)}
\]
with a finite constant $C_B=C_{B,\mu,p,q}$
that does not depend on $F$.
\end{lemma}

\begin{proof}
Fix $q\in[1,p)$. The inequality $(a+b)^q\leq 2^{q-1} (a^q+b^q)$ implies
\begin{align*}
\big\|\one_{\Omega\times B}D F\big\|_{L^q(\Omega\times\X;H)}^q
&\leq 2^{q-1}
\E\int_{B}\varepsilon_{x}^+\|F\|^q\,\m(\diffin{x})
+
2^{q-1}\m(B)
\|F\|_{L^q(\Omega;H)}^q.
\end{align*}
By Mecke's formula \eqref{eq:Mecke1} and the H\"older inequality with exponent $p/q>1$ and dual exponent $(p/q)'=p/(p-q)$, 
\begin{align*}
\E\int_{B}\varepsilon_{x}^+\|F\|^q\,\m(\diffin{x})
&=\E\int_B \|F\|^q N(\diffin{x})
=
\E\big(N(B)\,\|F\|^q\big)\\
&\leq   \big\|N(B)\big\|_{L^{p/(p-q)}(\Omega;\R)}
\|F\|_{L^p(\Omega;H)}^q.
\end{align*}
Since $N(B)$ has finite moments of all orders and since $\m(B)\leq \|N(B)\|_{L^{(p/q)'}(\Omega;\R)}$, we have shown the assertion.
\end{proof}

Lemma~\ref{lem:local_estimate_D} has the following immediate and very useful consequence. Recall the conventions on $L^p$ spaces, $p\in\{0\}\cup[1,\infty]$, from Section~\ref{sec:Notation_and_conventions}.

\begin{corollary}\label{cor:localizing_lemma_1}
For all $p>q\geq1$ the restriction of the difference operator $D\colon$ $L^0(\Omega;H)\to L^0(\Omega\times E;H)$ to $L^p(\Omega;H)$ is a continuous mapping from $L^p(\Omega;H)$ to $L^0(E;L^q(\Omega;H))$.
\end{corollary}

\begin{proof}
\felix{see notes 23.11.'16}
\felix{(notes 3.10.'16 = old version)}
Consider $F, F_n\in L^p(\Omega;H)$, $n\in\N$, such that $F_n\xrightarrow{n\to\infty} F$ in $L^p(\Omega;H)$ and take $A\in\cE_0$. Using the Hölder inequality and the local $L^q$-estimate from Lemma~\ref{lem:local_estimate_D}, we have
\begin{align*}
\int_A\big(1\,\wedge &\,\|D_xF_n-D_xF\|_{L^q(\Omega;H)}\big)\m(\dl x)\\
&\leq \m(A)\wedge \Big(\m(A)^{1-\frac1q}\big\|\one_{\Omega\times B}(DF_n-DF)\|_{L^q(\Omega\times E;H)}\Big)\\
&\leq C\big(1\wedge\|F_n-F\|_{L^p(\Omega;H)}\big)\xrightarrow{n\to\infty}0.
\end{align*}
Recalling from Section~\ref{sec:Notation_and_conventions} the canonical metric $d$ on an $L^0$ space, this implies that $DF_n\xrightarrow{n\to\infty}DF$ in $L^0(E;L^q(\Omega;H))$.
\end{proof}

With Corollary~\ref{cor:localizing_lemma_1} at hand we obtain the completeness of $\bD^{1,p}(H)$.

\begin{prop}\label{thm:Banach}
For all $p>1$ the space $\bD^{1,p}(H)$ is complete, i.e.\ a Banach space. In particular, the restriction of $D:L^0(\Omega;H)\to L^0(\Omega\times\X;H)$ to $\bD^{1,p}(H)$ is closed from $L^p(\Omega;H)$ to $L^p(\Omega\times\X;H)$. Moreover, the space $\bD^{1,2}(H)$ is a Hilbert space.
\end{prop}

\begin{proof}
Let $(F_n)_{n\in\N}$ be a Cauchy sequence in $\bD^{1,p}(H)$. There exists $F\in L^p(\Omega;H)$ such that $F_n\to F$ in $L^p(\Omega;H)$ and $\Phi\in L^p(\Omega\times\X;H)$ such that $DF_n\to\Phi$ in $L^p(\Omega\times\X;H)$. It remains to check that $D F=\Phi$ in $L^0(\Omega\times E;H)$, which follows readily from Corollary~\ref{cor:localizing_lemma_1}. 
\end{proof}

In Gaussian Malliavin calculus, see \cite{Nualart2006}, the derivative is in the first step defined on a core of smooth random variables and in a second step extended to Sobolev spaces, by proving closability. This procedure naturally provides an approximation class for limiting arguments. In our approach to Poisson Malliavin calculus, no such class is obtained for free but our next lemma provides one.

\begin{lemma}[Approximation in $\bD^{1,p}(H)$ and $L^p(\Omega;H)$]\label{lem:coreD12}
Let $p>1$. 
\felix{here $p=1$ also ok}
\begin{enumerate}[(i)]
\item
Let $F\in \bD^{1,p}(H)$ and $(e_k)_{k\in\N}$ be an orthonormal basis of $H$. For $n\in\N$ define $F_n\in L^p(\Omega;H)$ by
$$F_n:=\sum_{k\in\N} \Big(\Big(-\frac n{k}\Big)\vee\langle F,e_k\rangle\wedge \frac n{k}\Big)\,e_k.$$
Then, $F_n\in\bD^{1,p}(H)\cap L^\infty(\Omega;H)$ and $F_n\to F$ in $\bD^{1,p}(H)$. In particular, the space $\bD^{1,p}(H)\cap L^\infty(\Omega;H)$ is dense in 
$\bD^{1,p}(H)$.
\item
Every $H$-valued random variable of the form $F=\varphi(N(B_1),\ldots,N(B_n))$, with $B_1,\ldots,B_n\in\cE_0$ and a bounded mapping $\varphi\colon\bN_0^n\to H$, belongs to $\bD^{1,q}(H)$ for all $q>1$. In particular, the space $\bigcap_{q>1}\bD^{1,q}(H)\cap L^\infty(\Omega;H)$ is dense in $L^p(\Omega;H)$.
\end{enumerate}
\end{lemma}

\begin{proof}
\felix{see notes 11.11.'16 (also 9.10.'15, 21.2.'16)}
$(i)$:
Oviously, $F_n$ is bounded. Set $F^{(k)}:=\langle F,e_k\rangle$, $F^{(k)}_n:=(-\frac n{k})\vee\langle F,e_k\rangle\wedge \frac n{k}$ 
and observe that $D_xF_n(\omega)=\sum_{k\in\N}(D_xF^{(k)}_n(\omega))e_k$ and $|D_xF^{(k)}_n(\omega)|\leq |D_xF^{(k)}(\omega)|$ for $x\in\X$, $\omega\in\Omega$, $k,n\in\N$. Thus $\|DF_n(\omega)\|\leq\|DF(\omega)\|$ and therefore $F_n\in\bD^{1,p}(H)$.
Moreover, we have $F^{(k)}_n(\omega)\xrightarrow{n\to\infty}F^{(k)}(\omega)$  as well as
$D_xF^{(k)}_n(\omega)\xrightarrow{n\to\infty} D_xF^{(k)}(\omega)$ for  $x\in\X$, $\omega\in\Omega$, $k\in\N$. 
The latter convergence holds since \linebreak$D_xF^{(k)}_n(\omega)-D_xF^{(k)}(\omega)=0$ whenever $|F^{(k)}(\omega)|\vee|\varepsilon_x^+F^{(k)}(\omega)|\leq \frac n{k}$ and since $\big\{(\omega,x)\in\Omega\times\X:|F^{(k)}(\omega)|\vee|\varepsilon_x^+F^{(k)}(\omega)|\leq \frac n{k}\big\}\nearrow\Omega\times\X$ as $n\to\infty$. 
As $|F^{(k)}_n(\omega)|\leq |F^{(k)}(\omega)|$ and $|D_xF^{(k)}_n(\omega)|\leq |D_xF^{(k)}(\omega)|$ we conclude that
\begin{align*}
&\|F-F_n\|^p_{\bD^{1,p}(H)}
=
\E\big[\|F-F_n\|^p\big]+\E\int_\X\|D_xF-D_xF_n\|^p\,\m(\diffin x)\\
&=
\E\Big[\Big(\sum_{k\in\N}\big(F^{(k)}-F^{(k)}_n\big)^2\Big)^{\frac p2}\Big]+\E\int_\X\Big(\sum_{k\in\N}\big(D_xF^{(k)}-D_xF^{(k)}_n\big)^2\Big)^{\frac p2}\m(\dl x)
\xrightarrow{n\to\infty}0
\end{align*}
by the dominated convergence theorem.

$(ii)$:\felix{see notes 28.12.'16}
As $D_xF=\varphi\big((N+\delta_x)(B_1),\ldots,(N+\delta_x)(B_n)\big)-\varphi(N(B_1),\ldots,N(B_n))$ vanishes on $\Omega\times(\bigcup_{i=1}^nB_i)^c$ and is bounded on $\Omega\times((\bigcup_{i=1}^nB_i)$, it is clear that $F$ belongs to $\bD^{1,q}(H)$ for all $q>1$. A standard monotone class argument (see, e.g., Theorem~A3 in \cite{LastPenrose2016}) yields that the set of all random variables $F$ of the described form lies dense in $L^p(\Omega;H)$.
\end{proof}

For our next result we introduce sub-$\sigma$-algebras of $\mathcal F$. To every set $A\in\cE$ we associate the $\sigma$-algebra 
\begin{equation}\label{eq:defFA}
\mathcal F_A:=\overline{\sigma\big(N(A\cap B):B\in\cE\big)}^\P.
\end{equation}
where $\overline{\vphantom{(}\dots}^\P$ denotes the $\P$-completion. 
By slight abuse of terminology we say that $F\in L^0(\Omega;H)$ is $\mathcal F_A$-measurable if the equivalence class of random variables $F$ has a $\mathcal F_A$-measurable representative.

\begin{lemma}\label{lem:DF_F_FA_measurable}
Let $A\in\cE$ and $F\in L^0(\Omega;H)$ be $\mathcal F_A$-measurable. Then $DF=0$ $\P\otimes \m$-almost everywhere on $\Omega\times A^c$.
\end{lemma}

\begin{proof}
Given a $\sigma$-finite $\N_0\cup\{+\infty\}$-valued measure $\eta\in\NN$, let $\eta_A\in\NN$ be defined by $\eta_A(B):=\eta(A\cap B)$, $B\in\cE$, i.e.\ by removing from $\eta$ all point masses that are located outside of $A$. Then, $\cF_A$ coincides with the 
$\P$-completion of the $\sigma$-algebra that is generated by the mapping $N_A\colon\Omega\to\NN,\,\omega\mapsto N_A(\omega):=(N(\omega))_A$.
The measurability assumption and the factorization lemma imply that $F$ has a representative of the form $f(N_A)$, where $f$ is a measurable function from $\NN$ to $H$. Obviously, $f(N_A)=f_A(N)$ with $f_A\colon\NN\to H$ defined by $f_A(\eta):=f(\eta_A)$, $\eta\in\NN$. Now the assertion follows from the identity $D_x F=f_A(N+\delta_x)-f_A(N)$ in $L^0(\Omega\times\X;H)$.
\end{proof}

We end this subsection with a remark concerning the natural extention of the difference operator $D$ to random fields, see also Remark~\ref{rem:epsilon_for_random_fields}.

\begin{remark}[Difference operator for random fields]
\label{rem:D_for_random_fields}
If $( S,\cS,m)$ is any $\sigma$-finite measure space, we define the operator 
$$
D\colon L^0(\Omega\times S;H)\to L^0(\Omega\times S\times E;H),\;F=(F(s))\mapsto DF=(D_xF(s))
$$
by setting 
$$
D_xF(s):= \varepsilon^+_xF(s)-F(s),\quad s\in S,\;x\in E.
$$ 
That is, we have $D_xF(s)=f(N+\delta_x,s)-f(N,s)$, where $f\colon\NN\times S\to H$ is a measurable mapping such that $f(N,\cdot)$ is a representative of (the equivalence class of random functions) $F\in L^0(\Omega\times S;H)$. Note that we are slightly abusing notation here as $D$ is not the same operator as in Definition \ref{def:diffOp}. 
However, if we consider 
\felix{see notes 1.11.'16}
a fixed version of $F\in L^0(\Omega\times S;H)$ and a fixed version of $DF\in L^0(\Omega\times S\times E;H)$ (denoted by $F$ and $DF$ again), then
$$
\text{for $m$-almost all }s\in S:\quad D_x(F(s))=D_x F(s) 
\text{ in } L^0(\Omega\times\X;H),
$$
where $D_x(F(s))$ is the derivative of $F(s)\in L^0(\Omega;H)$ in the sense of Definition~\ref{def:diffOp} and $D_xF(s)$ is the derivative of the random field $F\in L^0(\Omega\times S;H)$ as introduced above.
\felix{see notes 23.11.'16 (notes 1.11.'16 = old version)}
Moreover, arguing similarly as in the proof of Corollary~\ref{cor:localizing_lemma_1}, we obtain the following continuity property which will be used repeatedly:
For all $p>q\geq1$, the restriction of the difference operator $D\colon$ $L^0(\Omega\times S;H)\to L^0(\Omega\times S\times E;H)$ to $L^0(S;L^p(\Omega;H))\subset L^0(\Omega\times S;H)$ is a continuous mapping from $L^0(S;L^p(\Omega;H))$ to $L^0(S\times E;L^q(\Omega;H))$. 
\end{remark}
\felix{old version of this remark commented our (better version?)}


\subsection{Kabanov-Skorohod integral and duality}
\label{subsec:Kabanov-Skorohod}

In the real-valued case $H=\bR$ the Kabanov-Skohorod integral, originally defined as a creation operator in terms of chaos expansions in an $L^2$ setting \cite{Kabanov1976},\felix{check reference} 
coincides with the adjoint of the restriction $D|_{\bD^{1,2}(\bR)}$ of the difference operator $D$ to the space $\bD^{1,2}(\bR)$, see \cite{LastPenrose2011} and the references therein. Pathwise interpretations can be found in \cite{LastPenrose2011, Picard1996a, privault2009}\felix{check (and add?) references}. In this section we first introduce a pathwise defined, Hilbert space-valued Kabanov-Skorohod integral $\delta\colon L^1(\Omega\times E;H)\to L^1(\Omega;H)$, based on a corresponding generalization of Mecke's formula in terms of pathwise Bochner integrals. This operator satisfies a duality relation w.r.t.\ the restriction $D|_{L^\infty(\Omega;H)}$ of $D$ to $L^\infty(\Omega;H)$. We then introduce a family of operators $\delta^{(p)}\colon\dom(\delta^{(p)})\subset L^p(\Omega\times E;H)\to L^p(\Omega;H)$, $p>1$, as adjoint operators to $D|_{\bD^{1,p'}(H)}$, $\frac1p+\frac1{p'}=1$, $p>1$, and show that they coincide with each other and with the pathwise Kabonov-Skorohod integral on the intersections of their domains.  
Although our approach does not rely on chaos-expansions, we use the terminology `Kabanov-Skorohod integral' known from the real-valued $L^2$ setting, cf.~\cite{Last2014,LastPenrose2011}.
We also present sufficient conditions for a random field $\Phi\colon\Omega\times E\to H$ to belong to $\dom(\delta^{(p)})$ as well as a local duality formula.


\begin{prop}[$H$-valued Mecke formula]\label{lem:stoch_int}
The integral mapping
\begin{align*}
  L^1(\Omega\times \X;H)
  \ni
  \Phi
  \mapsto
  \int_\X
    \varepsilon_x^-\Phi(x)
  N(\diffin{x})
  \in
  L^1(\Omega;H),
\end{align*}
where $\int_\X\varepsilon_x^-\Phi(x)N(\diffin{x})$ is $\bP$-almost surely defined as an $H$-valued Bochner integral, is well-defined and continuous. 
For all $\Phi\in L^1(\Omega\times\X;H)$ it holds that
\begin{align}
\label{eq:bound_stoch_int}
  \Big\|
    \int_{\X}
      \varepsilon_x^-
      \Phi(x)  
    N(\diffin{x})
  \Big\|_{L^1(\Omega;H)}
  \leq
  \|\Phi\|_{L^1(\Omega\times\X;H)}
\end{align}
and 
\begin{align}
\label{eq:Mecke_H-valued}
  \E
  \Big[
  \int_{\X}
    \varepsilon_x^-
    \Phi(x)  
  N(\diffin{x})
  \Big]
  =
  \E
  \Big[
    \int_\X
      \Phi(x)
    \,\m(\diffin{x})
  \Big].
\end{align}
\end{prop}

\begin{proof}
By the definition \eqref{eq:epsilon-} of $\varepsilon^-$ and Mecke's formula \eqref{eq:Mecke2} 
we have
\begin{align}\label{eq:Mecke_stoch_int}
  \E
  \Big[
    \int_\X
      \big\|
        \varepsilon_x^-
        \Phi(x)
      \big\|
    \,N(\diffin{x})
  \Big]
  =
  \E
  \Big[
  \int_{\X}
    \varepsilon_x^-
    \big\|
      \Phi(x)
    \big\|
  \,N(\diffin{x})
  \Big]
  =
  \E
  \Big[
    \int_\X
      \big\|\Phi(x)\big\|
    \,\m(\diffin{x})
  \Big]
  .
\end{align}
This implies that
$
    \int_\X
      \|
        \varepsilon_x^-
        \Phi(x)
      \|
    N(\diffin{x})
$
is finite $\P$-almost surely if $\Phi\in L^1(\Omega\times E;H)$, and therefore
$\int_\X\varepsilon_x^-\Phi(x)\,N(\diffin{x})$
exists $\P$-almost surely as a Bochner integral in $H$. By Lemma~\ref{lem:epsilon-welldefined} the value of $\int_\X\varepsilon_x^-\Phi(x)\,N(\diffin{x})$ is $\bP$-almost surely independent of the choice of the representative of $\Phi$. The $\mathcal F$-$\mathcal B(H)$-measurability of the mapping $\omega\mapsto  \int_\X\varepsilon_x^-\Phi(\omega,x)
N(\omega,\diffin{x})$ 
follows from the completeness of $\cF$ and, e.g., the measureability of (any fixed version of) $(\omega,x)\mapsto \varepsilon^-_x\Phi(\omega,x)$ and the representation \eqref{eq:representationN} of $N$.
By the Bochner inequality it holds that
\begin{align*}
  \Big\|
    \int_{\X}
      \varepsilon_x^-
      \Phi(x)  
    N(\diffin{x})
  \Big\|
  \leq
  \int_{\X}
    \big\|
      \varepsilon_x^-
      \Phi(x)
    \big\|
  N(\diffin{x})
\end{align*}
$\P$-almost surely. This and \eqref{eq:Mecke_stoch_int} imply the estimate \eqref{eq:bound_stoch_int}. 
Concerning the last assertion we notice that for simple integrands $\Phi\in L^1(\Omega\times E;H)$ of the form 
\begin{align*}
  \Phi
  =
  \sum_{k=1}^K 
  f_k \otimes h_k,
  \quad
  (f_k)_{k=1}^K\subset L^1(\Omega\times E;\R^+)
  ,\ 
  (h_k)_{k=1}^K\subset H
  ,\ 
  K\in\bN,
\end{align*}
it holds by Mecke's fomula~\eqref{eq:Mecke2} that
\begin{equation}
\begin{split}
\label{eq:Mecke_simple}
  \E
  \Big[
  \int_{\X}
    \varepsilon_x^-
    \Phi(x)  
  N(\diffin{x})
  \Big]
&=
  \sum_{k=1}^K
  \E
  \Big[
    \int_\X
      \varepsilon_x^- f_k(x)
    N(\diffin{x})
  \Big]
  \,h_k
  =
  \sum_{k=1}^K
  \E
  \Big[
    \int_\X
      f_k(x)
    \,\m(\diffin{x})
  \Big]
  \,h_k\\
&=
  \E
  \Big[
    \int_\X
      \sum_{k=1}^K
        f_i(x)\,
        h_k
    \,\m(\diffin{x})
  \Big]
  =
  \E
  \Big[
    \int_\X
      \Phi(x)
    \,\m(\diffin{x})
  \Big]
  \in H.
\end{split}
\end{equation}
For general integrands $\Phi\in L^1(\Omega\times E;H)$ the equality~\eqref{eq:Mecke_H-valued} follows from \eqref{eq:Mecke_simple} and approximation of $\Phi$ by simple integrands in $L^1(\Omega\times E;H)$, using the continuity estimates~\eqref{eq:bound_stoch_int} and $\big\|\int_E\Phi(x)\,\m(\dl x)\big\|_{L^1(\Omega;H)}\leq\|\Phi\|_{L^1(\Omega\times E;H)}$ and the fact that the expectation operator $\bE[\ldots]$ is a continuous from $L^1(\Omega;H)$ to $H$.
\end{proof}

Due to Proposition~\ref{lem:stoch_int} the following definition is meaningful.

\begin{definition}[Pathwise Kabanov-Skorohod integral]
\label{def:pathwise_Kabanov}
Let the operator 
\[
  \delta\colon L^1(\Omega\times E;H)\to L^1(\Omega;H),\;\Phi\mapsto\delta(\Phi)
\]
be defined by
\begin{align*}
  \delta(\Phi)
  :=
  \int_\X
    \varepsilon_x^-\Phi(x)
  N(\diffin{x})
  -
  \int_\X
    \Phi(x)
  \,\m(\diffin{x}).
\end{align*}
\end{definition}

The $H$-valued Mecke formula immediately implies that 
$
  \E\,\delta(\Phi)=0
$ for all
$
  \Phi\in L^1(\Omega\times\X;H)
$. Our next proposition shows that the adjoint of $\delta$ coincides with the restriction $D|_{L^\infty(\Omega;H)}\colon L^\infty(\Omega;H)\to L^\infty(\Omega\times E;H)$ of $D$  to $L^\infty(\Omega;H)$.
\begin{prop}[$L^1$-$L^\infty$-duality]\label{lem:duality}
For all $F\in L^\infty(\Omega;H)$ and $\Phi\in L^1(\Omega\times\X;H)$ it holds that
\begin{align*}
  \tensor[_{L^\infty(\Omega,H)}]
  {\big\langle
    F, \delta(\Phi)
  \big\rangle}{_{L^1(\Omega;H)}}
  =
 \tensor[_{L^\infty(\Omega\times\X;H)}]
 	{\big\langle DF,\Phi\big\rangle}
 	{_{L^1(\Omega\times\X;H)}}.
\end{align*}
\end{prop}

\begin{proof}
Using the definitions of $\delta$ and $D$ the statement reads
\begin{align*}
&\E\,
  \Big\langle
    F
    ,
    \int_\X
      \varepsilon_x^-
      \Phi(x)
    \,N(\diffin{x})
    -    
    \int_\X
      \Phi(x)
    \,\m(\diffin{x})
  \Big\rangle
  =
  \E
  \int_\X
    \big\langle 
      \varepsilon_x^+F-F,\Phi(x)
    \big\rangle
  \,\m(\diffin{x}).
\end{align*}
It reduces to
$$
  \E
  \Big\langle
    F
    ,
    \int_\X
      \varepsilon_x^-\Phi(x)
    \,N(\diffin{x})
  \Big\rangle
  =
  \E
    \int_\X
      \big\langle
        \varepsilon_x^+F
        ,
        \Phi(x)
      \big\rangle
    \,\m(\diffin{x}).
$$
Due to standard properties of Bochner integrals, the inner product on the left hand side can be moved inside the pathwise integral w.r.t.~$N$. Further, 
\begin{align*}
\bE\int_E\big\langle F,\varepsilon^-_x\Phi(x)\big\rangle\,N(\dl x)
=
\bE\int_E\varepsilon^-_x\big\langle \varepsilon^+_xF,\Phi(x)\big\rangle\,N(\dl x)
=
\bE\int_E\big\langle \varepsilon^+_xF,\Phi(x)\big\rangle\,\m(\dl x),
\end{align*} 
by Mecke's formula~\eqref{eq:N+delta},
which finishes the proof.
\end{proof}

The duality relation in Proposition~\ref{lem:duality} motivates the following definition of operators $\delta^{(p)}$, $p>1$, as realizations of $\delta$ in $L^p$. 

\begin{definition}[Abstract Kabanov-Skorohod integral]
\label{def:Kabanov-Skorohod}
Let $p,p'>1$ be such that $\frac1p+\frac1{p'}=1$. The operator
 \[\delta^{(p)}\colon\dom\big(\delta^{(p)}\big)\subset L^p(\Omega\times\X;H)\to L^p(\Omega;H) \]
is defined as the adjoint of 
\begin{align*}
D|_{\bD^{1,p'}(H)}\colon\bD^{1,p'}(H)\subset L^{p'}(\Omega;H) \to L^{p'}(\Omega\times\X;H),
\end{align*}
i.e.\ of the restriction of $D$ to $\bD^{1,p'}(H)$, considered as a densely defined operator from $L^{p'}(\Omega;H)$ to $L^{p'}(\Omega\times\X;H)$. Here\felix{compare notes 14.11.'16 --2--} we identify $L^p(\Omega\times E;H)$ and $L^p(\Omega;H)$ with the dual spaces of $L^{p'}(\Omega\times E;H)$ and $L^{p'}(\Omega;H)$, respectively.
\end{definition}

We now verify the compatibility of Definition~\ref{def:pathwise_Kabanov} and \ref{def:Kabanov-Skorohod} and furthermore show that an integrand $\Phi\in L^1(\Omega\times\X;H)\cap L^p(\Omega\times\X;H)$ belongs to $\dom(\delta^{(p)})$ iff the pathwise integral $\delta(\Phi)$ belongs to $L^p(\Omega;H)$.

\begin{prop}\label{prop:coincides}
The definitions of the operators $\delta$, $\delta^{(p)}$, $p>1$, are compatible in the following sense:
\begin{enumerate}[(i)]
\item
For $\Phi\in L^1(\Omega\times\X;H)\cap L^p(\Omega\times\X;H)$ it holds that $\Phi\in\dom(\delta^{(p)})$ if, and only if, $\delta(\Phi)\in L^p(\Omega;H)$. In this case we have $\delta(\Phi)=\delta^{(p)}(\Phi)$.
\item
If $\Phi\in\dom(\delta^{(p)})\cap\dom(\delta^{(q)})$ for some $p,q>1$, then $\delta^{(p)}(\Phi)=\delta^{(q)}(\Phi)$.
\end{enumerate}
\end{prop}

\begin{proof}
Let $p':=\frac{p}{p-1}$, $q':=\frac{q}{q-1}$ be the conjugate exponents to $p,q>1$.
\felix{see notes 29.12.'16}
We first prove the assertion (i):
Let $\Phi\in L^1(\Omega\times\X;H)\cap L^p(\Omega\times\X;H)$ be such that $\delta(\Phi)\in L^p(\Omega;H)$.
By the duality formula in Lemma~\ref{lem:duality} we have for all $F\in\bD^{1,p'}(H)\cap L^\infty(\Omega;H)$ that
\begin{align*}
&\tensor[_{L^{p'}(\Omega;H)}]{\big\langle F,\delta(\Phi)\big\rangle}{_{L^p(\Omega;H)}}
=
\tensor[_{L^{\infty}(\Omega;H)}]{\big\langle F,\delta(\Phi)\big\rangle}{_{L^1(\Omega;H)}}\\
&\quad=
\tensor[_{L^{\infty}(\Omega\times E;H)}]{\big\langle DF,\Phi\big\rangle}{_{L^1(\Omega\times E;H)}}
=
\tensor[_{L^{p'}(\Omega\times E;H)}]{\big\langle DF,\Phi\big\rangle}{_{L^p(\Omega\times E;H)}}.
\end{align*}
Since $\bD^{1,p'}(H)\cap L^\infty(\Omega;H)$ is dense in $\bD^{1,p'}(H)$ according to Lemma~\ref{lem:coreD12}(i), 
we obtain 
$\tensor[_{L^{p'}(\Omega;H)}]{\langle F,\delta(\Phi)\rangle}{_{L^p(\Omega;H)}}=\tensor[_{L^{p'}(\Omega\times E;H)}]{\langle DF,\Phi\rangle}{_{L^p(\Omega\times E;H)}}$
for all $F\in\bD^{1,p'}(H)$.
We conclude that $\Phi\in\dom(\delta^{(p)})$ and $\delta(\Phi)=\delta^{(p)}(\Phi)$.
The converse implication 
\felix{see notes 4.2.'17--B--}
follows similarly as the assertion (ii) below if one uses a suitable monotone class argument. 
Next we verify the assertion (ii):
For $F\in\bD^{1,p'}(H)\cap\bD^{1,q'}(H)$, the definition of the operators $\delta^{(p)}$, $\delta^{(q)}$ and the assumption that $\Phi\in \dom(\delta^{(p)})\cap\dom(\delta^{(q)})$ imply
\begin{align*}
\bE\big\langle F,\delta^{(p)}(\Phi)\big\rangle
=
\bE\int_E\big\langle D_xF,\Phi(x)\big\rangle\,\m(\dl x)
=
\bE\big\langle F,\delta^{(q)}(\Phi)\big\rangle.
\end{align*}
As the space $\bD^{1,p'}(H)\cap\bD^{1,q'}(H)$ is dense in $L^{p'}(\Omega;H)$ and in $L^{q'}(\Omega;H)$ according to Lemma~\ref{lem:coreD12}(ii), the assertion follows.
\end{proof}

Proposition~\ref{prop:coincides} allows us to simplify notation: In the sequel we write 
\begin{equation}\label{eq:notation_Kabanov-Skorohod}
\delta(\Phi):=\delta^{(p)}(\Phi),\quad \Phi\in\dom(\delta^{(p)}),\;p>1.
\end{equation}
The next lemma is a complement to the local $L^q$-estimate in Lemma~\ref{lem:local_estimate_D}. In combination with Propostion~\ref{prop:coincides}(i) it particularly implies that every random field $\Phi\in L^p(\Omega\times E;H)$, $p>1$, which vanishes outside a set $\Omega\times B$ , $B\in\cE_0$, belongs to $\dom(\delta^{(q)})$ for any $q\in(1,p)$. We refer to Remark~\ref{rem:Kabanov-Skorohod_1} and Propostion~\ref{prop:L2D12_subset_domdelta}, \ref{prop:space_time1} below for further sufficient conditions for a random field $\Phi$ to belong to $\dom(\delta^{(p)})$.

\begin{lemma}[Local $L^q$-estimate]\label{lem:local_estimate_delta}
Let $p>1$ and $\Phi\in L^p(\Omega\times\X;H)$ be such that $\Phi=0$ $\P\otimes \m$-almost everywhere on $\Omega\times B^c$ for some $B\in\cE_0$. Then $\delta(\Phi)\in L^q(\Omega;H)$ for all $q\in[1,p)$, and
\begin{align*}
\|\delta(\Phi)\|_{L^q(\Omega;H)}
\leq
C_B\,\|\Phi\|_{L^p(\Omega\times\X;H)}
\end{align*}
with a finite constant 
$C_B=C_{B,\mu,p,q}$ 
that does not depend on $\Phi$.
\end{lemma}

\begin{proof}\felix{see notes 3.10.'15}
We fix $q\in[1,p)$ and note that
\begin{align*}
\|\delta(\Phi)\|_{L^q(\Omega;H)}^q
\leq 
2^{q-1}\E\Big(\Big\|\int_{\X} \varepsilon_x^-\Phi(x)N(\diffin{x}) \Big\|^q
+\Big\|\int_{\X} \Phi(x)\m(\diffin{x}) \Big\|^q\Big).
\end{align*}
Using Jensen's inequality and Mecke's formula \eqref{eq:Mecke1}, we obtain
\begin{align*}
  \E\Big\|\int_{\X} \varepsilon_x^-\Phi(x)N(\diffin{x}) \Big\|^q
&\leq
  \E\Big(
  N(B)^{q-1}\int_B
  \big\|
    \varepsilon_x^- \Phi(x)
  \big\|^q N(\diffin{x})\Big)\\
&=
  \E
  \int_B
  (N(B)+1)^{q-1}
  \|
    \Phi(x)  
  \|^q 
  \,\m(\diffin{x}).
\end{align*}
Next we apply Hölder's inequality with exponent $p/q>1$ and dual exponent $(p/q)'=p/(p-q)$ to get
\begin{align*}
  \E\Big\|\int_{\X} \varepsilon_x^-\Phi(x)N(\diffin{x}) \Big\|^q
&\leq
  \Big(\E\big[(N(B)+1)^{\frac{(q-1)p}{p-q}}\big]\m(B)\Big)^{\frac{p-q}{p}}  \|\Phi\|_{L^p(\Omega\times\X;H)}^q.
\end{align*}
We complete the proof by observing that
\begin{align*}
  \E\Big\|\int_{\X} \Phi(x)\m(\diffin{x}) \Big\|^q
  \leq
\m(B)^{\frac{(p-1)q}{p}}  \|\Phi\|_{L^p(\Omega\times\X;H)}^q,
\end{align*}
which holds by a similar calculation.
\end{proof}

\begin{remark}[Generalized predictability]
\label{rem:Kabanov-Skorohod_1}
\felix{see notes 4.1.'17}
Arguing similarly as in the proof of Lemma~3.9 and using the closedness of $\delta=\delta^{(p)}$ one obtains the following improvement of Lemma~\ref{lem:local_estimate_delta} for specific integrands $\Phi$ which enjoy a measurability structure that is related to the notion of predictabiliy known from space-time settings: If $B\in\cE_0$, $p>1$ and $F\in L^p(\Omega;H)$ is $\cF_{B^c}$-measurable, then $\Phi:=\one_{\Omega\times B} F$ belongs to $\dom(\delta^{(p)})$ and
\begin{align}\label{eq:Kabanov_predictable}
\delta(\Phi)=\tilde N(B)F.
\end{align}
Equality \eqref{eq:Kabanov_predictable} remains valid for $\cF_{B^c}$-measurable $F\in L^1(\Omega;H)$, i.e.\ for $p=1$. It can even be generalized to all $\Phi\in L^1(\Omega\times E;H)$ which have a version that is measurable w.r.t.\ the $\sigma$-algebra $\sigma(B\times A:B\in\cE_0, A\in \cF_{B^c})$ (the measures $\bP\otimes\m$ and $\bP\otimes N$ coincide on this $\sigma$-algebra, compare \cite[Théorème 1]{Picard1996a}).  However for $p>1$ such a generalization requires a suitable order structure on $E$, compare~\cite{LastPenrose2011b}. We therefore refrain from further details at this point and refer to Section~\ref{sec:space_time}, in particular Lemma~\ref{lem:space_time1} and Proposition~\ref{prop:space_time1}, for corresponding considerations in a space-time setting.
\end{remark}
%

Lemma~\ref{lem:local_estimate_delta} also allows us to establish the following `local' duality formula, which is useful as it does not rely on a integrability assumption for $DF$. The formula will particularly be crucial
 in the proofs of Proposition~\ref{prop:L2D12_subset_domdelta} ($L^2$-isometry) and Proposition~\ref{prop:commutator_relation_L^2} ($L^2$-commutation relation) below.
\begin{lemma}[Local duality formula]\label{lemma:duality2}
Let $1\leq q<p$, let $\Phi\in L^{p}(\Omega\times\X;H)$ be such that $\Phi=0$ $\P\otimes\m$-almost everywhere on $\Omega\times B^c$ for some $B\in\cE_0$, and $F\in  L^{q'}(\Omega;H)$ with $q':=\frac q{q-1}$ (where we set $\frac10:=\infty$). Then $\langle F,\delta(\Phi)\rangle\in L^1(\Omega;\bR)$, $\langle DF,\Phi\rangle\in L^1(\Omega\times E;\bR)$, and 
\[
\E\langle F,\delta(\Phi)\rangle=\E\int_B\langle DF,\Phi\rangle\,\m(\diffin x).
\]
\end{lemma}

\begin{proof}
\felix{see notes 3.10.'15}
Note that $\Phi\in L^1(\Omega\times E;H)$, so that $\delta(\Phi)$ is defined in the sense of Definition~\ref{def:pathwise_Kabanov}. Lemma~\ref{lem:local_estimate_delta} implies that $\delta(\Phi)\in L^q(\Omega;H)$ and therefore $\langle F,\delta(\Phi)\rangle\in L^1(\Omega;\bR)$ by Hölder's inequality. Moreover, according to Lemma~\ref{lem:local_estimate_D} we know that $\one_{\Omega\times B}DF\in L^{p'}(\Omega\times E;H)$ with $p':=\frac p{p-1}<q'$. Thus $\langle DF,\Phi\rangle\in L^1(\Omega\times E;\bR)$ by Hölder's inequality.
The claimed identity follows by Mecke's formula \eqref{eq:Mecke2} similarly as in the proof of Lemma~\ref{lem:duality}, using both the fact that $\int_B\varepsilon^-_x\Phi(x)N(\diffin x)\in L^{q}(\Omega;H)$, which holds due to Lemma~\ref{lem:local_estimate_delta} and since $\int_B\Phi(x)\,\m(\diffin x)\in L^{q}(\Omega;H)$, and the fact that $\one_{\Omega\times B}\,\varepsilon^+_\cdot F\in L^{p'}(\Omega\times\X;H)$, which follows from Lemma~\ref{lem:local_estimate_D} and the assumption that $F\in L^{q'}(\Omega;H)\subset L^{p'}(\Omega;H)$.
\end{proof}

Complementing Remark~\ref{rem:D_for_random_fields} on the application of the difference operator $D$ to random fields, we end this subsection with a remark concerning the natural extension of the Kabanov-Skorohod integral $\delta$ to parameter-dependent random fields, compare also Remark~\ref{rem:epsilon_for_random_fields} and \ref{rem:D_for_random_fields}. 

\begin{remark}[Parameter-dependent Kabanov-Skorohod integral]
\label{rem:Kabanov-Skorohod_2}
\felix{see notes 15.1.'17}
Let $(S,\cS,m)$ \linebreak be any $\sigma$-finite measure space and $p\geq1$; in case $p=1$ we set $\dom(\delta^{(p)}):=L^1(\Omega\times E;H)$. If $\Phi\in L^0(\Omega\times S\times E;H)$ is such that for any version of $\Phi$, denoted again by $\Phi$, we have $\Phi(s,\cdot)\in \dom(\delta^{(p)})$ for $m$-almost all $s\in S$, then the parameter-dependent Kabanov-Skorohod integral $\delta(\Phi(s,\cdot))$ is defined for $m$-almost all $s\in S$ as an element in $L^p(\Omega;H)$. Recalling the conventions concerning $L^0$-spaces from Secion~\ref{sec:Notation_and_conventions}, we identify $\Phi$ with the corresponding element in $L^0(S;\dom(\delta^{(p)}))\subset L^0(S;L^p(\Omega\times E;H))\subset L^0(\Omega\times S\times E;H)$ and, accordingly, $\big(\delta(\Phi(s,\cdot))\big)_{s\in S}$ is considered as an element of $L^0(S;L^p(\Omega;H))\subset L^0(\Omega\times S;H)$. We thus obtain an operator
\begin{align*}
\delta\colon L^0(S;\dom(\delta^{(p)}))\subset L^0(S;L^p(\Omega\times E;H))\to L^0(S;L^p(\Omega;H)),
\end{align*}
for which, by slightly abusing notation, we will
use the same notation  $\delta$ as for the operators in Definition~\ref{def:pathwise_Kabanov} and in \eqref{eq:notation_Kabanov-Skorohod}.
\end{remark}

\subsection{Commutation relations}
\label{sec:commutation}
In addition to Setting~\ref{setting} we assume throughout this subsection that $(E,\cE)$ is a Borel space, i.e., that there exists a Borel-measurable bijection from $E$ to a Borel subset of the unit interval $[0,1]$ with measurable inverse.  
This additional assumption is used to obtain Lemma~\ref{lem:commutator} below.
It is well known that every Borel subset of a Polish space is a Borel space.

It is often of interest to determine the derivative $DF$ of a random variable $F=\delta(\Phi)$ which is itself given as the Skorohod-Kabanov integral of some random field $\Phi$. The commutation relations presented here describe the structure of the concatenation of $\delta$ and $D$. In this general setting we restrict ourselves to results in $L^1$ and in $L^2$. In Section~\ref{sec:space_time}, where we consider a space-time setting, we also obtain $L^p$-results, $p\in[1,2]$, by exploiting the continuity properties of the Itô integral w.r.t. the compensated Poisson random measure $\tilde N$.

The proof of the next lemma follows the lines of the proof of \cite[Proposition~6.2]{LastPenrose2016}, up to some obvious modifications, and is therefore
 omitted. The assumption that $(E,\cE)$ is a Borel space is essential at this point.
The lemma is used in the proof of Proposition~\ref{lem:commutator1} to verify the measurability of the mapping $f\colon\mathbf N\to H$ therein.

\begin{lemma}\label{lem:commutator}\felix{Notation?}
\felix{see notes 22.1.'17}
Let $(B_i)_{i\in\bN}\subset\cE$ be a partition of $E$ such that $\m(B_i)<\infty$ for all $i\in\bN$ and set $\mathbf N_{\m}:=\{\eta\in\mathbf N:\eta(B_i)<\infty\text{ for all }i\in\bN\}$. Then there exist measurable mappings $\pi_n\colon \mathbf N\to E$, $n\in\bN$, such that
\begin{align*}
\eta=\sum_{n=1}^{\eta(E)}\delta_{\pi_n(\eta)},\quad\eta\in\mathbf N_\m.
\end{align*}
\end{lemma}

The commutation relations in Propositions~\ref{lem:commutator1} and \ref{prop:commutator_relation_L^2} below involve applications of $D$ to random fields and applications of $\delta$ to parameter-dependent random fields, compare Remarks \ref{rem:D_for_random_fields} and \ref{rem:Kabanov-Skorohod_2}. Note that the derivative $D\Phi$ of a random function $\Phi\in L^0(\Omega\times E;H)$ is an element of $L^0(\Omega\times E\times E;H)$, and for fixed $x\in E$ we have $D_x\Phi\in L^0(\Omega\times E;H)$.

\begin{prop}[Commutation relation -- $L^1$-version]\label{lem:commutator1}
Let $\Phi\in L^1(\Omega\times\X;H)$ be such that $D_x\Phi\in L^1(\Omega\times\X;H)$ for $\m$-almost all $x\in E$. 
Then the equality  
\begin{align*}
  D_x\delta(\Phi)
  =
  \delta(D_x\Phi)+\Phi(x)
\end{align*}
holds $\bP\otimes\m$-almost everywhere.
\end{prop}

\begin{proof}
It holds that 
\begin{equation}\label{eq:proof_commutation_L1_1}
\begin{aligned}
  &D_x\delta(\Phi)
=
  D_x \int_\X
      \varepsilon_y^-
      \Phi(y)
    \,N(\diffin{y})
    -
    D_x\int_\X
      \Phi(y)
    \,\m(\diffin{y})\\
&\quad=
  \int_E
    \varepsilon_x^+\varepsilon_y^-\Phi(y)
  \big(
    N + \delta_x
  \big)(\diffin{y})
  -
  \int_E
    \varepsilon_y^-\Phi(y)
    N(\diffin{y})
  -
  \int_\X
    D_x\Phi(y)
  \m(\diffin{y})\\
&\quad=
  \int_E
    \varepsilon_y^-\varepsilon_x^+\Phi(y)
  N(\diffin{y})
  +
  \Phi(x)
  -
  \int_E
    \varepsilon_y^-\Phi(y)
    N(\diffin{y})
  -
  \int_\X
    D_x\Phi(y)
  \m(\diffin{y})\\  
&\quad=
  \delta(D_x\Phi)+\Phi(x).
\end{aligned}
\end{equation}  
In order to justify the application of $D_x$ 
\felix{see notes 22.1.'17}
to $\int_E\varepsilon_y^-\Phi(y)\,N(\diffin{y})$ in \eqref{eq:proof_commutation_L1_1} we take a measurable function $\varphi\colon\NN\times\X\to H$ such that $\varphi(N,\cdot)$ is a version of $\Phi$ and let $\mathbf N_\m\in\mathcal N$ be as in Lemma~\ref{lem:commutator}. Define $f\colon\NN\to H$ by
\[
f(\eta):=
\begin{cases}
\int_\X\varphi\big(\eta\setminus\delta_y,y\big)\,\eta(\diffin y),&\text{if }\eta\in\mathbf N_\m\text{ and }\int_\X\big\|\varphi\big(\eta\setminus\delta_y,y\big)\big\|\,\eta(\diffin y)<\infty\\
0,&\text{else.}
\end{cases}
\]
Thanks to Lemma~\ref{lem:commutator} one can check that $f$ is $\mathcal N$-$\cB(H)$-measurable, compare Remark~\ref{rem:measurability}. We have $N\in\mathbf N_\m$ $\bP$-a.s. and $\int_E\|\varphi(N\setminus\delta_y,y)\|N(\dl y)<\infty$ $\bP$-a.s.; the latter holds due to Mecke's formula \eqref{eq:Mecke2} and since $\varphi(N,\cdot)=\Phi\in L^1(\Omega\times\X;H)$.
As a consequence, $f(N)$ is a version of
$\int_E\varepsilon_y^-\Phi(y)\,N(\diffin{y})$
and therefore, $\bP$-almost surely,
\begin{equation}\label{eq:proof_commutation_L1_2}
\begin{aligned}
& D_x\int_\X\varepsilon^-_y\Phi(y)N(\diffin y)
  =
 f(N+\delta_x)-f(N)\\
&\quad
  = 
  \int_E\varphi\big((N+\delta_x)\setminus\delta_y,y\big)\,(N+\delta_x)(\diffin y)
  -
  \int_E\varphi\big(N\setminus\delta_y,y\big)\,N(\diffin y).
\end{aligned}
\end{equation}
Note that, for $\m$-almost all $x\in E$, the first integral in the second line of \eqref{eq:proof_commutation_L1_2} exists $\bP$-almost surely.
\felix{see notes 25.1.'17}
This is due to the fact that  $\bE\int_E\|\varphi(N+\delta_x,y)\|\,\m(\dl y)=\bE\int_E\|\varepsilon_x^+\Phi(y)\|\,\m(\dl y)<\infty$ for $\m$-almost all $x\in E$, which follows from our assumptions, and due to Mecke's formula \eqref{eq:Mecke2}, which yields that $\bE\int_E\|\varphi((N+\delta_x)\setminus \delta_y,y)\|N(\dl y)<\infty$ for $\m$-almost all $x\in E$. In a similar way one can justify the application of $D_x$ to $\int_E\Phi(y)\m(\dl y)$ in \eqref{eq:proof_commutation_L1_1}. 
Finally, observe that the third equality in \eqref{eq:proof_commutation_L1_1} holds as 
\begin{align*}
\int_E\varepsilon_x^+\varepsilon_y^-\Phi(y)\,N(\dl y)
&=
\int_E\varphi\big((N+\delta_x)\setminus\delta_y,y\big)\,N(\dl y)\\
&=
\int_E\varphi\big((N\setminus\delta_y)+\delta_x,y\big)\,N(\dl y)
=
\int_E\varepsilon_y^-\varepsilon_x^+\Phi(y)\,N(\dl y).\qedhere
\end{align*}
\end{proof}

Our next proposition generalizes 
the related results \cite[Theorem~4.1]{Nualart1990} and \cite[Corollaire~4]{Picard1996a} in several ways. An analogue in the Gaussian case is \cite[Proposition 1.3.1]{Nualart2006}.

\begin{prop}[$L^2$-isometry]\label{prop:L2D12_subset_domdelta}
The space $L^2(\X;\bD^{1,2}(H))\subset L^2(\Omega\times E;H)$ 
is contained in $\dom(\delta^{(2)})$. The restriction $\delta|_{L^2(\X;\bD^{1,2}(H))}$ is continuous as a mapping from $L^2(\X;\bD^{1,2}(H))$ to $L^2(\Omega;H)$, and for all $\Phi,\Psi\in L^2(\X;\bD^{1,2}(H))$ we have
\begin{equation}\label{eq:L2D12_subset_domdelta_1}
\begin{aligned}
\E\big\langle\delta(\Phi),\delta(\Psi)\big\rangle
&=
\E\int_\X\langle\Phi(x),\Psi(x)\rangle\,\m(\diffin x)\\
&\quad+
\E\int_\X\int_\X\big\langle D_x\Phi(y),D_y\Psi(x)\big\rangle\,\m(\diffin y)\,\m(\diffin x).
\end{aligned}
\end{equation}
\end{prop}

\begin{proof}
Let us first consider $\Phi,\Psi\in L^2(\X;\bD^{1,2}(H))\subset L^2(\Omega\times E;H)$ of the form 
$\Phi=\sum_{k=1}^K\one_{\Omega\times A_k} F_k$, $\Psi=\sum_{k=1}^K\one_{\Omega\times B_k}\,G_k $
with $F_k,G_k\in \bD^{1,2}(H)\cap L^\infty(\Omega;H)$ and $A_k,B_k\in\cE_0$, $K\in\bN$. In particular, for all $p\in[1,\infty]$ it holds that $\Phi\in L^p(\Omega\times\X;H)$ and, for $\m$-almost all $x\in E$, $D_x\Phi=\sum_{k=1}^K\one_{\Omega\times A_k}D_xF_k\in L^p(\Omega\times\X;H)$. Thus Lemma~\ref{lem:commutator1} implies that $D_x\delta(\Phi)=\delta(D_x\Phi)+\Phi(x)$ $\P\otimes \m$-almost everywhere. Note that 
$\bE\int_B\|D_x\delta(\Phi)\|^p\m(\dl x)<\infty$ 
for all $p\in[1,\infty)$ due to Lemmata~\ref{lem:local_estimate_D}  and \ref{lem:local_estimate_delta}; as a consequence we also have 
$\bE\int_B\|\delta(D_x\Phi)\|^p\m(\dl x)<\infty$ 
for all $p\in[1,\infty)$. Thus we can write
\begin{align}\label{eq:proof_isometry_1}
\E\int_\X\big\langle D_x\delta(\Phi),\Psi(x)\big\rangle\,\m(\diffin x)
=
\E\int_\X\big\langle\delta(D_x\Phi)+\Phi(x),\Psi(x)\big\rangle\,\m(\diffin x)
\end{align}
and all integrals are defined.
Since $\Psi\in L^p(\Omega\times\X;H)$ for all $p\in[1,\infty]$ and $\delta(\Phi)\in L^{q'}(\Omega;H)$ for all $q'\in[1,\infty)$ by Lemma~\ref{lem:local_estimate_delta}, 
we can apply Lemma~\ref{lemma:duality2} to obtain
\begin{align}\label{eq:proof_isometry_2}
\E\int_\X\big\langle D_x\delta(\Phi),\Psi(x)\big\rangle\,\m(\diffin x)=\bE\big\langle\delta(\Phi),\delta(\Psi)\big\rangle.
\end{align}
Similarly, since for $\m$-almost all $x\in E$ we have $D_x\Phi\in L^p(\Omega\times E;H)$ for all $p\in[1,\infty]$ and $\Psi(x)\in L^{q'}(\Omega;H)$ for all $q'\in[1,\infty]$, a further application of Lemma~\ref{lemma:duality2} yields
\begin{align}\label{eq:proof_isometry_3}
\bE\big\langle\delta(D_x\Phi),\Psi(x)\big\rangle=\bE\int_E\big\langle D_x\Phi(y),D_y\Psi(x)\big\rangle\,\m(\dl y)
\end{align}
for $\m$-almost all $x\in E$. The combination of \eqref{eq:proof_isometry_1}, \eqref{eq:proof_isometry_2} and \eqref{eq:proof_isometry_3} gives \eqref{eq:L2D12_subset_domdelta_1} for $\Phi$ and $\Psi$ of the considered elementary form.
Choosing $\Phi=\Psi$ and applying the Cauchy-Schwarz inequality, we obtain
\begin{align*}
\|\delta(\Phi)\|_{L^2(\Omega;H)}^2\
\leq
\|\Phi\|^2_{L^2(\Omega\times\X;H)}
+
\|D\Phi\|^2_{L^2(\Omega\times\X^2;H)}
=
\|\Phi\|^2_{L^2(\X;\bD^{1,2}(H))}.
\end{align*}
By the approximation result in Lemma~\ref{lem:coreD12}(i) and the closedness of $\delta=\delta^{(2)}$ this estimate extends to all $\Phi\in L^2(\X;\bD^{1,2}(H))$. In particular we have $L^2(\X;\bD^{1,2}(H))\subset \dom(\delta)$.
In the same way, the identity \eqref{eq:L2D12_subset_domdelta_1} extends to arbitrary $\Phi$ and $\Psi$ in $L^2(\X;\bD^{1,2}(H))$.
\end{proof}

We are now ready to state and prove the commutation relation in $L^2$.

\begin{prop}[Commutation relation -- $L^2$-version]
\label{prop:commutator_relation_L^2}
Let $\Phi\in L^2(\X;\bD^{1,2}(H))$ $\subset L^2(\Omega\times E;H)$ 
be such that $D_x\Phi\in\dom(\delta^{(2)})$ for $\m$-almost all $x\in\X$ and 
assume that
$\bE\int_E\|\delta(D_x\Phi)\|^2\,\m(\dl x)<\infty$.
Then $\delta(\Phi)\in\bD^{1,2}(H)$ and
\begin{align*}
  D_x\delta(\Phi)=\delta(D_x\Phi)+\Phi(x)
\end{align*}
as an equality in $L^2(\Omega\times E;H)$.
\end{prop}

\begin{proof}
Due to Proposition~\ref{prop:L2D12_subset_domdelta}, the assumptions on $\Phi$ and the duality relation between $\delta=\delta^{(2)}$ and $D|_{\bD^{1,2}(H)}$ we have
\begin{align*}
\E\big\langle\delta(\Phi),\delta(\Psi)\big\rangle
=
\E\int_\X\big\langle\Phi(x)+\delta(D_x\Phi),\Psi(x)\big\rangle\,\m(\diffin x)
\end{align*}
for all $\Psi\in L^2(\X;\bD^{1,2}(H))$. 
On the other hand, taking $\Psi=\sum_{k=1}^K\one_{\Omega\times B_k}G_k$ with $G_k\in\bD^{1,2}(H)\cap L^\infty(\Omega;H)$ and $B_k\in\cE_0$, $K\in\bN$, we have $\Psi\in L^{p}(\Omega\times E;H)$ for all $p\in[1,\infty]$ so that the local duality formula in Lemma~\ref{lemma:duality2} with $q=q'=2$ implies 
\begin{align*}
\E\big\langle\delta(\Phi),\delta(\Psi)\big\rangle
=
\E\int_{\bigcup_{k=1}^KB_k}\big\langle D_x\delta(\Phi),\Psi(x)\big\rangle\,\m(\dl x).
\end{align*}
Note that $\bE\int_C\|D_x\delta(\Phi)\|^{p'}\,\m(\dl x)<\infty$ for all $C\in\cE_0$ and $p'\in[1,2)$ due to Lemma~\ref{lem:local_estimate_D}.
By Lemma~\ref{lem:coreD12}(ii), for all $C\in \cE_0$ and all $p>1$ the linear span of $\big\{\one_{\Omega\times B}\,G:G\in \bD^{1,p}(H)\cap\bD^{1,2}(H)\cap L^\infty(\Omega;H),\,B\in\cE_0,\,B\subset C\big\}$ is dense in 
$\{\Psi\in L^p(\Omega\times\X;H):\Psi=0\text{ on }\Omega\times C^c\}$ w.r.t.\ the norm in $L^p(\Omega\times\X;H)$.
We conclude that $\Phi(x)+\delta(D_x\Phi)=D_x\delta(\Phi)$ $\P\otimes \m$-almost everywhere in $\Omega\times C$. As $(E,\cE,\m)$ is $\sigma$-finite, we have $\Phi(x)+\delta(D_x\Phi)=D_x\delta(\Phi)$ $\P\otimes \m$-almost everywhere in $\Omega\times E$. In particular this implies $D\delta(\Phi)\in L^2(\Omega\times\X;H)$ and thus $\delta(\Phi)\in\bD^{1,2}(H)$.
\end{proof}
\color{black}

\subsection{Higher order calculus}
\label{subsec:higher_order}

Here we extend some of the results of the previous sections to higher order difference operators and multiple Kabanov-Skorohod integrals. This is mainly done by elementary induction arguments. The structure of this section follows that of the previous ones.
Let us introduce some suitable notation:
If $k\in\bN$, $\mathbf{x}=(x_1,\dots,x_k)\in \X^k$ and $I=\{i_1,\dots,i_{|I|}\}\subset\{1,\dots,k\}$ with $i_1<\dots < i_{|I|}$, then the vector $\mathbf{x}_I\in \X^{|I|}$ is given by $\mathbf{x}_I=(x_{i_1},\dots,x_{i_{|I|}})$. 
We also set $[k]:=\{1,\ldots,k\}$.
Moreover, for $\mathbf x=(x_1,\dots,x_k) \in \X^k$ we denote by $\varepsilon_{\mathbf x}^{+}$ the composition 
$\varepsilon_{\mathbf x}^+:=\varepsilon_{x_k}^+\circ\dots\circ\varepsilon_{x_1}^+$
whenever it is meaningful, compare Remark~\ref{rem:epsilon_for_random_fields}.
\color{black}

\subsubsection*{Higher order difference operators and Sobolev spaces}

We start by defining the action of higher order difference operators on random variables.

\begin{definition}[Higher order difference operators]\label{def:higherDiffOp}
Let the operators 
\begin{align*}
  &D^k\colon L^0(\Omega;H)\rightarrow L^0(\Omega\times E^{k};H)
  ,\quad k\in\bN,
\end{align*}
be defined 
by iteration of $D$. 
That is, we iteratively set $D^k:=D\circ D^{k-1}$, where $D$ is understood as an operator from $L^0(\Omega\times E^{k-1};H)$ to $L^0(\Omega\times E^k;H)=L^0(\Omega\times E^{k-1}\times E;H))$, compare Remark~\ref{rem:D_for_random_fields}.
For $\mathbf{x}=(x_1,\dots,x_k)$ we write 
$D_{\mathbf{x}}^k F$ instead of $(D^kF)(x_1,\dots,x_k)$.
\end{definition}

Analogously to what has been said in Remark~\ref{rem:D_for_random_fields} we may extend the operators $D^k$ to random fields and consider, e.g., $D^k\colon L^0(\Omega\times E^n;H)\to L^0(\Omega\times E^{n+k};H)$ for $n\in\bN$.

Similar to the first order situation we derive elementary non-recursive identities.
\begin{prop}\label{lem:Dn_explicit}
Let $F\in L^0(\Omega;H)$, $k\in\bN$, and  $h$ be a measurable function from $H$ to another (real and separable) Hilbert space $V$. Then, $\bP\otimes\m^{\otimes k}$-almost everywhere,
\begin{align*}
  D_{\mathbf{x}}^k
  F
  &=
  \sum_{I\subset[k]}
  (-1)^{k-|I|}
  \varepsilon_{\mathbf{x}_{I}}^+F
  \quad\textrm{and}\quad
  D_{\mathbf{x}}^k
  h(F)
  =
  \sum_{I\subset [k]}
  (-1)^{k-|I|}h
  \Bigg(
    \sum_{J\subset I}
    D^{|J|}_{\mathbf{x}_J}F
  \Bigg).
\end{align*}
\end{prop}

\begin{proof}
For $k=1$ the first statement is true by the definition of $D$ and since $\emptyset\subset\{1\}$. In order to perform an induction assume that the first statement holds for $k-1\in\N$. Writing $\tilde{\mathbf x}:=(x_1,\dots,x_{k-1})$ we have
\begin{align*}
  D^k_{\mathbf{x}}F
&=
  D_{x_k}D^{k-1}_{\tilde{\mathbf x}}F
  =
  D_{x_k}
  \sum_{I\subset[k-1]}
  (-1)^{k-1-|I|}
  \varepsilon_{\tilde{\mathbf x}_{I}}^+F\\
&=
  \sum_{I\subset[k-1]}
  (-1)^{k-1-|I|}
  \varepsilon_{\mathbf{x}_{I\cup\{k\}}}^+F
  +
  \sum_{I\subset[k-1]}
  (-1)^{k-|I|}
  \varepsilon_{\mathbf{x}_{I}}^+F=
  \sum_{I\subset[k]}
  (-1)^{k-|I|}
  \varepsilon_{\mathbf{x}_{I}}^+F.  
\end{align*}
By the induction principle the statement is valid for all $k\in \N$. The second identity admits a similar proof.
\felix{see notes 27.1.'17}
\end{proof}

We now introduce the higher order Sobolev spaces.

\begin{definition}
For $k\in\N$ and $p> 1$ we define $\bD^{k,p}(H)$ as the space of all random variables $F\in L^p(\Omega;H)$ such that $D^j F\in L^p(\Omega\times\X^j;H)$ for all $j\in\{1,\ldots,k\}$. It is equipped with the norm 
\begin{align*}
  \|X\|_{\bD^{k,p}(H)}
  =
  \Bigg(
    \|X\|_{L^p(\Omega;H)}^p
    +
    \sum_{j=1}^k
    \|D^j X\|_{L^p(\Omega\times\X^j;H)}^p
  \Bigg)^\frac1p.
\end{align*}
\end{definition}

\begin{prop}
For all $k\in\N$ and $p>1$ the space $\bD^{k,p}(H)$ is complete, i.e.\ a Banach space. In particular, the restriction of $D^k\colon L^0(\Omega;H)\to L^0(\Omega\times\X^k;H)$ to $\bD^{k,p}(H)$ is closed from $L^p(\Omega;H)$ to $L^p(\Omega\times\X^k;H)$. The space $\bD^{k,2}(H)$ is a Hilbert space.
\end{prop}

\begin{proof}
The continuity of the operators $D\colon L^0(E^{j-1}\!;L^p(\Omega;H))\to L^0(E^j;L^q(\Omega;H))$, $j=1,\ldots,k$, $q\in[1,p)$, stated in Corollary~\ref{cor:localizing_lemma_1} and Remark~\ref{rem:D_for_random_fields}, implies that $D^k$ is continuous from $L^p(\Omega;H)$ to $L^0(E^k;L^q(\Omega;H))$ for all $q\in[1,p)$.
The proof is completed analogously to proof of Proposition~\ref{thm:Banach}.
\end{proof}

\subsubsection*{Multiple Kabanov-Skorohod integrals and duality}

Next we treat multiple Kabanov-Skorohod integrals. 
We begin with a generalization of the pathwise $L^1$-integral introduced in Definition~\ref{def:pathwise_Kabanov}.

\begin{definition}[Pathwise multiple Kabanov-Skorohod integral] 
\label{def:Kabanov-Skorohod_multiple_pathwise}
The operators
\[
  \delta^k
  \colon 
  L^1(\Omega\times E^k;H)
  \to 
  L^1(\Omega;H)
  ,\quad
  k\in\bN,
\] 
are defined by iteration of the pathwise Kabanov-Skorohod integral. That is, we iteratively set
$\delta^k:=\delta^{k-1}\circ\delta$, where $\delta$ is understood as an operator from $L^1(\Omega\times E^k;H)=L^1(\Omega\times E^{k-1}\times E;H)$ to $L^1(\Omega\times E^{k-1};H)$, compare Remark~\ref{rem:Kabanov-Skorohod_2}.
\end{definition}

Analogously to what has been said in Remark~\ref{rem:Kabanov-Skorohod_2} we may extend the operators $\delta^k$ to parameter-dependent random fields and consider, e.g., the mapping $\delta^k\colon L^0(E^n;L^1(\Omega\times E^k;H))\to L^0(E^n;L^1(\Omega;H))$ for $n\in\bN$.

In order to derive an explicit representation formula for multiple Kabanov-Skorohod integrals we need to make some preparations. The $k$-th factorial measure of $N$ is the $\mathbf N(E^k)$-valued random variable $N^{(k)}$ given by
\begin{align*}
N^{(k)}(\omega):=\sum_{\substack{1\leq n_1,\ldots,n_k\leq N(\omega,E)\\n_i\neq n_j\text{ for }i\neq j}}\delta_{(X_{n_1}(\omega),\ldots,X_{n_k}(\omega))},\quad\omega\in\Omega.
\end{align*}
Here $\mathbf N(E^k)$ denotes the space of all $\sigma$-finite $\mathds N_0\cup\{\infty\}$-valued measures on $(E^k,\cE^{\otimes k})$ and $\delta_{(X_{n_1}(\omega),\ldots,X_{n_k}(\omega))}$ is the Dirac measure at $(X_{n_1}(\omega),\ldots,X_{n_k}(\omega))\in E^k$.
The following multivariate version of Mecke's formula is well known, see \cite[Theorem~4.5]{LastPenrose2016}:
For measurable functions $f\colon \mathbf N\times E^k\to [0,\infty]$ we have
\begin{align}\label{eq:Mecke_multivariat}
\bE\int_{E^k}f(N\setminus\delta_{x_1}\setminus\ldots\setminus\delta_{x_k},\mathbf x)\,N^{(k)}(\dl\mathbf x)=\bE\int_{E^k}f(N,\mathbf x)\,\m^{\otimes k}(\dl\mathbf x),
\end{align}
where we write $\mathbf x=(x_1,\ldots,x_k)\in E^k$. Hereby the measurability of the mapping $\Omega\times E^k\ni(\omega,\mathbf x)\mapsto N(\omega)\setminus\delta_{x_1}\setminus\ldots\setminus\delta_{x_k}\in\mathbf N(E)$ appearing in the integrand on the left hand side can be checked analogously as in Remark~\ref{rem:measurability}, using the representation \eqref{eq:representationN} of $N$ and the assumtion that the diagonal in $E^2$ is contained in $\cE^{\otimes 2}$. 
\felix{see notes 27.1.'17}
Given $F\in L^0(\Omega\times E^k;H)$ and a measurable mapping $f\colon
\times E^k\to H$ such that $f(N,\cdot)$ is a version of $F$ we set
$
\varepsilon_\mathbf x^-F(\mathbf x):=f(N\setminus\delta_{x_1}\setminus\ldots\setminus\delta_{x_k},\mathbf x).
$
By \eqref{eq:Mecke_multivariat} this definition is $\bP\otimes N^{(k)}$-almost everywhere independent of choice of representative $f$, compare Lemma~\ref{lem:epsilon-welldefined}.
\felix{see notes 27.1.'17}
We are now able to state the following Hilbert space-valued multivariate Mecke formula. Its proof is completely analogous to that of Proposition~\ref{lem:stoch_int} and therefore omitted.

\begin{prop}[$H$-valued multivariate Mecke formula]\label{lem:stoch_int_multiple}
The integral mapping
\begin{align*}
  L^1(\Omega\times \X^k;H)
  \ni
  \Phi
  \mapsto
  \int_{\X^k}
    \varepsilon_\mathbf{x}^-\Phi(\mathbf x)
  N^{(k)}(\diffin{\mathbf x})
  \in
  L^1(\Omega;H),
\end{align*}
where $\int_{\X^k}
    \varepsilon_\mathbf{x}^-\Phi(\mathbf x)
  N^{(k)}(\diffin{\mathbf x})$ is $\bP$-almost surely defined as an $H$-valued Bochner integral, is well-defined and continuous.
For all $\Phi\in L^1(\Omega\times E^k;H)$ it holds that
\begin{align}
\label{eq:bound_stoch_int_multi}
  \Big\|
  \int_{\X^k}
    \varepsilon_\mathbf{x}^-\Phi(\mathbf x)
  N^{(k)}(\diffin{\mathbf x})  
  \Big\|_{L^1(\Omega;H)}
  \leq
  \|\Phi\|_{L^1(\Omega\times\X^k;H)},
\end{align}
and
\begin{align*}
  \E
  \Big[
  \int_{\X^k}
    \varepsilon_\mathbf{x}^-\Phi(\mathbf x)
  N^{(k)}(\diffin{\mathbf x})
  \Big]
  =
  \E
  \Big[
    \int_{\X^k}
      \Phi(\mathbf x)
    \,\m^{\otimes k}(\diffin{\mathbf x})
  \Big].
\end{align*}
\end{prop}

With the integral mapping of Proposition~\ref{lem:stoch_int_multiple} at hand we can prove a non-recursive formula for $\delta^k$. As in Subsection~\ref{sec:commutation} our argumentation requires the additional assumption that $(E,\cE)$ is a Borel-space.

\begin{prop}\label{lem:deltak}
In addition to Setting \ref{setting} assume that $(E,\cE)$ is a Borel space. For all $\Phi\in L^1(\Omega\times\X^k;H)$, $k\in\N$, it holds that 
\begin{align*}
  \delta^k(\Phi)
  =
  \sum_{I\subset[k]}
  (-1)^{k-|I|}
  \int_{\X^{|I^c|}}
  \int_{\X^{|I|}}
    \varepsilon_{\mathbf{x}_I}^-\Phi(\mathbf{x})
  \,N^{(|I|)}(\diffin{\mathbf{x}_I})
  \,\m^{\otimes |I^c|}(\diffin{\mathbf{x}_{I^c}}),
\end{align*}
and the order of integration has no importance.
\end{prop}

\begin{proof}
The case $k=1$ holds by the definition of $\delta$. We refrain from presenting the general induction argument but instead consider the case $k=2$. We have
\begin{equation*}\label{eq:deltak}
\begin{aligned}
&\delta^2(\Phi)
  =
  \delta
  \Big(
    \int_\X
      \varepsilon_{x_1}^-
      \Phi(x_1,\cdot)
    N(\diffin x_1)
    -
    \int_\X
      \Phi(x_1,\cdot)
    \m(\diffin x_1)
  \Big)\\
&\quad
  =
  \int_\X
  \varepsilon_{x_2}^-
  \Big(
    \int_\X
      \varepsilon_{x_1}^-
      \Phi(x_1,x_2)
    N(\diffin x_1)
  \Big)
  N(\diffin x_2)
  -
  \int_\X
    \int_\X
      \varepsilon_{x_1}^-
      \Phi(x_1,x_2)
    N(\diffin x_1)  
  \m(\diffin{x_2})\\
&\qquad
  -
  \int_{\X}
    \varepsilon_{x_2}^-
    \Big(
      \int_\X
        \Phi(x_1,x_2)
      \m(\diffin x_1)
  \Big)
  N(\diffin x_2)  
  +
  \int_{\X}
      \int_\X
        \Phi(x_1,x_2)
      \m(\diffin x_1)
  \m(\diffin x_2).
\end{aligned}
\end{equation*}
Looking at the first integral we use Lemma~\ref{lem:commutator} 
and argue analogously to the proof of Proposition~\ref{lem:commutator1} to see that it equals
\begin{align*}
  \int_\X
  \int_\X
  \varepsilon_{x_2}^-
  \varepsilon_{x_1}^-
      \Phi(x_1,x_2)
  \big(
    N
    \setminus
    \delta_{x_2}
  \big)
  (\diffin x_1)
  N(\diffin x_2)
  =
  \int_{\X^2}
    \varepsilon_{\mathbf x}^-\Phi(\mathbf x) N^{(2)}(\diffin \mathbf x ).
\end{align*}
\felix{see notes 27.1.'17}
Similarly, the third integral equals
$\int_E\int_E\varepsilon_{x_2}^-\Phi(x_1,x_2)\m(\dl x_1)N(\dl x_2)$.
Finally, the assumption $\Phi\in L^1(\Omega\times E^2;H)$, Mecke's formula \eqref{eq:Mecke2}, and pathwise applications of Fubini's theorem allow to exchange the order of integration in the mixed iterated integrals w.r.t.\ $N$ and $\m$ $\bP$-almost everywhere.
\felix{see notes 27.1.'17}
\end{proof}

\begin{prop}[Higher order $L^1$-$L^\infty$-duality]\label{lem:higher_duality}
For all $F\in L^\infty(\Omega;H)$ and $\Phi\in L^1(\Omega\times\X^k;H)$, $k\in\N$, it holds that
\begin{align*}
  \tensor[_{L^\infty(\Omega;H)}]{\big\langle
    F, \delta^k(\Phi)
  \big\rangle}{_{L^1(\Omega;H)}}
  =
  \tensor[_{L^\infty(\Omega\times\X^k;H)}]{\big\langle
    D^kF,\Phi
  \big\rangle}{_{L^1(\Omega\times\X^k;H)}}.
\end{align*}
\end{prop}

\begin{proof}
  This follows from Proposition~\ref{lem:duality} via Fubini's theorem and induction.
\end{proof}

Analogously to the introduction of the operators $\delta^{(p)}$ in Definition~\ref{def:Kabanov-Skorohod} we define for $k\in\bN$ and $p>1$ realizations $\delta^{k,(p)}$ of $\delta^k$ in $L^p$.
Note that $\bD^{k,p'}(H)$ is dense in $L^{p'}(\Omega;H)$ for all $k\in\bN$, $p'>1$; this can be seen similarly as in the proof of Lemma~\ref{lem:coreD12}(ii).


\begin{definition}[Abstract multiple Kabanov-Skorohod integral]
\label{def:Kabanov-Skorohod_multiple}
Let $p,p'>1$ be sucht that $\frac1p+\frac1{p'}=1$. For $k\in\N$ the operator
\[\delta^{k,(p)}\colon \,\dom(\delta^{k,(p)})\subset L^{p}(\Omega\times\X^k;H)\to L^p(\Omega;H) \]
is defined as the adjoint of 
\begin{align*}
D^k|_{\bD^{k,p'}(H)}\colon\bD^{k,p'}(H)\subset L^{p'}(\Omega;H) \to L^{p'}(\Omega\times\X^k;H),
\end{align*}
i.e., of the restriction of $D^k$ to $\bD^{k,p'}(H)$, considered as a densely defined operator from $L^{p'}(\Omega;H)$ to $L^{p'}(\Omega\times\X^k;H)$.
\end{definition}

As in the first order case we verify the compatibility of Definition~\ref{def:Kabanov-Skorohod_multiple_pathwise} and~\ref{def:Kabanov-Skorohod_multiple}. To this end we first compare $\delta^{k,(p)}$ to the $k$-fold iteration of $\delta^{(p)}$. The latter is the operator 
$
(\delta^{(p)})^k\colon\dom\big((\delta^{(p)})^k\big)\subset L^p(\Omega\times E^k;H)\to L^p(\Omega;H)
$
iteratively defined by $(\delta^{(p)})^k:=(\delta^{(p)})^{k-1}\circ\delta^{(p)}$, where $\delta^{(p)}$ is understood as a mapping from $L^p(E^{k-1};\dom(\delta^{(p)}))\subset L^p(\Omega\times E^k;H)$ to $L^p(\Omega\times E^{k-1};H)$, compare Remark~\ref{rem:Kabanov-Skorohod_2}. Here $\dom\big((\delta^{(p)})^k\big)$ is iteratively defined as the space of all $\Phi\in L^p(\Omega\times E^k;H)$ such that $\Phi(\tilde{\mathbf x},\cdot)\in\dom(\delta^{(p)})$ for $\m^{\otimes(k-1)}$-almost all $\tilde{\mathbf x}\in E^{k-1}$ and such that $\big(\delta^{(p)}(\Phi(\tilde{\mathbf x},\cdot))\big)_{\tilde{\mathbf x}\in E^{k-1}}\in\dom\big((\delta^{(p)})^{k-1}\big)$.

\begin{lemma}\label{lem:coincides_multi}
For all $k\in\bN$, $p>1$ we have $\dom\big((\delta^{(p)})^{k}\big)\subset\dom(\delta^{k,(p)})$ and $(\delta^{(p)})^{k}=\delta^{k,(p)}$ on $\dom\big((\delta^{(p)})^{k}\big)$.
\end{lemma}

\begin{proof}\felix{see notes 4.2.'17--A--}
The assertion follows directly by induction over $k$, using Fubini's theorem and the duality between $D|_{\bD^{1,p'}(H)}$ and $\delta^{(p)}$.
\end{proof}

\begin{prop}\label{prop:coincides_multi}
Let $k\in\bN$. The operators $\delta^k$, $\delta^{k,(p)}$, $p>1$, coincide on the intersections of their domains. More precisely:
\begin{enumerate}[(i)]
\item If $\Phi\in L^1(\Omega\times E^k;H)\cap L^p(\Omega\times E^k;H)$ for some $p>1$ and if for all\linebreak $j\in\{1,\ldots,k\}$ it holds that $\delta^j(\Phi(\tilde{\mathbf x},\cdot))\in L^p(\Omega;H)$ for $\m^{\otimes(k-j)}$-almost all $\tilde{\mathbf x}\in E^{k-j}$, then $\Phi\in\dom(\delta^{k,(p)})$ and $\delta^k(\Phi)=\delta^{k,(p)}(\Phi)$.
\item If $\Phi\in \dom(\delta^{k,(p)})\cap\dom(\delta^{k,(q)})$ for some $p,q>1$, then $\delta^{k,(p)}(\Phi)=\delta^{k,(q)}(\Phi)$.
\end{enumerate}
\end{prop}

\begin{proof}
\felix{see notes 4.2.'17--A-- u. 4.2.'17--B--}
Let $p':=\frac{p}{p-1}$, $q':=\frac{q}{q-1}$ be the conjugate exponents to $p,q>1$. We prove the assertion (i): For $k=1$ it is trivially true. In order to perform an induction over $k$, let $k\geq2$ and assume that the assertion holds for $k-1$.
We have
\begin{align*}
\delta^k(\Phi)
&=
\delta^{k-1}\circ\delta(\Phi)
=
\delta^{k-1}\Big(\big(\delta(\Phi(\tilde{\mathbf x},\cdot)\big)_{\tilde{\mathbf x}\in E^{k-1}}\Big)
=
\delta^{k-1}\Big(\big(\delta^{(p)}(\Phi(\tilde{\mathbf x},\cdot)\big)_{\tilde{\mathbf x}\in E^{k-1}}\Big)\\
&=
\delta^{k-1,(p)}\Big(\big(\delta^{(p)}(\Phi(\tilde{\mathbf x},\cdot)\big)_{\tilde{\mathbf x}\in E^{k-1}}\Big)
=
(\delta^{(p)})^k(\Phi)
=
\delta^{k,(p)}(\Phi).
\end{align*}
Here the third equality follows Proposition~\ref{prop:coincides}(i) and the assumptions of $\Phi$, the fourth equality is due to the induction hypothesis, and the last two equalities are consequences of Lemma~\ref{lem:coincides_multi} in combination with Proposition~\ref{prop:coincides}(i). In particular we obtain that $\Phi\in\dom(\delta^{k,(p)})$.
%
Assertion (ii) can be verified in complete analogy to the proof of Proposition~\ref{prop:coincides}(ii).
\end{proof}

Note that Proposition~\ref{prop:coincides_multi} allows us to simplify notation by setting $\delta^k(\Phi):=\delta^{k,(p)}(\Phi)$ for $\Phi\in\dom(\delta^{k,(p)})$, $p>1$

%

\begin{remark}[Higher order commutation relations]
In addition to Setting~\ref{setting} assume that $(E,\cE)$ is a Borel space.
Let  $1\leq k\leq \ell<\infty$ and assume that $\Phi\in L^1(\Omega\times\X^k;H)$ or $\Phi\in L^2(\X^k;\bD^{k,2}(H))$. In the former case the multiple Kabanov-Skorohod integral $\delta^k(\Phi)$ is defined in the pathwise sense by Definition~\ref{def:Kabanov-Skorohod_multiple_pathwise}; in the latter case it follows from Lemma~\ref{lem:coincides_multi}, Proposition~\ref{prop:L2D12_subset_domdelta}, Proposition~\ref{prop:commutator_relation_L^2} and induction over $k$ that $\Phi\in \dom(\delta^{k,(2)})$. 
\felix{see notes 6.2.'17}
Under suitable additional technical assumptions on $\Phi$, similar to those in Proposition~\ref{lem:commutator1} and Proposition~\ref{prop:commutator_relation_L^2}, by a technical induction argument, it holds that
\begin{align}\label{eq:hocr}
  D^\ell_{\mathbf{x}}
  \delta^k(\Phi)
  =
  \delta^k(D^\ell_{\mathbf{x}}\Phi)
  +
  \sum_{i=1}^k
  \sum_{\substack{I\subset[k]\\|I|=i}}
  \sum_{\substack{J\subset[\ell]\\|J|=i}}
  \delta^{k-i}
  \big(
    D_{\mathbf{x}_{J^c}}^{\ell-i}
    \Phi_I
    (\mathbf{x}_J,\cdot)
  \big).
\end{align}
Here we denote for $I\subset[k]=\{1,\ldots,k\}$ and $\tilde{\mathbf x}\in E^{|I|}$ by $\Phi_I(\tilde{\mathbf x},\cdot)$ the random function on $\Omega\times E^{k-|I|}$ obtained by evaluating $\Phi$ in the coordinates corresponding to $I$ at $\tilde{\mathbf x}$ and letting the other variables be undetermined.
Since we have no immediate application in mind for this formula we refrain from further details.
%
\end{remark}


\section{A space-time setting}
\label{sec:space_time}

In this section we assume that the underlying measure space in Setting~\ref{setting} is of the special form
\begin{align}\label{eq:E_space-time}
(E,\cE,\mu)=\big([0,T]\times U,\;\cB([0,T]\times U),\;\lambda\otimes\nu\big),
\end{align}
where $T\in(0,\infty)$, $(U,\|\cdot\|_U,\langle\cdot,\cdot\rangle_U)$ is a separable real Hilbert space, $\cB(\ldots)$ denotes the Borel-$\sigma$-algebra, $\lambda$ denotes Lebesgue measure on $([0,T],\cB([0,T]))$, and
$\nu$ is a $\sigma$-finite measure on $(U,\cB(U))$. In applications $\nu$ is typically a Lévy measure. The Poisson Malliavin calculus from Section~\ref{sec:Malliavin} is perfectly valid in this special case. 
In Subsections~\ref{subsec:Malliavin_operators_space-time} and \ref{subsec:duality_comm_space-time} below we complement the general theory from Section~\ref{sec:Malliavin} by a series of results which are specifically adapted to the space-time structure \eqref{eq:E_space-time} and particularly relevant for the analysis of evolutionary SPDE with Lévy noise. In Subsection~\ref{subsec:PRMandLP} we describe the connection of our framework to Hilbert space-valued Lévy processes.
We start by introducing some additional notation.  
\begin{notation}
\label{notation_space-time}
In the sequel, we use the following notation:
\begin{itemize}
\item
$(\cF_t)_{t\in[0,T]}$ denotes the filtration given by 
$\cF_t:=\bigcap_{u\in(t,T]}\cF_{[0,u]\times U},$
where $\cF_{[0,u]\times U}$ is defined according to \eqref{eq:defFA} as the $\bP$-completion of the $\sigma$-algebra generated by the random variables $N\big(([0,u]\times U)\cap B\big)$, $B\in\cB([0,T]\times U)$.
\item
We set $\Omega_T:=\Omega\times [0,T]$, $\bP_T:=\bP\otimes\lambda$ and $\cP_T\subset\cF\otimes\cB([0,T])$  denotes the $\sigma$-algebra of predictable sets corresponding to $(\cF_t)_{t\in[0,T]}$, i.e.,
$
\cP_T:=\sigma\big(\{F_s\times (s,t]:0\leq s\leq t\leq T,\, F_s\in\cF_s\}\cup\{\{0\}\times F_0:F_0\in\cF_0\}\big).
$
\item  
We set
$
L^p_{\pred}(\Omega_T\times U; H):=L^p\big(\Omega_T\times U,\cP_T\otimes\cB(U),\P_T\otimes\nu;H\big)
$
for $p\in\{0\}\cup[1,\infty]$.
\item If $\Phi\in L^p_{\pred}(\Omega_T\times U; H)$ for some $p\in[1,2]$, then the stochastic integral of $\Phi$ w.r.t.~the compensated Poisson random measure $\tilde N$ is denoted by $I^{\tilde N}_t(\Phi)=\int_0^t\int_U\Phi(t,x)\,\tilde N(\dl t,\dl x)$, $t\in[0,T]$, cf.~Remark~\ref{rem:stoch_int_space-time} below.
\end{itemize}
\end{notation}
Note that the filtration $(\cF_t)_{t\in[0,T]}$ satisfies the `usual conditions' ($\bP$-completeness and right-continuity). This simplifies several argumentations in applications of our theory, e.g., when it comes to dealing with stopping times in the context of Lévy-driven SPDE.

\subsection{Malliavin operators in the space-time setting}
\label{subsec:Malliavin_operators_space-time} 

Here we collect some useful results concerning the difference operator $D$ and the Kabanov-Skorohod integral $\delta$ in the space-time setting determined by \eqref{eq:E_space-time}.

\subsubsection*{Difference operator in the space-time setting}
Note that by \eqref{eq:E_space-time} the Poisson ran-dom measure $N\colon \Omega\to\mathbf N$ from Setting~\ref{setting} now takes values in the space $\mathbf N=\mathbf N([0,T]\times U)$ 
of $\sigma$-finite $\bN_0\cup\{+\infty\}$-valued measures on $([0,T]\times U,\cB([0,T]\times U))$, endowed with the corresponding $\sigma$-algebra $\cN=\cN([0,T]\times U)$.
Accordingly, the difference operator  
$$
D\colon L^0(\Omega;H)\to L^0\big(\Omega\times([0,T]\times U);H\big)
$$
introduced in Definition~\ref{def:diffOp} now maps random variables $F$ to random functions 
$
DF=(D_{t,x}F)_{(t,x)\in[0,T]\times U}
$
on the space-time domain $[0,T]\times U$.

Our first result is an analogue of Lemma~\ref{lem:DF_F_FA_measurable} in the current setting.
 
\begin{corollary}\label{cor:DF_is_zero}
Let $t\in[0,T]$ and $F\in L^0(\Omega;H)$ be $\cF_t$-measurable. 
Then $DF=0$ $\bP_T\otimes \nu$-almost everywhere on $\Omega\times(t,T]\times U$.
\end{corollary}

\begin{proof}
As $F$ is $\cF_{[0,u]\times U}$-measurable for all $u\in(t,T]$, Lemma~\ref{lem:DF_F_FA_measurable} yields that $DF=0$ $\P_T\otimes \nu$-almost everywhere on $\Omega\times(u,T]\times U$ for all $u\in(t,T]$, which implies the assertion.
\end{proof}

Next we show two   auxiliary results concerning the Malliavin derivative of stochastic processes $X=(X(t))_{t\in[0,T]}$. The following lemma compares the derivative of $X(t)$ for fixed $t$ in the sense of Definition~\ref{def:diffOp} with the derivative of $X$ in the sense of Remark~\ref{rem:D_for_random_fields}.

\begin{lemma}[Malliavin derivative of stochastic processes]\label{lem:DX}
Let 
$X\colon \Omega_T\to H$  
be a stochastic process which is $\cF\otimes\cB([0,T])$-$\cB(H)$-measurable and stochastically continuous. For $t\in[0,T]$ let
\[D(X(t))
\in 
L^0\big(\Omega\times([0,T]\times U);\,H\big)\]
be the Malliavin derivative of $X(t)\in L^0(\Omega;H)$, in the sense of Definition~\ref{def:diffOp} with $F=X(t)$, and let
\[DX
\in 
L^0\big(\Omega_T\times([0,T]\times U);\,H\big)\]
be the Malliavin derivative of $X\in L^0(\Omega_T;H)$, in the sense of Remark~\ref{rem:D_for_random_fields} with  and $F=X$ and $(S,\cS,m)=([0,T],\cB([0,T]),\lambda)$. Then there exists a 
representative 
$(D_{s,x}X(t))_{t\in[0,T],(s,x)\in [0,T]\times U}$
of 
(the equivalence class of random functions) 
$DX$ such that
\[\text{for all }t\in[0,T]:\quad D_{s,x}(X(t))=D_{s,x}X(t)\text{ in }L^0(\Omega\times([0,T]\times U);H).
\]
%
\end{lemma}

\begin{remark}
The significance of Lemma~\ref{lem:DX} lies in the fact that the equality  $D_{s,x}(X(t))=D_{s,x}X(t)$ in $L^0(\Omega\times([0,T]\times U);H)$ holds for \emph{all} $t\in[0,T]$.
We already know from Remark~\ref{rem:D_for_random_fields} that it holds for \emph{almost all} $t\in[0,T]$.  
To show the stronger assertion in Lemma~\ref{lem:DX} we additionally assume the stochastic continuity of $X$. The latter is used to handle the fact that the factorization lemma from measure theory yields equalities only in an almost sure sense because the underlying $\sigma$-algebra $\cF$ is defined as the $\P$-completion of $\sigma(N)$ and not as $\sigma(N)$ itself.
\end{remark}


\begin{proof}[Proof of Lemma~\ref{lem:DX}]\felix{cf.~notes 24.9.'16}
The factorization lemma ensures for all $t\in[0,T]$ the existence of a $\cN$-$\cB(H)$-measurable mapping $\tilde f(\cdot,t)\colon \NN\to H$ such that $X(t)=\tilde f(N,t)$ $\P$-almost surely. This defines a stochastic process $\tilde f\colon\NN\times[0,T]\to H$ on the probability space $(\NN,\cN,\P_N)$. Clearly, the stochastic continuity of $X$ implies the stochastic continuity of $\tilde f$. Hence, by a standard Borel-Cantelli argument (see, e.g., \cite[Proposition~3.21]{PesZab2007}), there exists a $\cN\otimes\cB([0,T])$-$\cB(H)$-measurable process $f\colon\NN\times[0,T]\to H$ such that, for all $t\in[0,T]$, $\tilde f(\cdot,t)= f(\cdot,t)$ $\P_N$-almost surely. As a consequence, we have both $D_{s,x}X(t)=f(N+\delta_{s,x},t)-f(N,t)$ as an equality in $L^0(\Omega_T\times([0,T]\times U);H)$ and, for all $t\in[0,T]$, $D_{s,x}(X(t))=f(N+\delta_{s,x},t)- f(N,t)$ as an equality in $L^0(\Omega\times([0,T]\times U);H)$.
\end{proof}

Our next result concerns the Malliavin derivative of time-integrals of stochastic processes. An analogy for stochastic integrals is given in Propositions~\ref{prop:comm_space-time1} and \ref{prop:comm_space-time2} below.

\begin{prop}[Malliavin derivative of time integrals]\label{lem:intDX}
Let $X\colon \Omega_T\to H$ be a stochastic process which is $\cF\otimes\cB([0,T])$-$\cB(H)$-measurable and stochastically continuous. Let $\meas$ be a $\sigma$-finite Borel-measure on $[0,T]$ and assume that $X$ belongs to $L^1([0,T],\meas;L^p(\Omega;H))$ for some $p>1$. Then 
for all $B\in \cB(U)$ with $\nu(B)<\infty$
we have
$\bE\big[\int_{[0,T]}\int_B\int_{[0,T]}\|D_{s,x}X(t)\|\,\meas(\dl t)\,\nu(\dl x)\,\dl s\big]<\infty$,
so that the integral $\int_{[0,T]} D_{s,x}X(t)\,\meas(\dl t)$ is defined $\bP\otimes(\lambda\otimes\nu)$-almost everywhere on $\Omega\times([0,T]\times U)$ as an $H$-valued Bochner integral. Moreover,
the equality
\[D_{s,x}\int_{[0,T]} X(t)\,\meas(\dl t)=\int_{[0,T]} D_{s,x}X(t)\,\meas(\dl t)\]
holds in $L^0(U;L^1(\Omega_T;H))\subset L^0\big(\Omega\times([0,T]\times U);H\big)$.  
\end{prop}

\begin{proof}\felix{cf.~notes 24.9.'16}
The first assertion is a consequence of the local $L^q$-estimate from Lemma~\ref{lem:local_estimate_D} since
\begin{equation}\label{proof:intDX_1}
\begin{aligned}
\int_{[0,T]}&\Big[\bE\int_0^T\int_B\|D_{s,x}X(t)\|\,\nu(\dl x)\,\dl s\Big]\,\meas(\dl t)\\
&=\int_{[0,T]}\big\|\one_{\Omega\times([0,T]\times B)} DX(t)\big\|_{L^1(\Omega\times ([0,T]\times U);H)}\,\meas(\dl t)\\
&\leq C_{[0,T]\times B}\int_{[0,T]}
\|X(t)\|_{L^p(\Omega;H)}\,\meas(\dl t)<\infty.
\end{aligned}
\end{equation}
The second assertion follows directly from the equality in $L^0\big(\Omega\times([0,T]\times U);H\big)$
\begin{align}\label{proof:intDX_2}
\varepsilon_{s,x}^+\int_{[0,T]} X(t)\,\meas(\dl t)=\int_{[0,T]} \varepsilon_{s,x}^+X(t)\,\meas(\dl t).
\end{align}
To justify \eqref{proof:intDX_2} we first note that \eqref{proof:intDX_1} implies
\begin{align}\label{proof:intDX_3}
\bE\int_0^T\int_B\Big[\int_{[0,T]}\|\varepsilon_{s,x}^+ X(t)\|\,\meas(\dl t)\Big]\,\nu(\dl x)\,\dl s<\infty,
\end{align}
hence $\int_0^T\varepsilon_{s,x}^+X(t)\,\meas(\dl t)$ is defined 
$\bP\otimes(\lambda\otimes\nu)$-almost everywhere on $\Omega\times([0,T]\times U)$
as an $H$-valued Bochner integral.
Let $f\colon\NN\times[0,T]\to H$ be as in the proof of Lemma~\ref{lem:DX}, so that $X(t)=f(N,t)$ $\P$-a.s.\ for all $t\in[0,T]$. Define $g\colon\NN\to H$ by 
\[g(\eta):=
\begin{cases}
\int_{[0,T]}f(\eta,t)\,\meas(\dl t)& \text{ if }\int_{[0,T]}\|f(\eta,t)\|\,\meas(\dl t)<\infty\\
0& \text{ else}
\end{cases},\quad \eta\in\NN,
\]
so that $g$ is $\cN$-$\cB(H)$-measurable and $\int_{[0,T]} X(t)\,\meas(\dl t)=g(N)$ $\P$-a.s.. It follows that 
$
\varepsilon_{s,x}^+\int_{[0,T]}X(t)\,\meas(\dl t)=g(N+\delta_{s,x})=\int_{[0,T]}f(N+\delta_{s,x},t)\,\meas(\dl t),
$
where the second equality holds $\bP\otimes(\lambda\otimes\nu)$-almost everywhere due to \eqref{proof:intDX_3}. This completes the proof as $\int_{[0,T]}f(N+\delta_{s,x},t)\,\meas(\dl t)=\int_{[0,T]}\varepsilon_{s,x}^+ X(t)\,\meas(\dl t)$.
\end{proof}

\subsubsection*{Kabanov-Skorohod integral in the space time setting}
For predictable integrands $\Phi$, the $L^1$-integral $\int_0^T\int_U\varepsilon^-_{t,x}\Phi(t,x)\, N(\dl t,\dl x)$ from Proposition~\ref{lem:stoch_int} and the Kabanov-Skorohod integral $\delta(\Phi)$ introduced in Definition~\ref{def:pathwise_Kabanov}, \ref{def:Kabanov-Skorohod} and \eqref{eq:notation_Kabanov-Skorohod} coincide with Itô-type integrals of $\Phi$ w.r.t.~$N$ and $\tilde N$, respectively. 
This has already been indicated in Remark~\ref{rem:Kabanov-Skorohod_1} and is now stated more precisely.

To this end, we shortly review Hilbert space-valued stochastic integration w.r.t.\ $N$ and $\tilde N$ 
in Remark~\ref{rem:stoch_int_space-time} below.
We thereby follow the lines of \cite[Section~8.7]{PesZab2007}, where real-valued integrands are considered; the arguments therein carry over to the case of Hilbert space-valued integrands if one uses the Burkholder-Davis-Gundy inequality from \cite[Eq.~(1.3)]{brzezniak2009} or \cite[Theorem ~1.1]{Marinelli20161}. We refer to \cite{Dirksen2013,Knoche2005,Ruediger2004} for further results on vector-valued Poisson integration in a space-time setting.

\begin{remark}[Stochastic integration]\label{rem:stoch_int_space-time}
Note that $N$ is a Poisson random measure w.r.t. the filtration $(\cF_t)_{t\in[0,T]}$ in the sense that $N([0,t]\times B)$ is $\cF_t$-measurable and $N((t,t+h]\times B)$ is independent of $\cF_t$ for all $t\in[0,T]$, $h\in[0,T-t]$ and $B\in\cB(U)$;
\felix{see \cite[Remark~2.14]{Knoche2005}} 
in particular, $\big(\tilde N([0,t]\times B)\big)_{t\in[0,T]}$ is an $(\cF_t)$-martingale whenever $\nu(B)<\infty$.
The $L^1$ stochastic integral w.r.t.\ $N$ 
\felix{see notes 19.2.'16}
\begin{equation}\label{eq:I^N_t}
I^N_t\colon L^1_{\pred}(\Omega_T\times U; H)\to L^1(\Omega,\cF_t;H),\;\Phi\mapsto I^N_t(\Phi)
\end{equation}
is defined for all $t\in[0,T]$ by linear and continuous extension of the integral for simple integrands
$I^N_t(F\,\one_{(r,s]\times B}) :=F\, N((r,s\wedge t]\times B)$, where $0\leq r\leq s\leq T$, $F\colon\Omega\to H$ is bounded and $\cF_r$-measurable, and 
$B\in\cB(U)$ with $\nu(B)<\infty$. 
It satisfies $\bE I^N_t(\Phi)=\bE\int_0^T\int_U\Phi(s,x)\,\nu(\dl x)\,\dl s$ and $\bE I^N_t(\|\Phi\|)=\bE\int_0^T\int_U\|\Phi(s,x)\|\,\nu(\dl x)\,\dl s$. We will usually write $\int_0^t\int_U\Phi(s,x)N(\diffin s,\diffin y)$ instead of $I^N_t(\Phi)$. For $p\geq1$ we denote by $\cM_T^p(H)$ be the space of $L^p$-integrable, $H$-valued $(\cF_t)$-martingales $M=(M(t))_{t\in[0,T]}$ with càdlàg (right continuous with left limits) paths. The norm in $\cM_T^p(H)$ is given by $\|M\|_{\cM^p_T(H)}:=(\E\sup_{t\in[0,T]}\|M(t)\|^p)^{1/p}=(\E\|M(T)\|^p)^{1/p}$.
For $p\in[1,2]$ the $L^p$ stochastic integral w.r.t.\ $\tilde N$ 
\felix{see notes 20.2.'16}
\begin{equation}\label{eq:I^tildeN}
I^{\tilde N}\colon L^p_{\pred}(\Omega_T\times U; H)\to \cM_T^p(H),\;\Phi\mapsto I^{\tilde N}(\Phi)=\big(I^{\tilde N}_t(\Phi)\big)_{t\in[0,T]},
\end{equation}
is a linear and continuous mapping
defined by linear and continuous extension of the integral for simple integrands
$I^N_t(F\,\one_{(r,s]\times B}) :=F\, \tilde N((r,s\wedge t]\times B)$, where $F,r,s$ and $B$ are as above. For $p=2$ it is in fact an isometry and we have $\E\|I^{\tilde N}_t(\Phi)\|^2=\bE\int_0^t\int_U\|\Phi(s,x)\|^2\,\nu(\dl x)\dl s$. For different $p\in[1,2]$ the respective integral mappings $I^{\tilde N}$ coincide on the intersections of their domains, and for $\Phi\in L^1_{\pred}(\Omega_T\times U; H)$ it holds that $I^{\tilde N}_t(\Phi)=I^{N}_t(\Phi)-\int_0^t\int_U\Phi(s,x)\,\nu(\dl x)\dl s$. We will usually write $\int_0^t\int_U\Phi(s,x)\tilde N(\diffin s,\diffin x)$ instead of $I^{\tilde N}_t(\Phi)$. 
\end{remark}

We now relate \eqref{eq:I^N_t} and \eqref{eq:I^tildeN} to the integral mappings considered in Section~\ref{subsec:Kabanov-Skorohod}.

\begin{prop}\label{lem:space_time1}
For $\Phi\in L^1_{\pred}(\Omega_T\times U;H)$ we have
\begin{align*}
  \int_0^T\int_U
    \varepsilon_{t,x}^-\Phi(t,x)\,
  N(\diffin{t},\diffin{x})
  =
  \int_0^T\int_U
    \Phi(t,x)\,
  N(\diffin{t},\diffin{x})
\end{align*}
as an equality in $L^1(\Omega;H)$, where the integral on the left hand side is defined according to Proposition~\ref{lem:stoch_int} and the integral on the right hand side is the $L^1$ stochastic integral $I^N_T(\Phi)$ introduced in Remark~\ref{rem:stoch_int_space-time}.
\end{prop}

\begin{proof}
\felix{see notes 19.2.'16}
The linear span of simple processes of the form
\begin{align}\label{eq:simpleIntegrand}
\Phi=F\,\one_{(r,s]\times B},\qquad &0\leq r\leq s\leq T,\; B\in\cB(U)\text{ with }\nu(B)<\infty,\\
&F\colon\Omega\to H\text{ bounded and $\cF_r$-$\cB(H)$-measurable},\notag
\end{align}
is dense in $L^1_{\pred}(\Omega_T\times U;H)$, and the considered integrals define continuous mappings from $L^1_{\pred}(\Omega_T\times U;H)$ to $L^1(\Omega;H)$. Hence it suffices to prove the equality for integrands $\Phi$ as in \eqref{eq:simpleIntegrand}. As $\varepsilon_{t,x}^-\Phi(t,x)=(\varepsilon_{t,x}^-F)\one_{(r,s]\times B}(t,x)$ it is enough to show
\begin{align}\label{eq:space_time1}
\varepsilon_{t,x}^-F=F\quad\P\otimes N\text{-almost everywhere on }\Omega\times(r,T]\times U,
\end{align}
where $\bP\otimes N$ is the product measure from \eqref{eq:PotimesN}.
Fix $\epsilon\in(0,T-r]$ and recall the definition of $\cF_t$ from Notation~\ref{notation_space-time}. Since $F$ is $\cF_{[0,r+\epsilon]\times U}$-$\cB(H)$-measurable, there exists a $\cN$-$\cB(H)$-measurable mapping $f\colon\mathbf N\to H$ such that $F=f\big(N_{[0,r+\epsilon]\times U}\big)$. Let us define $f_\epsilon\colon\mathbf{N}\to H$ by $f_\epsilon(\eta):=f\big(\eta_{[0,r+\epsilon]\times U}\big)$, $\eta\in\mathbf N$, where $\eta_{[0,r+\epsilon]\times U}\in\mathbf N$ is given by $\eta_{[0,r+\epsilon]\times U}(A):=\eta\big(A\cap ([0,r+\epsilon]\times U)\big)$, $A\in\cB([0,T]\times U)$. Then it holds
$\P\otimes N$-almost everywhere in $\Omega\times(r+\epsilon,T]\times U$ that
\begin{align*}
\varepsilon_{t,x}^-F=f_{\epsilon}\big(N\setminus\delta_{t,x}\big)=f_{\epsilon}(N)=F.
\end{align*}
As $\epsilon$ can be chosen arbitrarily small we obtain \eqref{eq:space_time1}.
\end{proof}

\begin{prop}\label{prop:space_time1}
For $p\in[1,2]$ the space $L^p_{\pred}(\Omega_T\times U;H)$ is contained in $\dom(\delta^{(p)})$ and for all $\Phi\in L^p_{\pred}(\Omega_T\times U;H)$ we have
\begin{align*}
\delta(\Phi)=\int_0^T\int_U\Phi(t,x)\tilde N(\dl t,\dl x)
\end{align*}
as an equality in $L^p(\Omega;H)$, where the integral on the right hand side is the $L^p$ stochastic integral $I^{\tilde N}_T(\Phi)$ introduced in Remark~\ref{rem:stoch_int_space-time}.
\end{prop}

\begin{proof}
\felix{compare notess 20.2.'16}
For $p=1$ the assertion follows directly from Definition~\ref{def:pathwise_Kabanov} and Proposition~\ref{lem:space_time1}. We fix $p\in(1,2]$. Due to the closedness of $\delta^{(p)}$ and the continuity of the integral mapping $I^{\tilde N}_T\colon L^p_{\pred}(\Omega_T\times U;H)\to L^p(\Omega;H)$, it is sufficient to verify the assertion for simple integrands $\Phi$ of the form \eqref{eq:simpleIntegrand}. Let $p':=p/(p-1)$ be the dual exponent of $p$, consider $F\in\bD^{1,p'}(H)$, and let $(F_n)_{n\in\bN}\subset\bD^{1,p'}(H)\cap L^\infty(\Omega;H)$ be an approximating sequence as in Lemma~\ref{lem:coreD12}(i). Then, using the assertion for $p=1$ and the duality formula in Proposition~\ref{lem:duality},
\begin{align*}
\tensor[_{L^{p'}(\Omega;H)}]
{\Big\langle F_n,\int_0^T\int_U\Phi(t,x)\tilde N(\dl t,\dl x),\Big\rangle}
{_{L^p(\Omega;H)}}
&=
\tensor[_{L^\infty(\Omega;H)}]
{\big\langle F_n,\delta(\Phi)\big\rangle}
{_{L^1(\Omega;H)}}\\
=
\tensor[_{L^\infty(\Omega_T\times U;H)}]
{\big\langle DF_n,\Phi\big\rangle}
{_{L^1(\Omega_T\times U;H)}}
&=
\tensor[_{L^{p'}(\Omega_T\times U;H)}]
{\big\langle DF_n,\Phi\big\rangle}
{_{L^p(\Omega_T\times U;H)}}.
\end{align*}
Since $F_n\to F$ in $\bD^{1,p'}(H)$ and $F\in\bD^{1,p'}(H)$ was chosen arbitrarily, this yields that $\Phi$ belongs to $\dom(\delta^{(p)})$ and $\delta(\Phi)=\int_0^T\int_U\Phi(t,x)\tilde N(\dl t,\dl x)$.
\end{proof}

\subsection{Duality and commutation relations in the space-time setting}
\label{subsec:duality_comm_space-time}

As a consequence of Proposition~\ref{prop:space_time1}, in the case of predictable integrands $\Phi$ the duality and commutation relations from Sections~\ref{subsec:Kabanov-Skorohod} and \ref{sec:commutation} are statements about Itô-type stochastic integrals w.r.t.~$\tilde N$. Moreover, by additionally exploiting the continuity of the stochastic integral mapping \eqref{eq:I^tildeN}, we obtain significant improvements and extensions of some of the general results from Sections~\ref{subsec:Kabanov-Skorohod} and \ref{sec:commutation}, cf.~Propositions~\ref{prop:duality_space_time_local} and \ref{prop:comm_space-time2} below. These extension will be crucial for applications, in particular for the analysis of Lévy-driven SPDE.

We first reformulate the global duality relations from Proposition~\ref{lem:duality} and Definition~\ref{def:Kabanov-Skorohod} for predictable integrands.

\begin{prop}[Duality formula]\label{prop:duality_space_time}
Let $p\in[1,2]$ and $p'\in[2,\infty]$ be such that $\frac1p+\frac 1{p'}=1$. Then, for all $\Phi\in L^{p}_{\pred}(\Omega_T\times U;H)$ and $F\in\bD^{1,p'}(H)$ we have 
\begin{align}\label{eq:duality_space_time}
\bE\,\Big\langle F,\int_0^T\int_U\Phi(t,x)\,\tilde N(\dl t,\dl x)\Big\rangle=\E\int_0^T\int_U\big\langle D_{t,x}F,\,\Phi(t,x)\big\rangle\,\nu(\dl x)\,\dl t.
\end{align}
\end{prop}

\begin{proof}
The assertion is a direct consequence of Proposition~\ref{prop:space_time1} and the duality relation between the operators $D\colon\bD^{1,p'}(H)\subset L^{p'}(\Omega;H)\to L^{p'}(\Omega_T\times U;H)$ and $\delta\colon\dom(\delta^{(p)})\subset L^p(\Omega_T\times U;H)\to L^p(\Omega;H)$.
\end{proof}

Next we prove a local version of the duality formula of Proposition~\ref{prop:duality_space_time}. 
It is particularly useful for the analysis of $\alpha$-stable noises, cf.~Section~\ref{sec:alpha_stable}, and 
it corresponds to the local duality formula in Lemma~\ref{lemma:duality2} in the general setting. 
In contrast to the proof of Lemma~\ref{lemma:duality2} we are now able to
exploit the continuity of the stochastic integral mapping~\eqref{eq:I^tildeN}, so that the integrand $\Phi$ does not have to vanish outside a set of finite measure any more but outside a set of arbitrary measure. As in Lemma~\ref{lemma:duality2} the significance of the local duality formula lies in the fact that it does not rely on a global integrability assumption on $DF$.

\begin{prop}[Local duality formula]\label{prop:duality_space_time_local}\felix{$p=1$ admissible?}
Let $p\in[1,2]$ and $p'\in[2,\infty]$ be such that $\frac1p+\frac 1{p'}=1$. Let $\Phi\in L^{p}_{\pred}(\Omega_T\times U;H)$ be such that $\Phi=0$ $\P_T\otimes\nu$-almost everywhere on $\Omega_T\times B^c$ for some $B\in\cB(U)$, and let $F\in L^{p'}(\Omega;H)$ be such that $\one_{\Omega_T\times B}DF\in L^{p'}(\Omega_T\times U;H)$. 
Then the duality relation \eqref{eq:duality_space_time} holds.
\end{prop}

\begin{proof}
\felix{notes 19.4.-17.5.'16--3--, compare 3.10.'15--4--}
Due to the assumptions we have 
$\big\langle F,\int_0^T\int_U\Phi(t,x)\,\tilde N(\dl t,\dl x)\big\rangle\in L^1(\Omega;\bR)$ and $\langle DF,\,\Phi\rangle\in L^1(\Omega_T\times U;\bR)$, so that the integrals in \eqref{eq:duality_space_time} are defined.
 Let $(B_n)_{n\in\bN}\subset\cB(U)$ be such that $\nu(B_n)<\infty$ for all $n\in\bN$ and $B_n\nearrow B$ as $n\to\infty$. Then $\one_{\Omega_T\times B_n}\Phi\xrightarrow{n\to\infty}\Phi$ in $L^p(\Omega_T\times U;H)$ and therefore
\begin{align*}
&\bE\,\Big\langle F,\int_0^T\int_{B_n}\Phi(t,x)\,\tilde N(\dl t,\dl x)\Big\rangle
\;\xrightarrow{n\to\infty}\;
\bE\,\Big\langle F,\int_0^T\int_U\Phi(t,x)\,\tilde N(\dl t,\dl x)\Big\rangle,\\
&\E\int_0^T\int_{B_n}\big\langle D_{t,x}F,\,\Phi(t,x)\big\rangle\,\nu(\dl x)\,\dl t
\;\xrightarrow{n\to\infty}\;
\E\int_0^T\int_U\big\langle D_{t,x}F,\,\Phi(t,x)\big\rangle\,\nu(\dl x)\,\dl t,
\end{align*}
where we use the continuity of the stochastic integral mapping $\eqref{eq:I^tildeN}$. Thus is suffices to verify the duality relation \eqref{eq:duality_space_time} with $\one_{\Omega_T\times B_n}\Phi$ in place of $\Phi$. This can be done similarly as in the proofs of Proposition~\ref{lem:duality} and Lemma~\ref{lemma:duality2}, using Mecke's formula~\eqref{eq:Mecke2}. The crucial point is that $\one_{\Omega_T\times B_n}\Phi\in L^1(\Omega_T\times U;H)$ and consequently, by Proposition~\ref{prop:coincides} and \ref{prop:space_time1}, the $L^p$-stochastic integral of $\one_{\Omega_T\times B_n}\Phi$ w.r.t.\ $\tilde N$ can be written as the difference of pathwise $L^1$-integrals $\int_0^T\int_{B_n}\varepsilon^-_{t,x}\Phi(t,x)\,N(\dl t,\dl x)-\int_0^T\int_{B_n}\Phi(t,x)\,\nu(\dl x)\,\dl t$ in the sense of Definition~\ref{def:pathwise_Kabanov}. Thereby we have that $\int_0^T\int_{B_n}\varepsilon^-_{t,x}\Phi(t,x)\,N(\dl t,\dl x)\in L^p(\Omega;H)$ since both $\int_0^T\int_{B_n}\Phi(t,x)\,\tilde N(\dl t,\dl x)$ and $\int_0^T\int_{B_n}\Phi(t,x)\,\nu(\dl x)\,\dl t$ belong to $L^p(\Omega;H)$.
\end{proof}

We proceed by considering the commutation relations between $D$ and $\delta$ in the space-time case. We need an auxiliary result.
\begin{lemma}[Predictability of the derivative]
\label{lem:DPhi_predictable}
Assume $\Phi\in L^{p}_{\pred}(\Omega_T\times U;H)$ for some $p\geq1$. Then the
derivative $D\Phi\in L^0\big(\Omega_T\times U\times([0,T]\times U);H\big)$ has a $\cP_T\otimes\cB(U)\otimes\cB([0,T]\times U)$-measurable version. That is, the mapping
\begin{equation*}\label{eq:space_time2}
D\Phi\colon\Omega_T\times U\times([0,T]\times U)\to H,\;(\omega,s,y,t,x)\mapsto D_{t,x}\Phi(\omega,s,y)
\end{equation*} 
has a $\P_T\otimes\nu\otimes(\lambda\otimes\nu)$-version which is $\cP_T\otimes\cB(U)\otimes\cB([0,T]\times U)$-measurable.
\end{lemma}

\begin{proof}\felix{see notes 21.2.'16}
Recall from Notation~\ref{notation_space-time} the definition of the filtration $(\cF_t)_{t\in[0,T]}$ and the definition of the predictable $\sigma$-algebra $\cP_T$. 
We first assume that $\Phi$ is a simple function of the form \eqref{eq:simpleIntegrand}, i.e., $\Phi=F\one_{(r,s]\times B}$ with $0\leq r\leq s\leq T$, $B\in\cB(U)$ such that $\nu(B)<\infty$, and $F\colon\Omega\to H$ bounded and $\cF_r$-$\cB(H)$-measurable. Then $F$ is $\cF_{[0,r+\varepsilon]\times U
}$-$\cB(H)$-measurable for all $\varepsilon\in(0,T-r]$, so that there exist $\cN$-$\cB(H)$-measurable mappings $f_\varepsilon\colon \mathbf N\to H$ such that $F=f_\varepsilon(N_{[0,r+\varepsilon]\times U})$ $\bP$-almost surely. 
Here we denote by $N_{[0,r+\varepsilon]\times U}\colon\Omega\to\mathbf N$ the point process obtained by removing from $N$ all point masses located outside of $[0,r+\varepsilon]\times U$, i.e., $N_{[0,r+\varepsilon]\times U}(\omega,C):=N(\omega,([0,r+\varepsilon]\times U)\cap C)$, $\omega\in\Omega$, $C\in\cB([0,T]\times U)$, compare the proof of  Lemma~\ref{lem:DF_F_FA_measurable}. 
It is straightforward to check that for all $\varepsilon\in (0,T-r]$ the derivate 
$D\Phi_\varepsilon$ of $\Phi_\varepsilon:=f_\varepsilon(N_{[0,r+\varepsilon]\times U})\one_{(r+\varepsilon,s]\times B}$ has a $\cP_T\otimes\cB(U)\otimes\cB([0,T]\times U)$-measurable version. 
This implies the assertion for simple $\Phi$ is above in the following way: As $\Phi_\varepsilon\xrightarrow{\varepsilon\to 0}\Phi$ in $L^p(\Omega_T\times U;H)$ by dominated convergence, we have that $D\Phi_\varepsilon\xrightarrow{\varepsilon\to 0}D\Phi$ in $L^0(\Omega_T\times U\times([0,T]\times U);H)$ by the continuity of $D$ stated in Remark~\ref{rem:D_for_random_fields}. 
As a consequence,
there exists a sequence $(\varepsilon_n)_{n\in\bN}\subset(0,T-r]$ decreasing to zero such that $D\Phi_{\varepsilon_n}\xrightarrow{n\to\infty}D\Phi$ $\bP_T\otimes\nu\otimes(\lambda\otimes\nu)$-almost everywhere. Thus, any $\cP_T\otimes\cB(U)\otimes\cB([0,T]\times U)$-measurable version of the pointwise limit $\lim_{n\to\infty}D\Phi_{\varepsilon_n}$ can be taken as the desired version of $D\Phi$. For general $\Phi\in L^p_{\pred}(\Omega_T\times U;H)$ the assertion follows from approximation by linear combinations of simple predictable functions, using again the continuity of $D$ as stated in Remark~\ref{rem:D_for_random_fields}.
\end{proof}

Due to Lemma~\ref{lem:DPhi_predictable} all integrals in the commutation relations below are defined.

\begin{prop}[Commutation relation -- $L^1$-version]
\label{prop:comm_space-time1}
Let $\Phi\in L^{1}_{\pred}(\Omega_T\times U;H)$ be such that $\E\int_0^T\int_U\|D_{t,x}\Phi(s,y)\|\,\nu(\dl y)\,\dl s<\infty$ for $\lambda\otimes\nu$-almost all $(t,x)\in[0,T]\times U$. Then the equality
\begin{align}\label{eq:comm_rel_space-time}
D_{t,x}\int_0^T\int_U\Phi(s,y)\,\tilde N(\dl s,\dl y)=\int_0^T\int_UD_{t,x}\Phi(s,y)\,\tilde N(\dl s,\dl y)+\Phi(t,x)
\end{align}
holds $\P_T\otimes\nu$-almost everywhere.
\end{prop}

\begin{proof}
The result follows from Proposition~\ref{prop:space_time1}, Lemma~\ref{lem:DPhi_predictable} and the $L^1$-commutation relation in Proposition~\ref{lem:commutator1}. Note that the integral $\int_0^T\int_UD_{t,x}\Phi(s,y)\tilde N(\dl s,\dl y)=I^{\tilde N}_T(D_{t,x}\Phi)$ exists in the $L^1$-sense for $\lambda\otimes\nu$-almost all $(t,x)\in[0,T]\times U$ if we consider a predictable version of $D\Phi$ as in Lemma~\ref{lem:DPhi_predictable}. 
Also note that the integral process $\big(\int_0^T\int_UD_{t,x}\Phi(s,y)\tilde N(\dl s,\dl y)\big)$ defines an element in $L^0([0,T]\times U;L^1(\Omega;H))\subset L^0(\Omega_T\times U;H)$. The latter follows from the continuity of the integral mapping $I_T^{\tilde N}\colon L^1_{\pred}(\Omega_T\times U;H)\to L^1(\Omega;H)$, cf.~$\eqref{eq:I^tildeN}$, and the fact that $D\Phi=(D_{t,x}\Phi)\in L^0\big([0,T]\times U;L^1_{\pred}(\Omega_T\times U;H)\big)$ according to our conventions concerning $L^0$-spaces stated in Section~\ref{sec:Notation_and_conventions}.
\end{proof}

Combining Proposition~\ref{prop:comm_space-time1} and the continuity of the stochastic integral mapping \eqref{eq:I^tildeN} we prove the following $L^p$-version of the commutation relation between $D$ and $\delta$ for predictable integrands. Let us remark  that even in the case $p=2$ the assertion differs from that of Proposition~\ref{prop:commutator_relation_L^2} since the subspaces $L^2([0,T]\times U;\bD^{1,2}(H))$ and $L^2_{\pred}(\Omega_T\times U;H)$ of $\dom(\delta^{(2)})$ obviously do not coincide.

\begin{prop}[Commutation relation -- $L^p$-version]
\label{prop:comm_space-time2}
Let $p\in(1,2]$, $\tilde p\in[1,2]$ and $\Phi\in L^p_{\pred}(\Omega_T\times U;H)$ be such that $\E\int_0^T\int_U\|D_{t,x}\Phi(s,y)\|^{\tilde p}\,\nu(\dl y)\,\dl s<\infty$ for $\lambda
\otimes\nu$-almost all $(t,x)\in[0,T]\times U$. Then the
equality \eqref{eq:comm_rel_space-time}
holds $\P_T\otimes\nu$-almost everywhere.
\end{prop}

\begin{proof}
Let $(B_n)_{n\in\bN}\subset\cB(U)$ be such that $\nu(B_n)<\infty$, $n\in\bN$, and $B_n\nearrow U$. We set $\Phi_n:=\one_{\Omega_T\times B_n}\Phi$. 
As $\Phi_n$ satisfies the assumptions in Proposition~\ref{prop:comm_space-time1}, the equality \eqref{eq:comm_rel_space-time} with $\Phi_n$ in place of $\Phi$ holds $\P_T\otimes\nu$-almost everywhere, for all $n\in\bN$.
Now the assertion follows by letting $n$ tend to infinity in each term of this equality:
For the first term we have that 
$
D_{t,x}\int_0^T\int_U\Phi_n(s,y)\,\tilde N(\dl s,\dl y)
\xrightarrow{n\to\infty}
D_{t,x}\int_0^T\int_U\Phi(s,y)\,\tilde N(\dl s,\dl y)
$
in $L^0([0,T]\times U;L^q(\Omega;H))\subset L^0(\Omega_T\times U;H)$ for all $q\in [1,p)$. This is a consequence of the convergence $\Phi_n\to\Phi$ in $L^p(\Omega_T\times U;H)$, the continuity of the stochastic  integral mapping $I_T^{\tilde N}\colon L^p_{\pred}(\Omega_T\times U;H)\to L^p(\Omega;H)$, cf.~$\eqref{eq:I^tildeN}$, and the continuity of $D|_{L^p(\Omega;H)}$ from $L^p(\Omega;H)$ to $L^0([0,T]\times U;L^q(\Omega;H))$, cf.~Corollary~\ref{cor:localizing_lemma_1}. 
Concerning the second term in \eqref{eq:comm_rel_space-time} with $\Phi_n$ in place of $\Phi$, observe that, for $\lambda\otimes\nu$-almost all $(t,x)\in[0,T]\times U$,
\begin{align*}
\big\|D_{t,x}\Phi-D_{t,x}\Phi_n\big\|_{L^{\tilde p}(\Omega_T\times U;H)}
=
\Big(\bE\int_0^T\int_{B_n^c}\big\|D_{t,x}\Phi(s,y)\big\|^{\tilde p}\nu(\dl y)\,\dl s\Big)^{1/\tilde p}
\xrightarrow{n\to\infty}0.
\end{align*}
This implies $D_{t,x}\Phi_n\xrightarrow{n\to\infty}D_{t,x}\Phi$ in $L^0\big([0,T]\times U;L^{\tilde p}_{\pred}(\Omega_T\times U;H)\big)$. Due to the continuity of the integral mapping $I_T^{\tilde N}\colon L^{\tilde p}_{\pred}(\Omega_T\times U;H)\to L^{\tilde p}(\Omega;H)$ we obtain the convergence $\int_0^T\int_UD_{t,x}\Phi_n(s,y)\,\tilde N(\dl s,\dl y)\xrightarrow{n\to\infty}\int_0^T\int_UD_{t,x}\Phi(s,y)\,\tilde N(\dl s,\dl y)$ in 
$L^0\big([0,T]\times U;L^{\tilde p}(\Omega;H)\big)\subset L^0(\Omega_T\times U;H)$.
Finally, concerning the third term in \eqref{eq:comm_rel_space-time} with $\Phi_n$ in place of $\Phi$, we have $\Phi_n(t,x)\xrightarrow{n\to\infty}\Phi(t,x)$ in $L^p(\Omega_T\times U;H)$.
\end{proof}
%

\begin{remark}
The proofs above additionaly imply that, in the situation of Proposition~\ref{prop:comm_space-time1} and \ref{prop:comm_space-time2}, the commutation relation \eqref{eq:comm_rel_space-time} holds as an equality in the spaces
$L^0\big([0,T]\times U;L^1(\Omega;H)\big)\subset L^0(\Omega_T\times U;H)$ 
and 
$L^0\big([0,T]\times U;L^{p\wedge\tilde p}(\Omega;H)\big)\subset L^0(\Omega_T\times U;H)$,
respectively.
\end{remark}

\subsection{Hilbert space-valued Lévy processes}
\label{subsec:PRMandLP}

We will use the Poisson Malliavin calculus developed so far to analyze Hilbert space-valued stochastic evolution equations with Lévy noise. 
Here we describe how  
Hilbert space-valued valued Lévy processes can be embedded into our framework 
and present some important examples for such processes. 
To simplify the exposition we restrict ourselves to  integrable, mean-zero Lévy processes without Gaussian part. This class particularly includes the $\alpha$-stable processes  considered in Section~\ref{sec:alpha_stable} below, where $\alpha\in(1,2)$.
Our standard reference 
in this context
is \cite{PesZab2007}.

We first formulate our assumptions and then describe the relation to Setting~\ref{setting}.
\begin{assumption}\label{ass:L}
$L=(L(t))_{t\in[0,T]}$ is a Lévy process taking values in a separable Hilbert space $(U,\|\cdot\|_U,\langle\cdot,\cdot\rangle_U)$, defined on a complete probability space $(\Omega,\cF,\P)$ such that the $\sigma$-algebra $\cF$ coincides with the $\P$-completion of $\sigma(L(t):t\in[0,T])$. We assume that
\begin{itemize}
\item $L$ is integrable with mean zero, i.e., $L(t)\in L^1(\Omega;U)$ and $\bE( L(t))=0$;
\item the Gaussian part of $L$ is zero.
\end{itemize}
\end{assumption}

Let $\nu$ denote the jump intensity measure (Lévy measure) of $L$ . Recall that the jump intensity measure $\nu$ of a general $U$-valued Lévy process satisfies $\nu(\{0\})=0$ and 
$\int_{U}\min(\|x\|_U^2,1)\,\nu(\dl x)<\infty$, 
cf.~\cite[Section~4]{PesZab2007}. 
Due to Assumption~\ref{ass:L} and the Lévy-Khinchin decomposition (see, e.g., \cite[Theorem~4.23]{PesZab2007}) we have
\begin{align}\label{eq:ass_nu_minimal}
\int_{U}\min(\|x\|_U^2,\|x\|_U)\,\nu(\dl x)<\infty
\end{align}
and the characteristic function of $L(t)$ is given by
\begin{align}\label{eq:L_char_fn}
\bE e^{i\langle x,L(t)\rangle_U}=\exp\Big(-t\int_U\big(1-e^{i\langle x,y\rangle_U}+i\langle x,y\rangle_U\big)\,\nu(\dl y)\Big),\quad x\in U,\;t\geq0.
\end{align}
Conversely, every $U$-valued Lévy process $L$ satisfying \eqref{eq:ass_nu_minimal} and \eqref{eq:L_char_fn} is integrable with mean zero and vanishing Gaussian part.

We always consider a càdlàg (right continuous with left limits) modification of $L$, i.e., a modification such that $L(t)=\lim_{s\searrow t}L(s)$ for all $t\geq0$ and $L(t-):=\lim_{s\nearrow t}L(s)$ exists for all $t>0$, where the limits are pathwise limits in $U$.
The jumps of $L$ determine a Poisson random measure as follows: 
For $(\omega,t)\in\Omega\times(0,T]$ we denote by $\Delta L(t)(\omega):=L(t)(\omega)-L(t-)(\omega)\in U$ the jump of a trajectory of $L$ at time $t$. Then
\begin{equation}\label{eq:LjumpPRM}
N(\omega):=\sum_{t\in(0,T]:\Delta L(t)(\omega)\neq 0}\delta_{(t,\Delta L(t)(\omega))},\quad\omega\in\Omega,
\end{equation}
defines a Poisson random measure $N$ on $(\X,\cE):=([0,T]\times U,\,\cB([0,T]\times U))$ with intensity measure $\lambda\otimes\nu$, where $\delta_{(t,y)}$, $\lambda$ and $\nu$ denote Dirac measure at $(t,y)\in[0,T]\times U$, Lebesgue measure on $[0,T]$ and the jump intensity measure of $L$, respectively. This follows, e.g., from Theorem~6.5 in \cite{PesZab2007} together with Theorems~4.9, 4.15, 4.23 and Lemma 4.25 therein.


Recall the notation $B_U=\{x\in U:\|x\|_U\leq 1\}$ for the closed unit ball in $U$ and $B_U^c=U\setminus B_U$ for its complement.

\begin{lemma}\label{lem:LvsN}
Let Assumption~\ref{ass:L} hold, let $N$ be the Poisson random measure on $([0,T]\times U,\,\cB([0,T]\times U))$ associated to $L$ via \eqref{eq:LjumpPRM}, and let $\nu$ be the Lévy measure of $L$. Then the assumptions in Setting~\ref{setting} are fulfilled with $(E,\cE,\mu)$ given by \eqref{eq:E_space-time}. 
Moreover,  we have
\begin{align*}
\P\Big(L(t)=\int_0^t\int_{B_U}x\,\tilde N(\dl s,\dl x)+\int_0^t\int_{B_U^c}x\,\tilde N(\dl s,\dl x)\quad\forall t\in[0,T]\,\Big)=1,
\end{align*}
where the 
càdlàg 
processes 
$\big(\int_0^t\int_{B_U}x\,\tilde N(\dl s,\dl x)\big)_{t\in[0,T]}$, $\big(\int_0^t\int_{B_U^c}x\,\tilde N(\dl s,\dl x)\big)_{t\in[0,T]}$
are defined in terms of $U$-valued stochastic integrals w.r.t.\ $\tilde N$ in the $L^2$ sense and in the $L^1$ sense, respectively. In particular, the $U$-valued Malliavin derivative of $L(t)$ is given by $D_{s,x}L(t)=\one_{s\leq t}\,x$, $(s,x)\in[0,T]\times U$.
\end{lemma}

\begin{proof}
The first assertion is obvious. The stochastic integral processes are well-defined as $\int_0^T\int_{B_U}\|x\|_U^2\,\nu(\dl x)\dl s$ and $\int_0^T\int_{B_U^c}\|x\|_U\,\nu(\dl x)\dl s$ are finite due to \eqref{eq:ass_nu_minimal}. The indistinguishability assertion follows from the Lévy-Khinchin decomposition of $L$ (\cite[Theorem~4.23]{PesZab2007}) and the continuity of the stochastic integral mapping $I^{\tilde N}\colon L^p_{\pred}(\Omega_T\times U;U)\to \cM^p_T(U)$, compare Remark~\ref{rem:stoch_int_space-time}.
Finally, the last assertion is a direct consequence of Proposition~\ref{prop:comm_space-time1}, \ref{prop:comm_space-time2} and Corollary~\ref{cor:DF_is_zero}.
\end{proof}

\begin{remark}
It is clear that the $U$-valued stochastic integral $\int_0^t\int_{B_U^c}x\,\tilde N(\dl s,\dl x)$ can also be defined $\omega$-wise since, with probability one, $L$ has only finitely many jumps of size $\|\Delta L(t)(\omega)\|_U>1$ on finite time intervals. However, as we will repeatedly use the boundedness of the integral operator $I^{\tilde N}\colon L^p_{\pred}(\Omega_T\times U;U)\to \cM^p_T(U)$ for $p\in[0,1]$, it is more convenient for us to continuously adopt the point of view on stochastic integrals described in Remark~\ref{rem:stoch_int_space-time}. 
\end{remark}

We end this subsection with two important examples of Lévy processes $L$ satisfying Assumption~\ref{ass:L}. 
\felix{see notes 21.5.'16}
For further details we refer to \cite[Examples~2.3, 2.5]{KovLinSch2015}, where similar processes are considered under stronger integrability assumptions.
\begin{example}[Subordinate cylindrical $Q$-Wiener process] 
\label{ex:subordWP_general}
Let $Q\in\LB(H)$ be a nonnegative and symmetric operator and $W=(W(t))_{t\in[0,T]}$ be a cylindrical $Q$-Wiener process in $H$ in the sense of \cite[Section~2.5.1]{roeckner2007}. Assume that the Hilbert space $U$ is such that the Cameron-Martin space $Q^{1/2}(H)$ of $W$ is emdedded into $U$ via a Hilbert-Schmidt embedding, where $Q^{1/2}(H)$ is endowed with the inner product $\langle Q^{-1/2}\,\cdot,Q^{-1/2}\,\cdot\rangle$, $Q^{-1/2}$ being the pseudo-inverse of $Q^{1/2}$. Then $W$ has a realization taking values in $U$. Let $Z=(Z(t))_{t\geq0}$ be a real-valued increasing Lévy process (subordinator) in the sense of \cite[Definition~21.4]{Sat1999}, such that $W$ and $Z$ are independent, the drift of $Z$ is zero, and the jump intensity measure $\varrho$ of $Z$ satisfies 
$
\int_{(0,\infty)} \min(s,\sqrt s)\,\varrho(\dl s)<\infty.
$
In this situation, subordinate cylindrical Brownian motion
\[L(t):=W(Z(t)),\quad t\in[0,T],\]
defines a $U$-valued Lévy process with Lévy measure $\nu=(\varrho\otimes\P_{W(1)})\circ\kappa^{-1}$, where $\kappa\colon (0,\infty)\times U\to U$ is defined by $\kappa(s,y)=\sqrt{s}y$. 
The process $L$ satisfies \eqref{eq:ass_nu_minimal} and \eqref{eq:L_char_fn}. To verify \eqref{eq:L_char_fn} one can argue as in \cite[Example~2.3]{KovLinSch2015}. The integrability property \eqref{eq:ass_nu_minimal} holds since
\begin{align*}
\int_{B_U^c}\|y\|_U\,\nu(\dl y)
&=\int_{(0,\infty)}\int_U\one_{B_U^c}(\sqrt sy)\|\sqrt sy\|_U\,\P_{W(1)}(\dl y)\varrho(\dl s)\\
&=\int_{(0,\infty)} \sqrt s\,\bE\big[\one_{(s^{-1/2},\infty)}(\|W(1)\|_U)\|W(1)\|_U\big]\,\varrho(\dl s),
\end{align*}
where the expectation in the last integral can be estimated by
\begin{align*}
\int_{s^{-1/2}}^\infty\P(\|W(1)\|_U\geq r)\,\dl r
\leq\sqrt{s}\int_0^\infty\!r\, \P(\|W(1)\|_U\geq r)\,\dl r
=\frac{\sqrt s}2\,\bE(\|W(1)\|_U^2),
\end{align*}
so that
\[
\int_{B_U^c}\|y\|_U\,\nu(\dl y)
\leq
\int_{(0,1]} \frac s2\,\varrho(\dl s)\,\bE(\|W(1)\|_U^2)+\int_{(1,\infty)} \sqrt s\,\varrho(\dl s)\,\bE(\|W(1)\|_U)<\infty;\] the finiteness of 
$\int_{B_U}\|y\|^2_U\,\nu(\dl y)$ can be verified in a similar way. Subordinate cylindrical Wiener processes have been analyzed, e.g., in \cite{BrzZab2010}.
\end{example}

\begin{example}[Impulsive cylindrical process]\label{ex:impCylProc_general}
Let $\cO\subset\bR^d$ be open and bounded and $\tilde\pi$ be a compensated Poisson random measure on $[0,T]\times\cO\times\bR$ with intensity measure $\lambda^1\otimes\lambda^d\otimes\varrho$, where $\lambda^d$ denotes $d$-dimensional Lebesgue measure and $\varrho$ is a Lévy measure on $\bR$ satsfying $\int_\bR\min(\sigma^2,\sigma)\,\varrho(\dl\sigma)<\infty$. We consider the measure-valued process $L=(L(t))_{t\in[0,T]}$ defined, informally, by 
\[L(t,\dl\xi)=\int_0^t\int_\bR\sigma\,\tilde\pi(\dl s,\dl\xi,\dl\sigma).\]
More rigorously, let $U\supset L^2(\cO)$ be a separable real Hilbert space containing the Dirac measures $\delta_\xi$, $\xi\in\cO$, such that the mapping $\cO\ni\xi\mapsto \delta_\xi\in U$ is measurable and $\|\delta_\xi\|_U=\|\delta_{\xi'}\|_U$ for all $\xi,\xi'\in\cO$; 
for instance, any negative order $L^2$-Sobolev space $U:=H^{-\gamma}(\cO)=(H^\gamma_0(\cO))^*$ with $\gamma>d/2$ is a suitable choice. 
Then we can define $L=(L(t))_{t\in[0,T]}$ as the $U$-valued Lévy process  given by 
\[L(t):=\int_0^t\int_{\cO}\int_{\{|\sigma|\leq1\}}\sigma\delta_\xi\,\tilde\pi(\dl s,\dl\xi,\dl\sigma)+\int_0^t\int_{\cO}\int_{\{|\sigma|>1\}}\sigma\delta_\xi\,\tilde\pi(\dl s,\dl\xi,\dl\sigma),\]
where the integrals on the right hand side are $U$-valued stochastic integrals w.r.t.~$\tilde \pi$ in the $L^2$ sense and in the $L^1$ sense, respectively. The stochastic integrals exist since
\begin{align*}
\int_0^T\int_{\cO}\int_{\{|\sigma|\leq1\}}\|\sigma\delta_\xi\|_U^2\,\varrho(\dl\sigma)\,\dl\xi\,\dl s+
\int_0^T\int_{\cO}\int_{\{|\sigma|>1\}}\|\sigma\delta_\xi\|_U\,\varrho(\dl\sigma)\,\dl\xi\,\dl s<\infty
\end{align*}
due to our assumptions. The Lévy measure $\nu$ of $L$ is given by $\nu=(\lambda^d\otimes\varrho)\circ\tau^{-1}$, where $\tau\colon\cO\times\bR\to U$ is defined by $\tau(\xi,\sigma):=\sigma\delta_\xi$. Obviously, the process $L$ satisfies \eqref{eq:ass_nu_minimal} and \eqref{eq:L_char_fn}. 
The Poisson random measure on $[0,T]\times U$ associated to $L$ via \eqref{eq:LjumpPRM} has the representation
\[N(A)=\pi\big((\id_{[0,T]},\tau)^{-1}(A)\big),\quad A\in\cB([0,T]\times U),\]
where $(\id_{[0,T]},\tau)^{-1}(A):=\{(t,\xi,\sigma)\in[0,T]\times\cO\times\bR:(t,\sigma\delta_\xi)\in A\}$ is the preimage of $A$ under $(\id_{[0,T]},\tau)\colon[0,T]\times\cO\times\bR\to[0,T]\times U$.
Impulsive cylindrical processes are considered, e.g., in the monograph \cite{PesZab2007}.
\end{example}

\section{Weak approximation of linear SPDE with $\alpha$-stable noise}
\label{sec:alpha_stable}

Here we present a concrete application of the general theory developed in the previous sections. 
We analyze the weak convergence of space-time discretizations of the solutions to linear stochastic evolution equations of the type
\begin{align}\label{eq:linear_additive_noise}
\dl X(t)+AX(t)\,\dl t=\dl L(t),\quad t\in[0,T];\quad X(0)=X_0,
\end{align}
where 
$-A$ is the generator of an analytic semigroup of bounded operators on $H$
and $L$ is an infinite-dimensional Lévy process of $\alpha$-stable type, $\alpha\in(1,2)$, 
as specified in Section~\ref{sec:setting_alpha_stable} below.
For comparison's sake we first estimate the strong approximation error in Section~\ref{sec:strong_convergence}. Our main result, Theorem~\ref{thm:weak_error_alpha} in Section~\ref{sec:weak_convergence}, states that for suitable test functions the weak order of convergence is $\alpha$ times the strong order of convergence.
Let us recall the notation $B_U=\{x\in U:\|x\|_U\leq 1\}$ and $B_U^c=U\setminus B_U$ for the closed unit ball and its complement in a Hilbert space $U$.

\subsection{Setting, examples and regularity of the solution}
\label{sec:setting_alpha_stable}

We describe the assumptions on Eq.~\eqref{eq:linear_additive_noise} in detail and analyze the regularity of solution. As before, $H$ is a separable Hilbert space with inner product $\langle\cdot,\cdot\rangle$ and norm $\|\cdot\|$.

\begin{assumption}[Operator $A$]\label{as:A}
The operator $A\colon D(A)\subset H\to H$ is densely defined, linear, self-adjoint, positive definite and has a compact inverse. 
In particular, $-A$ is the generator of an analytic semigroup of contractions, which we denote by $(S(t))_{t\geq0}\subset\cL(H)$. 
\end{assumption}
Under Assumption~\ref{as:A} the fractional powers $A^{\frac \rho2}$, $\rho\in\bR$, of $A$ are defined
and there exist constants $C_\rho\in[0,\infty)$ (independent of $t$) such that
\begin{align}
  \big\|A^\frac{\rho}2S(t)\big\|_{\LB(H)}
&\leq
  C_\rho\, t^{-\frac{\rho}2},
  \quad
  t>0, \ \rho\geq0,\label{eq:estA1}\\
  \big\|A^{-\frac{\rho}2}(S(t)-\id_H)\big\|_{\LB(H)}
&\leq
  C_\rho \, t^{\frac{\rho}2},
  \quad \;\;
  t\geq0, \ \rho\in(0,2],\label{eq:estA2}
\end{align}
see, e.g.,~\cite[Section~2.6]{pazy1983}.
In particular, the operator $A$ gives rise to the scale of spaces $\dot H^\rho$, $\rho\in\R$, which we use to measure spatial regularity. These spaces are defined for $\rho\geq0$ as $\dot H^\rho:= D(A^{\frac{\rho}2})$ with norm 
$\|\cdot\|_{\dot H^\rho}:=\|A^{\frac{\rho}2}\cdot\|$ and for $\rho<0$ as the closure of $H$ w.r.t.\ the analogously defined
$\|\cdot\|_{\dot H^\rho}$-norm. 

\begin{example}\label{ex:A}
For $d\in\{1,2,3\}$ let $\cO\subset \bR^d$ be an open, bounded, convex, poly-gonal/polyhedral domain and set $H:=L^2(\cO)$. 
Our standard example for $A$ is a second order elliptic partial differential operator with zero Dirichlet boundary condition of the form
\begin{align*}
Au:=-\nabla\cdot(a\nabla u)+cu, \quad u\in D(A):=H^1_0(\cO)\cap H^2(\cO),
\end{align*}
with sufficiently smooth coefficients $a,c\colon\cO\to\bR$ such that $a(\xi)\geq\theta>0$ and $c(\xi)\geq 0$ for all $\xi\in\cO$. Here $H^1_0(\cO)$ and $H^2(\cO)$ are the classical $L^2$ Sobolev spaces of order one with zero Dirichlet boundary condition and of order two, respectively. It is well known that in this situation Assumption~\ref{as:A} is fulfilled and the abstract spaces $\dot H^{\rho}$ are related to the classical $L^2$-Sobolev spaces via, e.g, $\dot H^1=H^1_0(\cO)$ and $\dot H^2=H^1_0(\cO)\cap H^2(\cO)$.
\end{example}

The parameter $\beta>0$ in the following assumption on the driving Lévy process is a regularity parameter, compare Proposition~\ref{prop:existence_alpha} below.

\begin{assumption}[Lévy process $L$]\label{ass:L_alpha_stable}
The process $L=(L(t))_{t\in[0,T]}$ is a Hilbert-space valued Lévy process as described in Assumption~\ref{ass:L} (integrable, with mean zero and vanishing Gaussian part). In addition, there exist $\alpha\in(1,2)$ and $\beta>0$ such that the state space $U$ of $L$ is given by $U=\dot H^{\beta-\frac2\alpha}$ and the Lévy measure $\nu$ of $L$ safisfies the integrability condition
\begin{align}\label{eq:ass_nu_alphastable}
\int_{U}\min\big(\|x\|_U^{\alpha_+},\|x\|_U^{\alpha_-}\big)\,\nu(\dl x)<\infty
\end{align}
for all $\alpha_-<\alpha<\alpha_+$. 
\end{assumption}

It is clear that the condition \eqref{eq:ass_nu_alphastable} implies \eqref{eq:ass_nu_minimal}.  Assumption~\ref{ass:L_alpha_stable} is particularly satisfied by the following infinite-dimensional $\alpha$-stable processes.

\begin{example}[$\alpha$-stable subordinate cylindrical $Q$-Wiener process]
\label{ex:subordWP_alpha}
Consider the\linebreak situation of Example~\ref{ex:subordWP_general} and fix $\alpha\in(1,2)$, $\beta>0$. Assume that the covariance operator $Q\in\cL(H)$ of $W$ is such that $A^{\frac{\beta}2-\frac1\alpha}Q^{\frac 12}\in\cL_2(H)$ is a Hilbert-Schmidt operator on $H$, so that $W$ takes values in $U=\dot H^{\beta-\frac 2\alpha}$. Moreover, let the jump intensity measure $\varrho$ of the subordinator process $Z$ be given by
\begin{align*}
\varrho(\dl s)=\frac\alpha{\Gamma(1-\frac\alpha2)s^{1+\frac\alpha2}}\one_{(0,\infty)}(s)\,\dl s,
\end{align*}
so that $Z$ is $\alpha/2$-stable with Laplace transform $\bE e^{-rZ(t)}=\exp(-t\,r^{\frac\alpha2})$, $t,r\geq0$, compare \cite[Example~4.34]{PesZab2007}, \cite{BrzZab2010}. In this situation, subordinate Brownian motion $(L(t))_{t\in[0,T]}=(W(Z(t)))_{t\in[0,T]}$ satisfies Assumption~\ref{ass:L_alpha_stable}. The integrability condition \eqref{eq:ass_nu_alphastable} can be verified similarly as the condition \eqref{eq:ass_nu_minimal} in Example~\ref{ex:subordWP_general}. For instance, for $\alpha_-<\alpha$ we have
\begin{align*}
\int_{B_U^c}\|y\|_U^{\alpha_-}\,\nu(\dl y)
&=\int_0^\infty\int_U\one_{B_U^c}(\sqrt sy)\|\sqrt sy\|_U^{\alpha_-}\,\P_{W(1)}(\dl y)\varrho(\dl s)\\
&=\int_0^\infty s^{\frac{\alpha_-}2}\,\bE\big[\one_{(s^{-1/2},\infty)}(\|W(1)\|_U)\|W(1)\|_U^{\alpha_-}\big]\,\varrho(\dl s),
\end{align*}
where the expectation in the last integral can be estimated by
\begin{align*}
\int_{s^{-1/2}}^\infty\alpha_-\cdot r^{\alpha_--1}\P(\|W(1)\|_U\geq r)\,\dl r
&\leq\alpha_-\cdot s^{1-\frac{\alpha_-}2}\int_0^\infty\!r\, \P(\|W(1)\|_U\geq r)\,\dl r\\
&=\frac{\alpha_-}2\cdot s^{1-\frac{\alpha_-}2}\,\bE(\|W(1)\|_U^2),
\end{align*}
so that $\int_{B_U^c}\|y\|_U^{\alpha_-}\,\nu(\dl y)$ is less than or equal to
\begin{align*}
\frac{\alpha_-}2\int_0^1 s\,\varrho(\dl s)\,\bE(\|W(1)\|_U^2)+\int_1^\infty s^{\frac{\alpha_-}2}\,\varrho(\dl s)\,\bE(\|W(1)\|_U^{\alpha_+})<\infty;
\end{align*}
the finiteness of
$\int_{B_U}\|y\|^{\alpha_+}_U\,\nu(\dl y)$ can be verified similarly.
\end{example}

\begin{example}[$\alpha$-stable impulsive cylindrical process]\label{ex:impCylProc_alpha}
In the situation of Example~\ref{ex:impCylProc_general}, let $\alpha\in(1,2)$, assume that the spatial dimension $d$ satisfies $d<4/\alpha$, and let $\beta\in(0,\frac2\alpha-\frac d2)$. Choose 
$U=\dot H^{\beta-\frac2\alpha}$ 
\felix{argue as in Thomee p.38f and p.320f and use Sobolev embedding to obtain that $\dot H^{\frac2\alpha-\beta}$ is embedded in $C_b$}
as the state space of $L$ and
and let the jump size intensity measure $\varrho$ of $L$ be given by
\begin{align*}
\varrho(\dl\sigma)= \frac 1{\sigma^{1+\alpha}}\one_{(0,\infty)}(\sigma)\dl\sigma,
\end{align*}
compare~\cite{Mytnik2002}, \cite[Example~7.26]{PesZab2007}.
Then $L$ satisfies Assumption~\ref{ass:L_alpha_stable}. The integrability condition \eqref{eq:ass_nu_alphastable} holds since  
\begin{align*}
\int_{B_U}\|y\|_U^{\alpha_+}\,\nu(\dl y)&=\int_{0}^{1/c}\int_\cO\|\sigma\delta_\xi\|_U^{\alpha_+}\,\dl\xi\,\varrho(\dl\sigma)\\
&=c\, \lambda^d(\cO)\int_0^{1/c}\frac{\sigma^{\alpha_+}}{\sigma^{1+\alpha}}\,\dl\sigma <\infty
\end{align*}
for all $\alpha_+>\alpha$, where we have set $c:=\|\delta_\xi\|_U$;
the finiteness of $\int_{B_U^c}\|y\|_U^{\alpha_-}\,\nu(\dl y)$ is checked analogously.
\end{example}

We are interested in the mild solution $X(t)=S(t)X_0+\int_0^tS(t-s)\,\dl L(s)$ to Eq.~\eqref{eq:linear_additive_noise}. 
To simplify the exposition, we assume that the initial condition $X_0\in H$ is deterministic. 
In view of Lemma~\ref{lem:LvsN} it is natural to define the stochastic convolution $\int_0^tS(t-s)\dl L(s)$ in terms of integrals w.r.t.\ the compensated Poisson random measure $\tilde N$ associated to $L$ via \eqref{eq:LjumpPRM} and to define the mild solution $X=(X(t))_{t\in[0,T]}$ as the $H$-valued adapted process given by
\begin{equation}\label{eq:mild_sol_alpha_stable}
X(t)= S(t)X_0+\int_0^t\int_{B_U} S(t-s)x\,\tilde N(\dl s,\dl x)
+\int_0^t\int_{B_U^c} S(t-s)x\,\tilde N(\dl s,\dl x),
\end{equation}
where the first integral is an $H$-valued $L^{\alpha_+}$-integral w.r.t.\ $\tilde N$ and the second integral is an $H$-valued $L^{\alpha_-}$-integral w.r.t.\ $\tilde N$ in the sense of Remark~\ref{rem:stoch_int_space-time}. \felix{check measurability of integrands (general $A$)}
Under Assumption~\ref{as:A} and Assumption~\ref{ass:L_alpha_stable} this definition is meaningful.
In particular, as a consequence of \eqref{eq:estA1}, the operators $S(t)\in\LB(H),\,
t>0$, admit unique extensions $S(t)\in\LB(\dot H^{-\varrho},H),\,t>0,$ for all $\varrho\geq0$. As shown in the following result, the parameter $\beta>0$ in Assumption~\ref{ass:L_alpha_stable} determines the spatio-temporal regularity of $X$. Its proof is postponed to the appendix.

\begin{prop}[Regularity]\label{prop:existence_alpha}
Let Assumption~\ref{as:A} and \ref{ass:L_alpha_stable} hold. Let $X_0\in H$ and let $X=(X(t))_{t\in[0,T]}$ be the mild solution to Eq.~\eqref{eq:linear_additive_noise} given by \eqref{eq:mild_sol_alpha_stable}.
Let $\alpha_-\in[1,\alpha)$ and $\beta_-\in[0,\beta)$. Then, 
\begin{align*}
X(t)\in L^{\alpha_-}(\Omega;\dot H^{\beta_-})\;\text{ for all } t\in(0,T].
\end{align*}
Moreover, the mapping $t\mapsto X(t)$ is continuous from $(0,T]$ to $L^{\alpha_-}(\Omega;\dot H^{\beta_-})$ and $\min(\frac{\beta_-}2,1)$-Hölder continuous from $[\varepsilon,T]$ to $L^{\alpha_-}(\Omega;H)$ for all $\varepsilon\in(0,T]$. If additionally $X_0\in \dot H^\beta$, then $t\mapsto X(t)$ is continuous from $[0,T]$ to $L^{\alpha_-}(\Omega;\dot H^{\beta_-})$ and $\min(\frac{\beta_-}2,1)$-Hölder continuous from $[0,T]$ to $L^{\alpha_-}(\Omega;H)$.
\end{prop}


\begin{remark}\label{rem:Bochner_integrability_alpha}
The continuity assertions in Proposition~\ref{prop:existence_alpha} imply that the $H$-valued solution process $X=(X(t))_{t\in[0,T]}$ is stochastically continuous and hence has a predictable modification, see, e.g., \cite[Proposition~3.21]{PesZab2007}. In the sequel we always consider such a modification of $X$. 
Proposition~\ref{prop:existence_alpha} also implies that 
$X$ belongs to $L^1([0,T],\meas;L^{\alpha_-}(\Omega;H))$ 
for every finite Borel measure $\meas$ on $[0,T]$ and all $\alpha_-\in[1,\alpha)$. As a consequence we have
\begin{align*}
\bE\int_0^T \|X(t)\|\,\meas(\dl t)=\int_0^T\|X(t)\|_{L^1(\Omega;H)}\,\meas(\dl t)\leq \int_0^T\|X(t)\|_{L^{\alpha_-}(\Omega;H)}\,\meas(\dl t)<\infty,
\end{align*}
so that $\P$-almost all trajectories of $X$ belong to $L^1([0,T],\meas;H)$.
\end{remark}

\subsection{Discretization scheme and strong convergence}
\label{sec:strong_convergence}

We discretize the mild solution \eqref{eq:mild_sol_alpha_stable} to Eq.~\eqref{eq:linear_additive_noise} by combining an implicit Euler scheme in time with an abstract finite element discretization in space.

\begin{assumption}[Discretization]\label{ass:discretization}
For the spatial discretization we use a family
$(V_h)_{h\in(0,1)}$  of finite dimensional subspaces of $H$ and linear operators $A_h\colon V_h\to V_h$ that serve as discretizations of $A$. By $P_h\colon H\to V_h$ we denote the orthogonal projectors w.r.t.\ the inner product in $H$. For the discretization in time we use a uniform grid  $t_m=km$, $m\in\{0,\ldots,M\}$, with stepsize $k\in(0,1)$, where $M=M_k\in\bN$ is determined by $t_M\leq T< t_M+k$. 
The operators 
$S_{h,k}:=(\id_{V_h}+kA_h)^{-1}P_h$
serve as discretizations of $S(k)$,
and 
$E_{h,k}^m:=S_{h,k}^m-S(t_m)$
are the corresponding error operators. There are constants 
$D_\rho,\,D_{\rho,\sigma}\in [0,\infty)$ (independent of $h$, $k$, $m$) such that, 
\felix{cannot change order of operators if they are not self-adjoint -- changes ok for later sections?}
\begin{align}
&\big\|A_h^{\frac{\rho}2}S_{h,k}^m \big\|_{\LB(H)} + \big\|S_{h,k}^m A^{\frac{\min(\rho,1)}2} \big\|_{\LB(H)}
\leq
  D_\rho\,t_m^{-\frac{\rho}2},
  \quad
  \rho\geq0,\label{eq:estAh1}\\
&\big\| E_{h,k}^m A^{\frac{\rho}2}\big\|_{\LB(H)}
\leq
  D_{\rho,\sigma} \,t_m^{-\frac{\rho+\sigma}2}
  \big(
    h^\sigma + k^{\frac{\sigma}2}
  \big),
  \quad
  \sigma\in[0,2],\ \rho\in [-\sigma,\min(1,2-\sigma)],\label{eq:estAh2}
 \end{align}
for all $h,k\in(0,1)$ and $m\in\{1,\ldots,M\}$.
\end{assumption}

\begin{example}\label{ex:Vh}
Consider the situation of Example~\ref{ex:A}. A concrete example for the spaces $V_h$ are standard finite element spaces consisting of continuous, piecewise linear functions w.r.t.\ regular triangulations of $\cO$, with maximal mesh size bounded by $h$. A proof of the estimates \eqref{eq:estAh1}, \eqref{eq:estAh2} in this case can be found in \cite[Section 5]{AnderssonKruseLarsson}.
\end{example}

Due to $\eqref{eq:estAh1}$, the operators $S_{h,k}^{m-j}\in\LB(H,V_h)$ have unique continuous extensions from $\dot H^{-1}$ to $V_h$, which we also denote by $S_{h,k}^{m-j}$. We will henceforth assume that $\beta >\frac 2\alpha-1$ so that the Lévy process $L$ takes values in $U=\dot H^{\beta-\frac2\alpha}\subset\dot H^{-1}$ and the following definition makes sense.
For $h,k\in(0,1)$ and $M=M_k\in\bN$, the discretization $(X^m_{h,k})_{m\in\{0,\ldots,M\}}$ of $(X(t))_{t\in[0,T]}$  in space and time is defined by
\begin{equation}\label{eq:mild_solution_discrete_alpha_stable}
\begin{aligned}
X^m_{h,k}&= S^m_{h,k}X_0+\sum_{j=0}^{m-1}S^{m-j}_{h,k}(L(t_{j+1})-L(t_j))
\end{aligned}
\end{equation}
or, equivalently, 
\begin{align*}
X^m_{h,k}&= S^m_{h,k}X_0+\sum_{j=0}^{m-1}\Big(\int_{t_j}^{t_{j+1}}\!\!\int_{B_U}S^{m-j}_{h,k}x\,\tilde N(\dl s,\dl x)+\int_{t_j}^{t_{j+1}}\!\!\int_{B_U^c}S^{m-j}_{h,k}x\,\tilde N(\dl s,\dl x)\Big)
\end{align*}
where the integrals involving the unit ball $B_U$ are $V_h$-valued $L^{\alpha_+}$-integrals w.r.t.\ $\tilde N$ and the integrals involving $B_U^c$ are $V_h$-valued  $L^{\alpha_-}$-integrals w.r.t.\ $\tilde N$.
We denote further by $(\tilde X_{h,k}(t))_{t\in[0,T]}$ the piecewise constant interpolation of $(X^m_{h,k})_{m\in\{0,\ldots,M\}}$ defined as
\begin{align}\label{eq:mild_solution_discrete_interpol_alpha_stable}
\tilde X_{h,k}(t):=\sum_{m=0}^{M-1}\one_{[t_m,t_{m+1})}(t) X^m_{h,k}+\one_{[t_M,T]}(t)X^M_{h,k}.
\end{align}
Using the notation $\lfloor t\rfloor_k:=\max\{n\in\bN_0:nk\leq t\}$ we can rewrite $\tilde X_{h,k}(t)$ more conveniently as $X_{h,k}^{\lfloor t\rfloor_k}$, hence \felix{see notes 11.10.'16}
\begin{equation}\label{eq:mild_solution_discrete_interpol_alpha_stable2}
\begin{aligned}
\tilde X_{h,k}(t)
&=S_{h,k}^{\lfloor t\rfloor_k}X_0+\int_0^{k\cdot\lfloor t\rfloor_k}\int_{B_U}S_{h,k}^{\lfloor t\rfloor_k-\lfloor s\rfloor_k}x\,\tilde N(\dl s,\dl x)\\
&\quad +\int_0^{k\cdot\lfloor t\rfloor_k}\int_{B_U^c}S_{h,k}^{\lfloor t\rfloor_k-\lfloor s\rfloor_k}x\,\tilde N(\dl s,\dl x),
\end{aligned}
\end{equation}
where the stochastic integrals are again understood in the $L^{\alpha_+}$ sense and in the $L^{\alpha_-}$ sense, respectively.
 
Finally, for convenience we also introduce the piecewise continuous error mapping $\tilde E_{h,k}\colon [0,T]\to\LB(H)$ given by $\tilde E_{h,k}(0):=P_h-\id_H$ and $\tilde E_{h,k}(t):=S_{h,k}^m-S(t)$ for $t\in(t_{m-1},t_m]$, so that
\begin{equation}\label{eq:error_alpha_stable}
\begin{aligned}
X_{h,k}^m-X(t_m)
&=E_{h,k}^m X_0+\int_0^{t_m}\int_{B_U} \tilde E_{h,k}(t_m-s)x\,\tilde N(\dl s,\dl x)\\
&\quad+\int_0^{t_m}\int_{B_U^c} \tilde E_{h,k}(t_m-s)x\,\tilde N(\dl s,\dl x),
\end{aligned}
\end{equation} 
the stochastic integrals being understood in the $L^{\alpha_+}$ sense and in the $L^{\alpha_-}$ sense, respectively.

\begin{remark}
\felix{see notes 14.10.'16}
The error estimate \eqref{eq:estAh2} extends to the piecewise continuous error mapping $\tilde E_{h,k}\colon[0,T]\to\LB(H)$. Indeed, as a consequence of the identity $\tilde E_{h,k}(t)=E_{h,k}^m+(S(t_m)-S(t))$, $t\in(t_{m-1},t_m]$, and the estimates \eqref{eq:estA1}, \eqref{eq:estA2}, \eqref{eq:estAh2}, we have
\begin{align}\label{eq:estAh3}
\big\|\tilde E_{h,k}(t)A^{\frac\rho2}\big\|_{\LB(H)}\leq (C_\sigma C_{\sigma+\rho}+D_{\rho,\sigma})\,t^{-\frac{\rho+\sigma}2}\big(h^\sigma+k^{\frac\sigma2}\big),
\end{align}
holding for $\sigma\in[0,2]$, $\rho\in [-\sigma,\min(1,2-\sigma)]$ and $h,k\in(0,1)$, $t\in(0,T]$.
\end{remark}

For comparison's sake we present the following strong convergence result. The proof of which is postponed to the appendix.

\begin{prop}[Strong order]\label{prop:strong_error_alpha_stable}
Let Assumption~\ref{as:A}, \ref{ass:L_alpha_stable} and \ref{ass:discretization} hold with $\beta\in(\frac2\alpha-1,\frac2\alpha]$. Let $X_0\in \dot H^\beta$,
let $(X(t))_{t\in[0,T]}$ be the mild solution to Eq.~\eqref{eq:linear_additive_noise} given by \eqref{eq:mild_sol_alpha_stable}, and $(\tilde X_{h,k}(t))_{t\in[0,T]}$ be its discretization given by \eqref{eq:mild_solution_discrete_alpha_stable}, \eqref{eq:mild_solution_discrete_interpol_alpha_stable}.
Then, for all $\alpha_-\in[1,\alpha)$ and $\beta_-\in[0,\beta)$  there exists a finite constant $C=C(X_0,T,\nu,\alpha,\alpha_-,\beta,\beta_-)$ such that 
\begin{align*}
\sup_{t\in[0,T]}\big\|\tilde X_{h,k}(t)-X(t)\big\|_{L^{\alpha_-}(\Omega;H)}\leq C\big(h^{\beta_-}+k^{\frac{\beta_-}2}\big),\quad h,k\in (0,1).
\end{align*}
\end{prop}

\begin{remark}
The restriction $\beta\in(\frac2\alpha-1,\frac2\alpha]$ for the regularity parameter $\beta$ in Proposition~\ref{prop:strong_error_alpha_stable} and in Theorem~\ref{thm:weak_error_alpha} below is made for the following two reasons: As mentioned above, the lower bound $\beta>\frac2\alpha-1$ ensures that the discrete solution \eqref{eq:mild_solution_discrete_alpha_stable} is defined. The upper bound $\beta\leq\frac2\alpha$ is needed to be able to apply the deterministic error estimates \eqref{eq:estAh1}, \eqref{eq:estAh2} in the proofs of our results. The restriction on the range of admissible regularity parameters $\beta$ can be relaxed if one considers higher order finite element spaces $V_h$ instead of the `second order spaces' in Assumption~\ref{ass:discretization}.
\end{remark}

\subsection{Weak order of convergence}
\label{sec:weak_convergence}

We now analyze the weak error $\bE\big(f(\tilde X_{h,k})-f(X)\big)$, where $f$ belongs to the following class of path-dependent test functions.
Recall the definition of the spaces $\cC^{1,\delta}(H,\bR)$ from Subsection~\ref{sec:Notation_and_conventions}.
\begin{assumption}[Test function]\label{as:Phi_alpha}
Let $\varphi\in\cC^{1,\alpha_--1}(H,\R)\cap \cC^{1,\alpha_+-1}(H,\R)$ for \linebreak 
some $1\leq\alpha_-<\alpha<\alpha_+\leq2$, where $\alpha\in(1,2)$ is as in Assumption~\ref{ass:L_alpha_stable}.
Let $\meas$ be a finite Borel measure on $[0,T]$. The functional 
$f\colon L^1([0,T],\meas;H) \to \R$ is given by
\begin{align*}
 f(x)=
  \varphi
  \Big(
    \int_{[0,T]}
      x(t)
    \,\meas(\diffin t)
  \Big),\quad x\in L^1([0,T],\meas;H).
\end{align*}
\end{assumption}

Let us discuss Assumption \ref{as:Phi_alpha} and give a concrete example for $f$.

\begin{remark}\label{rem:condition_varphi_alpha}
\felix{cf.~notes~21.5.'16}
Consider $\varphi\colon H\to\bR$ and $1\leq\alpha_-<\alpha<\alpha_+\leq2$.
\begin{enumerate}[(i)]
\item 
A sufficient condition for $\varphi$ belonging to $\cC^{1,\alpha_--1}(H,\R)\cap \cC^{1,\alpha_+-1}(H,\R)$ is the following: 
$\varphi\in\cC^2(H,\bR)$ and $\varphi''$ satisfies the growth bound $\|\varphi''(x)\|_{\LB(H)}\leq C\min(\|x\|^{\alpha_--2},\|x\|^{\alpha_+-2})$, $x\in H$, with a finite constant $C$ that does not depend on $x$. This condition is natural in view of the typical assumptions in the Gaussian case (see, e.g., \cite{AnderssonKruseLarsson}) and the limited integrability properties of $\alpha$-stable random variables.
\item 
The assumption $\varphi\in\cC^{1,\alpha_--1}(H,\R)\cap \cC^{1,\alpha_+-1}(H,\R)$ is equivalent to assuming that $\varphi$ is Fréchet differentiable and that there exists a constant $C\in[0,\infty)$ such that $\|\varphi'(x)-\varphi'(y)\|\leq C\|x-y\|^{\alpha_+-1}$ for all $x,y\in H$ with $\|x-y\|\leq 1$ and $\|\varphi'(x)-\varphi'(y)\|\leq C\|x-y\|^{\alpha_--1}$ for all $x,y\in H$ with $\|x-y\|\geq1$.
\item
The assumption $\varphi\in \cC^{1,\alpha_--1}(H,\bR)$ implies the growth condition $|\varphi(x)|\leq C(1+\|x\|^{\alpha_-})$, $x\in H$, with $C\in[0,\infty)$ independent of $x$.
\end{enumerate}
\end{remark}

\begin{example}
Let the situation of Example~\ref{ex:A} and \ref{ex:Vh} be given, where $H=L^2(\cO)$. Assumption~\ref{as:Phi_alpha} is particularly satisfied by local space or space-time averages of the form
\begin{align*}
f(x):=\frac{1}{\lambda^d(\cD)}\big\langle x(\tau_1),\one_\cD\big\rangle
\quad\text{ or }\quad
f(x):=\frac{1}{(\tau_1-\tau_0)\lambda^d(\cD)}\Big\langle\int_{\tau_0}^{\tau_1}x(t)\,\dl t,\,\one_\cD\Big\rangle,
\end{align*}
where $0\leq \tau_0 <\tau_1\leq T$, $\cD\subset\cO$ is a Borel subset with positive Lebesgue measure $\lambda^d(\cD)$,  
and $x\in L^1([0,T],\meas;L^2(\cO))$ for $\meas=\delta_{\tau_1}$ and $\meas=\one_{[\tau_0,\tau_1]}\cdot\lambda$, respectively.
\end{example}

Recall from Remark~\ref{rem:Bochner_integrability_alpha} that, under Assumption~\ref{as:A} and \ref{ass:L_alpha_stable}, $\P$-almost all trajectories of the mild solution $X$ are Bochner integrable w.r.t.\ any finite Borel measure $\meas$ on $[0,T]$.  Hence, under Assumption~\ref{as:A}, \ref{ass:L_alpha_stable} and \ref{as:Phi_alpha}, the real-valued random variable $f(X)$ is well defined. Moreover, Remark~\ref{rem:Bochner_integrability_alpha} and \ref{rem:condition_varphi_alpha}(iii) together with Minkowski's inequality for integrals imply that $f(X)$ is integrable:
\begin{align*}
\bE|f(X)|\
&\leq C\bE\Big[1+\Big\|\int_{[0,T]}X(t)\,\meas(\dl t)\Big\|^{\alpha_-}\Big]\\
&\leq C\Big(1+\bE\Big[\Big(\int_{[0,T]}\|X(t)\|\,\meas(\dl t)\Big)^{\alpha_-}\Big]\Big)\\
&\leq C\Big(1+\Big(\int_{[0,T]}\|X(t)\|_{L^{\alpha_-}(\Omega;H)}\,\meas(\dl t)\Big)^{\alpha_-}\Big)<\infty.
\end{align*}
Obviously, the same is true for the real-valued random variable $f(\tilde X_{h,k})$.

Here is the main result of this section. Note that the obtained convergence rate for the weak error $\bE\big(f(\tilde X_{h,k})-f(X)\big)$ is $\alpha$ times that of the strong error $\sup_{t\in[0,T]}\big\|\tilde X_{h,k}(t)-X(t)\big\|_{L^{\alpha_-}(\Omega;H)}$ from Proposition~\ref{prop:strong_error_alpha_stable}.

\begin{theorem}[Weak order]\label{thm:weak_error_alpha}
Let Assumption~\ref{as:A}, \ref{ass:L_alpha_stable}, \ref{ass:discretization} and \ref{as:Phi_alpha} hold with $\beta\in(\frac2\alpha-1,\frac2\alpha]$. Let $X_0\in \dot H^{\alpha\beta}$,
let $(X(t))_{t\in[0,T]}$ be the mild solution to Eq.~\eqref{eq:linear_additive_noise} given by \eqref{eq:mild_sol_alpha_stable}, and $(\tilde X_{h,k}(t))_{t\in[0,T]}$ be its discretization given by \eqref{eq:mild_solution_discrete_alpha_stable}, \eqref{eq:mild_solution_discrete_interpol_alpha_stable}.
Then, for all $\beta_-\in[0,\beta)$  there exists a finite constant $C=C(X_0,T,\nu,\meas,f,\alpha,\beta,\beta_-)$ such that 
\felix{(statement to be adapted)}
\begin{align*}
\big|\bE\;\!f(\tilde X_{h,k})-\bE\;\!f(X)\big|\leq C\big(h^{\alpha\beta_-}+k^{\frac{\alpha\beta_-}2}\big),\quad h,k\in(0,1).
\end{align*}
\end{theorem}

In the sequel we will often omit the explicit notation of the discretization parameters $k$, $h$ and write, e.g., $\tilde X(t)$ instead of $\tilde X_{h,k}(t)$. In the proof of Theorem~\ref{thm:weak_error_alpha} we will deal with the $H$-valued random variable 
\begin{align}\label{eq:def_Fhk_alpha}
F:=F_{h,k}:=\int_0^1\varphi'\Bigg(\int_{[0,T]} X(t)\,\meas(\dl t)+\theta\int_{[0,T]}(\tilde X_{h,k}(t)-X(t))\,\meas(\dl t)\Bigg)\,\dl \theta.
\end{align}
Several integrability and Malliavin regularity properties of $F$ are collected in the following lemma.

\begin{lemma}\label{lem:F_DF_alpha}
In the setting of Theorem~\ref{thm:weak_error_alpha}, let $F=F_{h,k}$ be  the $H$-valued random variable given by \eqref{eq:def_Fhk_alpha}. Let $1<\alpha_-<\alpha<\alpha_+<2$ 
\felix{(check! $\leq$ vs. $<$)} 
be as in Assumption~\ref{as:Phi_alpha} and $\alpha_-'=\frac{\alpha_-}{\alpha_--1}$, $\alpha_+'=\frac{\alpha_+}{\alpha_+-1}$ be the corresponding dual exponents. The following assertions hold:
\begin{enumerate}[(i)]
\item 
$F\in L^{\alpha_-'}(\Omega;H)$ 
and 
$\sup\limits_{h,k\in(0,1)}\|F\|_{L^{\alpha_-'}(\Omega;H)}<\infty$,
\item 
$\one_{\Omega_T\times B_U}D F\in L^{\alpha_+'}(\Omega_T\times U;H)$ 
and 
$\sup\limits_{h,k\in(0,1)}\|\one_{\Omega_T\times B_U}D F\|_{L^{\alpha_+'}(\Omega_T\times U;H)}<\infty$,
\item
$\one_{\Omega_T\times B_U^c}D F\in L^{\alpha_-'}(\Omega_T\times U;H)$
and 
$\sup\limits_{h,k\in(0,1)}\|\one_{\Omega_T\times B_U^c}D F\|_{L^{\alpha_-'}(\Omega_T\times U;H)}<\infty$.
\end{enumerate}
\end{lemma}

\begin{proof}
Let us set $Y:=\int_{[0,T]}X(t)\meas(\dl t)$, $\tilde Y:=\int_{[0,T]}\tilde X(t)\meas(\dl t)$. Assertion (i) follows from the assumption that $\varphi\in\cC^{1,\alpha_--1}(H,\bR)$, which implies the growth bound $\|\varphi'(x)\|\leq C(1+\|x\|^{\alpha_--1})$  with $C\in (0,\infty)$ independent of $x\in H$. Indeed, we have
\begin{align*} 
\|F\|_{L^{\alpha_-'}(\Omega;H)}
&\leq \int_0^1\big\|\varphi'\big(Y+\theta(\tilde Y-Y)\big)\big\|_{L^{\alpha_-'}(\Omega;H)}\dl\theta\\
&\leq \sup_{\theta\in[0,1]}C\,\big\|1+\|Y+\theta(\tilde Y-Y)\|^{\alpha_--1}\big\|_{L^{\alpha_-'}(\Omega;\bR)}\\
&\leq C\,\big\|1+\|Y\|^{\alpha_--1}+\|\tilde Y-Y\|^{\alpha_--1}\big\|_{L^{\alpha_-'}(\Omega;\bR)}\\
&\leq C\,\big(1+\|Y\|_{L^{\alpha_-}(\Omega;H)}^{\alpha_--1}+\|\tilde Y-Y\|_{L^{\alpha_-}(\Omega;H)}^{\alpha_--1}\big).
\end{align*}
The latter term is finite as a consequence of Remark~\ref{rem:Bochner_integrability_alpha} and Minkowski's integral inequality. Using also the strong convergence from Proposition~\ref{prop:strong_error_alpha_stable} we see that it is even uniformly bounded in $h,k\in(0,1)$.

Next, we verify the assertion (ii). \felix{see notes 11.10.'16} 
Applying the chain rule from Proposition~\ref{lem:chain} to $(Y,\tilde Y)$ and the function $h\colon H\times H\to H$, $(y,\tilde y)\mapsto\int_0^1\varphi'((1-\theta)y+\theta \tilde y)\,\dl\theta$, we obtain
\begin{align*}
D_{s,x}F=\int_0^1\Big(\varphi'\big((1-\theta)(Y+D_{s,x}Y)+\theta (\tilde Y+D_{s,x}\tilde Y)\big)-\varphi'\big((1-\theta)Y+\theta\tilde Y\big)\Big)\,\dl\theta,
\end{align*}
$(s,x)\in[0,T]\times U$.
Using Lemma~\ref{lem:intDX} and the commutator relation from Proposition~\ref{prop:comm_space-time2} together with \eqref{eq:mild_sol_alpha_stable} and \eqref{eq:mild_solution_discrete_interpol_alpha_stable2}, we have
\begin{align*}
D_{s,x}Y=\int_{[s,T]}S(t-s)x\,\meas(\dl t),\quad D_{s,x}\tilde Y=\int_{[k\lfloor s\rfloor_k+1),T]}S_{h,k}^{\lfloor t\rfloor_k-\lfloor s\rfloor_k}x\,\meas(\dl t).
\end{align*}
Note that the $H$-valued Bochner integrals w.r.t.\ $\meas$ are defined for $\lambda\otimes \nu$-almost all $(s,x)\in[0,T]\times U$ according to Lemma~\ref{lem:intDX}.
Now we can use the assumption $\varphi\in\cC^{1,\alpha_+-1}(H,\bR)$, Jensen's inequality, and the identity $(\alpha_+-1)\cdot\alpha_+'=\alpha_+$ to estimate
\begin{align*}
&\bE\int_0^T\int_{B_U}\|D_{s,x}F\|^{\alpha_+'}\,\nu(\dl x)\,\dl s\\
&\leq C \sup_{\theta\in[0,T]}\bE\int_0^T\int_{B_U}\big\|(1-\theta)D_{s,x}Y+\theta D_{s,x}\tilde Y\big\|^{(\alpha_+-1)\cdot\alpha_+'}\,\nu(\dl x)\,\dl s\\
&\leq C\,\bE\int_0^T\int_{B_U}\Big(\|D_{s,x}Y\|^{\alpha_+}+\|D_{s,x}\tilde Y\|^{\alpha_+}\Big)\,\nu(\dl x)\,\dl s\\
&\leq C\!\int_{[0,T]}\int_0^T\!\int_{B_U}\!\Big(\one_{s\leq t}\,\big\|S(t-s)x\big\|^{\alpha_+}+\one_{s\leq k\lfloor t\rfloor_k}\big\|S_{h,k}^{\lfloor r\rfloor_k-\lfloor s\rfloor_k}x\big\|^{\alpha_+}\Big)\,\nu(\dl x)\,\dl s\,\meas(\dl t)\\
&\leq C\!
\int_{[0,T]}\!\Big(\int_0^t\big\|S(t-s)A^{\frac1\alpha-\frac\beta2}\big\|_{\LB(H)}\,\dl s + \int_0^{k\lfloor t\rfloor_k}\!\!\big\|S_{h,k}^{\lfloor t\rfloor_k-\lfloor s\rfloor_k}A^{\frac1\alpha-\frac\beta2}\big\|_{\LB(H)}\,\dl s\Big)\,\meas(\dl t).
\end{align*}
For the last inequality we have used the identity $\|x\|_U=\|A^{\frac\beta2-\frac1\alpha}x\|$, $x\in U$, and the integrability assumption \eqref{eq:ass_nu_alphastable}. The proof of assertion (ii) is finished by applying the estimates \eqref{eq:estA1} and \eqref{eq:estAh1}.

Assertion (iii) can be verified analogously to the proof of assertion (ii), using the assumption that $\varphi\in\cC^{1,\alpha_--1}(H,\bR)$.
\end{proof}

\begin{proof}[Proof of Theorem~\ref{thm:weak_error_alpha}]
\felix{see notes 19.4.'16 - 17.5.'16}
By a simple application of the fundamental theorem of calculus, the weak approximation error can be rewritten as
\begin{equation}\label{eq:proof_weak_error_alpha_1}
\begin{aligned}
&\big|\bE\big(f(\tilde X)-f(X)\big)\big|\\
&=\Big|\bE\Big\langle F,\int_{[0,T]}\big(\tilde X(t)-X(t)\big)\,\meas(\dl t)\Big\rangle\Big|\\
&=\Big|\int_{[0,T]}\bE\big\langle F,\tilde X(t)-X(t)\big\rangle\,\meas(\dl t)\Big|\\
&\leq\Big|\int_{[0,T]}\bE\big\langle F,X_{h,k}^{\lfloor t\rfloor_k}-X(k\lfloor t\rfloor_k)\big\rangle\,\meas(\dl t)\Big|
+\Big|\int_{[0,T]}\bE\big\langle F,X(k\lfloor t\rfloor_k)-X(t)\big\rangle\,\meas(\dl t)\Big|\\
&=:I+I\!I,
\end{aligned}
\end{equation}
with $F=F_{h,k}\in L^{\alpha_-'}(\Omega;H)$ given by \eqref{eq:def_Fhk_alpha}. We estimate $I$ and $II$ separately.

Concerning the term $I$ in \eqref{eq:proof_weak_error_alpha_1} it is enough to show that there exists a finite constant $C=C(X_0,\meas,f,\alpha,\beta,\beta_-)$, which does not depend on $h,k\in(0,1)$ or $m\in\{0,1,\ldots,M\}$, such that
\begin{equation*}
\big|\bE\big\langle F,X_{h,k}^m-X(t_m)\big\rangle\big|\leq C\big(h^{\alpha\beta_-}+k^{\frac{\alpha\beta_-}2}\big).
\end{equation*}
According to \eqref{eq:error_alpha_stable}, we have
\begin{equation}\label{eq:proof_weak_error_alpha_3}
\begin{aligned}
\big|\bE\big\langle F,X_{h,k}^m-X(t_m)\big\rangle\big|
&\leq\big|\bE\big\langle F, E_{h,k}^mX_0\big\rangle\big|\\
&\quad+\Big|\bE\Big\langle F,\int_0^{t_m}\int_{B_U}\tilde E_{h,k}(t_m-s)x\,\tilde N(\dl s,\dl x)\Big\rangle\Big|\\
&\quad +\Big|\bE\Big\langle F,\int_0^{t_m}\int_{B_U^c}\tilde E_{h,k}(t_m-s)x\,\tilde N(\dl s,\dl x)\Big\rangle\Big|\\
&\quad=:I_a+I_b+I_c.
\end{aligned}
\end{equation}
Estimate \eqref{eq:estAh2} with $\sigma=\alpha\beta\in(1,2]$ and $\rho=-\sigma=-\alpha\beta$ implies
\begin{align*}
I_a
\leq\|F\|_{L^{1}(\Omega;H)}\|X_0\|_{\dot H^{\alpha\beta}}\|E_{h,k}^m A^{-\frac{\alpha\beta}2}\|_{\LB(H)}\leq C\big(h^{\alpha\beta}+k^{\frac{\alpha\beta}2}\big),
\end{align*}
$C\in(0,\infty)$ being independent of $h,k$ and $m$.
To estimate the second term on the right hand side of \eqref{eq:proof_weak_error_alpha_3}, we apply the local duality formula from Proposition~\ref{prop:duality_space_time_local} with $p=\alpha_+$, $\Phi(s,x):=\one_{[0,t_m]\times B_U}(s,x)\tilde E_{h,k}(t_m-s)x$, and use Hölder's inequality:
\begin{equation*}
\begin{aligned}
I_b
&= \Big|\bE\int_0^{t_m}\int_{B_U}\big\langle DF,\tilde E_{h,k}(t_m-s)x\big\rangle\,\nu(\dl x)\,\dl s\Big|\\
&\leq \bE\int_0^{t_m}\Big(\int_{B_U}\|D_{s,x}F\|^{\alpha_+'}\,\nu(\dl x)\Big)^{\frac1{\alpha_+'}}\Big(\int_{B_U}\|\tilde E_{h,k}(t_m-s)x\|^{\alpha_+}\,\nu(\dl x)\Big)^{\frac1{\alpha_+}}\,\dl s.
\end{aligned}
\end{equation*}
Note that Proposition~\ref{prop:duality_space_time_local} is applicable due to Lemma~\ref{lem:F_DF_alpha} and the fact that $\int_0^{t_m}\int_{B_U}\|\tilde E_{h,k}(t_m-s)x\|^{\alpha_+}\,\nu(\dl x)\,\dl s$ is finite, as seen in the proof of Proposition~\ref{prop:strong_error_alpha_stable}. 
Using the assumption $\varphi\in\cC^{1,\alpha_+-1}(H;\bR)$, a similar argumentation as in the proof of Lemma~\ref{lem:F_DF_alpha}(ii) yields, for $\lambda$-almost all $s\in[0,T]$, 
\felix{see notes 12.10.'16}
\begin{equation}\label{eq:proof_weak_error_alpha_5}
\begin{aligned}
&\Big(\int_{B_U}\|D_{s,x}F\|^{\alpha_+'}\nu(\dl x)\Big)^{\frac1{\alpha_+'}}\\
&\leq C\int_{[s,T]}\Big(\int_{B_U}\|S(t-s)x\|^{\alpha_+}\nu(\dl x)\Big)^{\frac1{\alpha_+'}}\,\meas(\dl t)\\
&\quad+C\int_{[k\lfloor s\rfloor_k+1,T]}\Big(\int_{B_U}\big\|S_{h,k}^{\lfloor t\rfloor_k-\lfloor s\rfloor_k}x\big\|^{\alpha_+}\nu(\dl x)\Big)^{\frac1{\alpha_+'}}\,\meas(\dl t)\\
&\leq C\,\Big(\int_{B_U}\|x\|_U^{\alpha_+}\nu(\dl x)\Big)^{\frac1{\alpha_+'}}
\int_{[0,T]}\Big(\one_{s\leq t}\cdot\big\|S(t-s)A^{\frac1\alpha-\frac\beta2}\big\|_{\LB(H)}^{\alpha_+-1}\\
&\quad + \one_{s\leq k\lfloor t\rfloor_k}\cdot\big\|S_{h,k}^{\lfloor t\rfloor_k-\lfloor s\rfloor_k}A^{\frac1\alpha-\frac\beta2}\big\|_{\LB(H)}^{\alpha_+-1}\Big)\,\meas(\dl t).
\end{aligned}
\end{equation}
As a consequence, we have
\begin{equation*}
\begin{aligned}
I_b
&\leq 
C\,\Big(\int_{B_U}\|x\|_U^{\alpha_+}\nu(\dl x)\Big)^{\frac1{\alpha_+'}+\frac1{\alpha_+}}
\int_{[0,T]}\int_0^{t_m}
\Big(\one_{s\leq t}\cdot\big\|S(t-s)A^{\frac1\alpha-\frac\beta2}\big\|_{\LB(H)}^{\alpha_+-1}\\
&\quad + \one_{s\leq k\lfloor t\rfloor_k}\cdot\big\|S_{h,k}^{\lfloor t\rfloor_k-\lfloor s\rfloor_k}A^{\frac1\alpha-\frac\beta2}\big\|_{\LB(H)}^{\alpha_+-1}\Big)\cdot\big\|\tilde E_{h,k}(t_m-s)A^{\frac1\alpha-\frac\beta2}\big\|_{\LB(H)}
\,\dl s\,\meas(\dl t)\\
&\leq C \int_{[0,T]}\int_0^{\min(t_m,k\lfloor t\rfloor_k)}
\Big(\big\|S(t-s)A^{\frac1\alpha-\frac\beta2}\big\|_{\LB(H)}^{\alpha_+-1}\\
&\quad + \big\|S_{h,k}^{\lfloor t\rfloor_k-\lfloor s\rfloor_k}A^{\frac1\alpha-\frac\beta2}\big\|_{\LB(H)}^{\alpha_+-1}\Big)\cdot\big\|\tilde E_{h,k}(t_m-s)A^{\frac1\alpha-\frac\beta2}\big\|_{\LB(H)}
\,\dl s\,\meas(\dl t)
\end{aligned}
\end{equation*}
Applying the estimates \eqref{eq:estA1}, \eqref{eq:estAh1} and \eqref{eq:estAh3} with $\rho:=\frac2\alpha-\beta\in[0,1)$ and $\sigma\in[0,2]$ to be determined below, the last expression can be bounded from above by
\felix{see notes 14.10.'16}
\begin{align*}
&C \int_{[0,T]}\int_0^{\min(t_m,k\lfloor t\rfloor_k)}
\Big((t-s)^{-\frac\rho2(\alpha_+-1)}+(k\lfloor t\rfloor_k-k\lfloor s\rfloor_k)^{-\frac\rho2(\alpha_+-1)}\Big)\\
&\quad\times(t_m-s)^{-\frac{\rho+\sigma}2}(h^\sigma+k^{\frac\sigma2})
\;\dl s\,\meas(\dl t)\\
&\leq C\,\meas([0,T])\int_0^T(T-s)^{-\frac\rho2(\alpha_+-1)-\frac{\rho+\sigma}2}\,\dl s\, (h^\sigma+k^{\frac\sigma2})
\end{align*}
The last integral is finite if, and only if, $\sigma<2-\rho\alpha_+$. This condition also implies that $\rho<2-\sigma$, so that the application of $\eqref{eq:estAh3}$ is justified for $\sigma\in[0,2-(\frac2\alpha-\beta)\alpha_+)\subset[0,2]$. In particular, we can use $\sigma:=\alpha\beta_-$ if $\alpha_+\in(\alpha,2)$ is chosen small enough.
The third term $I_c$ on the right hand side of \eqref{eq:proof_weak_error_alpha_3} can be estimated analogously to $I_b$ if one considers the integrability exponent $\alpha_-$ instead of $\alpha_+$ and uses the assumption $\varphi\in\cC^{1,\alpha_--1}(H,\bR)$.

In order to handle the term $I\!I$ in \eqref{eq:proof_weak_error_alpha_1} it suffices to show that for $t\in[t_m,t_{m+1})$, $m\in\{0,\ldots,M-1\}$ or $t\in[t_m,T]$, $m=M$, and $k\in(0,1)$
\begin{align*}
\big|\bE\big\langle F,X(t_m)-X(t)\big\rangle\big|\leq C\,k^{\frac{\alpha\beta_-}2}
\end{align*}
with a finite constant $C=C(X_0,T,\meas,f,\alpha,\beta,\beta_-)$ that does not depend on $t,m$ or $k$. Fix $m$ and $t$ as above. According to \eqref{eq:mild_sol_alpha_stable} we may write
\begin{equation}\label{eq:proof_weak_error_alpha_5}
\begin{aligned}
&\big|\bE\big\langle F,X(t_m)-X(t)\big\rangle\big|
\,\leq\,
\big|\bE\big\langle F,(S(t_m)-S(t))X_0\big\rangle\big|\\
&\quad+\Big|\bE\Big\langle F,\int_0^{t_m}\int_{B_U}(S(t_m-s)-S(t-s))x\,\tilde N(\dl s,\dl x)\Big\rangle\Big|\\
&\quad+\Big|\bE\Big\langle F,\int_0^{t_m}\int_{B_U^c}(S(t_m-s)-S(t-s))x\,\tilde N(\dl s,\dl x)\Big\rangle\Big|\\
&\quad+\Big|\bE\Big\langle F,\int_{t_m}^t\int_{B_U}S(t-s)x\,\tilde N(\dl s,\dl x)\Big\rangle\Big|
+\Big|\bE\Big\langle F,\int_{t_m}^t\int_{B_U^c}S(t-s)x\,\tilde N(\dl s,\dl x)\Big\rangle\Big|\\
&=:I\!I_a+I\!I_{b}+I\!I_{c}+I\!I_{d}+I\!I_{e}.
\end{aligned}
\end{equation}
As a consequence of the estimate \eqref{eq:estA2} with $\rho=\alpha\beta\in(1,2]$,
\begin{align*}
I\!I_a&\leq \|F\|_{L_1(\Omega;H)}\|X_0\|_{\dot H^{\alpha\beta}}\|S(t_m)\|_{\LB(H)}\big\|(\id_H-S(t-t_m))A^{-\frac{\alpha\beta}2}\big\|_{\LB(H)}\\
&\leq C\,(t-t_m)^{\frac{\alpha\beta}2}\,\leq\, C\,k^{\frac{\alpha\beta}2},
\end{align*}
$C\in[0,\infty)$ being independent of $m$, $t$ and $k$. 
\felix{$C=C(T)$ here}
Concerning the second term on the right hand side of \eqref{eq:proof_weak_error_alpha_3}, 
we apply the local duality formula from Proposition~\ref{prop:duality_space_time_local} with $p=\alpha_+$ and $\Phi(s,x):=\one_{[0,t_m]\times B_U}(s,x)(S(t_m-s)-S(t-s))x$, Hölder's inequality, and the estimate \eqref{eq:proof_weak_error_alpha_5} to obtain:
\felix{see notes 15.10.'16}
\begin{equation*}
\begin{aligned}
I\!I_b
&= \Big|\bE\int_0^{t_m}\int_{B_U}\big\langle DF,(S(t_m-s)-S(t-s))x\big\rangle\,\nu(\dl x)\,\dl s\Big|\\
&\leq C\,\Big(\int_{B_U}\|x\|_U^{\alpha_+}\nu(\dl x)\Big)^{\frac1{\alpha_+'}+\frac1{\alpha_+}}
\int_{[0,T]}\int_0^{t_m}
\Big(\big\|S(t-s)A^{\frac1\alpha-\frac\beta2}\big\|_{\LB(H)}^{\alpha_+-1}\\
&\quad + \|S_{h,k}^{m-\lfloor s\rfloor_k}A^{\frac1\alpha-\frac\beta2}\big\|_{\LB(H)}^{\alpha_+-1}\Big)\cdot\big\|(S(t_m-s)-S(t-s))A^{\frac1\alpha-\frac\beta2}\big\|_{\LB(H)}
\,\dl s\,\meas(\dl t)\\
\end{aligned}
\end{equation*}
Set $\rho:=\frac2\alpha-\beta\in[0,1)$ and let $\gamma\in (0,2]$. As $\|(S(t_m-s)-S(t-s))A^{\frac\rho2}\|_{\LB(H)}\leq\|(\id_H-S(t-t_m))A^{-\frac\gamma2}\|_{\LB(H)}\|S(t_m-s)A^{\frac{\rho+\gamma}2}\|_{\LB(H)}$ and $t-t_m\leq k$, the estimates \eqref{eq:estA1}, \eqref{eq:estA2} and \eqref{eq:estAh1} yield
\begin{align*}
I\!I_b\leq C\,\meas([0,T])\int_0^T(T-s)^{-\frac\rho2(\alpha_+-1)-\frac{\rho+\gamma}2}\,\dl s \; k^{\frac\gamma2}.
\end{align*}
The integral is finite if, and only if, $\gamma<2-(\frac2\alpha-\beta)\alpha_+$. Hence we can use $\gamma:=\alpha\beta_-$ if $\alpha_+\in(\alpha,2)$ is chosen small enough. The term $I\!I_c$ on the right hand side of \eqref{eq:proof_weak_error_alpha_5} can be treated analogously to $I\!I_b$, with integrability exponent $\alpha_-$ instead of $\alpha_+$. Finally, the remaining terms $I\!I_d$ and $I\!I_e$ on the right hand side of \eqref{eq:proof_weak_error_alpha_5} can be estimated in a similar way; 
\felix{see notes 15.10.'16}
we omit the details as no new arguments are involved.
\end{proof}

\addtocontents{toc}{\protect\setcounter{tocdepth}{0}}
\smallskip
\subsection*{Acknowledgement}
We thank Kristin Kirchner, Raphael Kruse, Annika Lang and Stig Larsson for the participation in early discussions regarding this work and~\cite{AnderssonLindner2017b}.
\addtocontents{toc}{\protect\setcounter{tocdepth}{2}}

\begin{appendix}
\section{}

Here we add the postponed proofs of the regularity and strong convergence results in Proposition~\ref{prop:existence_alpha} and \ref{prop:strong_error_alpha_stable}.

\begin{proof}[Proof of Proposition~\ref{prop:existence_alpha}]
Let $\alpha_+\in(\alpha,2]$. Due to the closedness of the operator $A^{\frac{\beta_-}2}$, the continuity of the stochastic integral mapping $I^{\tilde N}_T$ from $L^{\alpha_+}_{\pred}(\Omega_T\times U;H)$ to $L^{\alpha_+}(\Omega;H)$ (cf.~Remark~\ref{rem:stoch_int_space-time}), and the fact that $\|x\|_U=\|A^{\frac\beta2-\frac1\alpha}x\|$ for  $x\in U=\dot H^{\beta-\frac2\alpha}$, we have
\begin{align*}
\Big\|\int_0^t\int_{B_U}&S(t-s)x\,\tilde N(\dl s,\dl x)\Big\|_{L^{\alpha_-}(\Omega;\dot H^{\beta_-})}\\
&\leq \Big\|\int_0^t\int_{B_U}S(t-s)x\,\tilde N(\dl s,\dl x)\Big\|_{L^{\alpha_+}(\Omega;\dot H^{\beta_-})}\\
&\leq C\Big(\int_0^t\int_{B_U}\big\|A^{\frac{\beta_-}2}S(t-s)x\big\|^{\alpha_+}\nu(\dl x)\,\dl s\Big)^{\frac1{\alpha_+}}\\
&\leq C\Big(\int_{B_U}\|x\|_U^{\alpha_+}\nu(\dl x)\int_0^t\big\|S(t-s)A^{\frac{\beta_-}2-\frac{\beta}2+\frac1\alpha}\big\|_{\LB(H)}^{\alpha_+}\dl s\Big)^{\frac 1{\alpha_+}}\\
&\leq C\cdot C_{\beta_--\beta+\frac2\alpha}\Big(\int_0^t(t-s)^{-\alpha_+\cdot(\frac{\beta_--\beta}2+\frac1\alpha)}\dl s\Big)^{\frac 1{\alpha_+}}.
\end{align*}
In the last step we have used the integrability assumption \eqref{eq:ass_nu_alphastable} as well as the estimate \eqref{eq:estA1}, assuming without loss of generality that $\beta_-\in[0,\beta)$ is big enough for $\beta_--\beta+\frac2\alpha$ to be nonnegative. The integral in the last line is finite if we choose $\alpha_+\in(\alpha,2]$ small enough so that $\alpha_+\cdot(\frac{\beta_--\beta}2+\frac1\alpha)$ is less than one. The finiteness of $\|\int_0^t\int_{B_U^c}S(t-s)x\,\tilde N(\dl s,\dl x)\|_{L^{\alpha_-}(\Omega;\dot H^{\beta_-})}$ can be checked analogously. The continuity assertions follows with similar arguments if one uses the estimate~\eqref{eq:estA2}. \felix{see notes 26.6.'16}
\end{proof}

\begin{proof}[Proof of Proposition~\ref{prop:strong_error_alpha_stable}]
We show that 
\begin{equation*}
\big\|X^m_{h,k}-X(t_m)\big\|_{L^{\alpha_-}(\Omega;H)}\leq C\big(h^{\beta_-}+k^{\frac{\beta_-}2}\big),\quad h,k\in(0,1),
\end{equation*}
for all $m\in\{0,\ldots,M\}$ with a finite constant $C=C(X_0,T,\nu,\alpha,\alpha_-,\beta,\beta_-)$ that does not depend on $h,k$ or $m$. Together with the $\beta_-/2$-Hölder continuity of the mapping $[0,T]\ni t\mapsto L(t)\in L^{\alpha_-}(\Omega;H)$ stated in Proposition~\ref{prop:existence_alpha} this implies the assertion.

Fix $\alpha_+\in (\alpha,2]$. The continuity of the stochastic integral mapping $I^{\tilde N}_T$ from $L^p_{\pred}(\Omega_T\times U;H)$ to $L^p(\Omega;H)$ for $p=\alpha_-$ and $p=\alpha_+$ (cf.~Remark~\ref{rem:stoch_int_space-time}) implies
\begin{equation}\label{eq:proof_strong_alpha_1}
\begin{aligned}
&\big\|X^m_{h,k}-X(t_m)\big\|_{L^{\alpha_-}(\Omega;H)}\\
&\leq \|E_{h,k}^m X_0\|+C\,\Big(\int_0^{t_m}\int_{B_U}\big\|\tilde E_{h,k}(t_m-s)x\big\|^{\alpha_+}\nu(\dl x)\,\dl s\Big)^{\frac1{\alpha_+}}\\
&\quad +C\,\Big(\int_0^{t_m}\int_{B_U^c}\big\|\tilde E_{h,k}(t_m-s)x\big\|^{\alpha_-}\nu(\dl x)\,\dl s\Big)^{\frac1{\alpha_-}}.
\end{aligned}
\end{equation}
We estimate the three terms on the right hand side separately and independently of $m$. For the first term we apply \eqref{eq:estAh2} with $\sigma=\beta_-\in [0,\beta)\subset [0,2)$ and $\rho=-\beta_-$ to obtain
\begin{align*}
\|E_{h,k}^m X_0\|
\leq \big\|E_{h,k}^m A^{-\frac{\beta_-}2}\big\|_{\LB(H)}\|X_0\|_{\dot H^{\beta_-}}
\leq C\big(h^{\beta_-}+k^{\frac{\beta_-}2}\big)\|X_0\|_{\dot H^{\beta_-}}.
\end{align*}
For the second term on the right hand side of \eqref{eq:proof_strong_alpha_1} we use the integrability assumption \eqref{eq:ass_nu_alphastable}, the fact that $\|y\|_U=\|A^{\frac\beta2-\frac1\alpha}y\|$ for  $y\in U=\dot H^{\beta-\frac2\alpha}$, and the error estimate \eqref{eq:estAh3} with $\sigma=\beta_-\in [0,\beta)\subset [0,2)$ and $\rho=\frac 2\alpha-\beta\in[0,\min(1,2-\beta_-))$:
\begin{align*}
&\int_0^{t_m}\int_{B_U}\big\|\tilde E_{h,k}(t_m-s)x\big\|^{\alpha_+}\nu(\dl x)\,\dl s\\
&\leq\int_0^{t_m}\big\|\tilde E_{h,k}(t_m-s)A^{\frac1\alpha-\frac\beta2}\big\|_{\LB(H)}^{\alpha_+}\,\dl s\cdot\int_{B_U}\|x\|_U^{\alpha_+}\nu(\dl x)\\
&\leq C\int_0^{t_m}\Big(\big(h^{\beta_-}+k^{\frac{\beta_-}2}\big)\cdot(t_m-s)^{-\frac{\frac 2\alpha-\beta+\beta_-}2}\Big)^{\alpha_+}\dl s\\
&=C\int_0^{t_m}s^{\alpha_+\cdot\frac{\beta-\beta_-}2-\frac{\alpha_+}\alpha}\dl s\;\big(h^{\beta_-}+k^{\frac{\beta_-}2}\big)^{\alpha_+}.
\end{align*}
We observe that the integral in the last line in finite if, and only if, $\frac{\beta_-}2<\frac\beta2-\frac1\alpha+\frac1{\alpha_+}$. The latter condition is fulfilled if we choose $\alpha_+\in (\alpha,2]$ small enough. The third term on the right hand side of \eqref{eq:proof_strong_alpha_1} is estimated in an analogous way. \felix{cf.~notes 17.4.'16}
\end{proof}

\color{black}

\end{appendix}

\bibliographystyle{plain}
\bibliography{../litLevy}

\end{document}